\documentclass{amsart}
\usepackage{amsfonts,amsmath,amsthm,amssymb,booktabs,color,comment,float,fullpage,hyperref,graphicx,subcaption,wrapfig,enumitem,setspace,lineno}
\usepackage[foot]{amsaddr}
\usepackage[mathscr]{eucal}
\usepackage{mathrsfs}
\usepackage{blkarray,bigstrut}
\usepackage{xparse}

\usepackage{tikz}
\usetikzlibrary{automata,arrows,positioning,calc}

\usepackage[mathscr]{eucal}
\usepackage{xfrac}

\usepackage{etoolbox}

\newtoggle{ARXIV}
\toggletrue{ARXIV}
\iftoggle{ARXIV}{}{
   \usepackage[backend=biber,style=alphabetic,sorting=ynt]{biblatex}
    \addbibresource{References/Kolmogorov.bib} 
    \addbibresource{References/Sowers.bib} 
}

\DeclareGraphicsExtensions{.png,.jpeg,.jpg,.pdf}
\newcommand{\imagewidth}{0.75\textwidth}
\graphicspath{{./}{./Graphics/}}
\allowdisplaybreaks

\theoremstyle{plain}
\newtheorem{theorem}{Theorem}[section]

\newtheorem{proposition}[theorem]{Proposition}
\newtheorem{corollary}[theorem]{Corollary}
\newtheorem{lemma}[theorem]{Lemma}
\newtheorem{assumption}[theorem]{Assumption}

\newtheorem{remark}[theorem]{Remark}

\newcommand{\lb}{\left\{}
\newcommand{\rb}{\right\}}
\newcommand{\Def}{\overset{\text{def}}{=}}
\newcommand{\R}{\mathbb{R}}

\newcommand{\Z}{\mathbb{Z}}
\newcommand{\N}{\mathbb{N}}
\newcommand{\eps}{\varepsilon}
\newcommand{\bOne}{\mathbf{1}}
\newcommand{\filt}{\mathscr{F}}
\newcommand{\BP}{\mathbb{P}}
\newcommand{\BE}{\mathbb{E}}

\newcommand{\bdy}{\partial}
\newcommand{\rvW}{\mathsf{W}}
\newcommand{\kk}{\textsf{c}}
\NewDocumentCommand{\KK}{o}{%
  \IfNoValueTF{#1}
    {\textsf{C}} 
    {\textsf{C}_{\eqref{#1}}}
}

\newcommand{\genL}{\mathscr{L}}
\newcommand{\genP}{\mathscr{P}}
\newcommand{\tgenP}{\breve \genP}
\newcommand{\Res}{\mathcal{R}}
\newcommand{\Poly}{\mathcal{P}}
\newcommand{\Rplus}{\R_+}
\newcommand{\Rplusint}{\R_+^\circ}
\newcommand{\ub}{\underline{b}_{\eqref{E:basics}}}
\newcommand{\hypo}{\mathfrak{B}}
\newcommand{\Cspace}{C_b(\domint)\times C_b(\bdy_s)}
\newcommand{\TT}{\textsf{T}}

\newcommand{\cE}{\mathcal{E}}
\newcommand{\Index}{\mathcal{I}}
\newcommand{\degree}{\textsf{d}}
\newcommand{\SP}{x_\circ,y_\circ}
\newcommand{\Mono}[3]{\mathfrak{m}_{\SP}^{#1,#2,#3}}
\newcommand{\Monoid}{\mathfrak{M}}
\newcommand{\Algebra}{\mathfrak{A}_{\SP}}
\newcommand{\Vspace}{\mathfrak{V}_{\SP}}
\newcommand{\bnorm}{\beta}

\newcommand{\VV}{\mathcal{V}}
\newcommand{\XX}{\mathsf{X}}
\newcommand{\YY}{\mathsf{Y}}
\newcommand{\xx}{\mathsf{x}}
\newcommand{\yy}{\mathsf{y}}
\newcommand{\degspace}{\mathcal{D}}
\newcommand{\openstrip}{\mathcal{U}}
\newcommand{\dom}{\mathcal{D}}
\newcommand{\domint}{\dom^\circ}
\newcommand{\domplus}{\dom^+}
\newcommand{\RR}{\R^2}

\newcommand{\kernel}{\mathcal{K}}
\newcommand{\sidekernel}{\mathcal{K}^\bdy}
\newcommand{\kernelQ}{\mathcal{K}^Q}
\newcommand{\sidekernelQ}{\mathcal{K}^{Q,\bdy}}
\newcommand{\kernelbound}[1]{\overline{\mathcal{K}}^{({#1})}}
\newcommand{\kernelboundintegrable}{\overline{\overline{\mathcal{K}}}}
\newcommand{\sidekernelboundintegrable}{\overline{\overline{\mathcal{K}}}^\bdy}
\newcommand{\QCorrector}{\mathcal{Q}}
\newcommand{\Proj}{\Pi_{\SP}}

\title{Side Boundary potentials for a Kolmogorov-type PDE}

\author{Richard Sowers}
\address{University of Illinois\\
    Urbana, IL 61801}
\email{r-sowers@illinois.edu}
\date{\today}
\subjclass[2020]{35K20,35K65}
\keywords{Boundary value problem, hypoellipticity, parametrix,  boundary potential, parabolic PDE}
\thanks{This material is based upon work supported by the
National Science Foundation under Grant Number CMMI 1727785.  Part of this research was performed while the author was visiting the Institute for Pure and Applied Mathematics (IPAM), which is supported by the National Science Foundation (Grant No. DMS-1925919).  The author like to thank L.C.G. Rogers for pointing out the work of McKean. The author would like to thank H.N.Z. Matin for many useful discussions on applied models associated with this problem.}

\begin{document}
\begin{abstract} We solve a Kolmogorov-type hypoelliptic parabolic partial differential equation with a \lq\lq side" boundary condition (in the direction of the weak H\"ormander condition).  We construct an approximate boundary potential which captures the effect of the boundary condition.  Integrals against this approximate boundary potential have a novel jump discontinuity at the boundary which includes a measure discovered by McKean.  We introduce some polynomial corrections to this approximate boundary potential and then construct a boundary-domain Volterra equation to solve the original partial differential equation.  This Volterra integral equation is iteratively solved, and the bounds contain a periodic behavior resulting from the boundary effects.
\end{abstract}
\maketitle

\section{Introduction}

We are interested in the degenerate parabolic boundary value problem
\begin{equation}\label{E:MainPDE} \begin{aligned}\frac{\partial u}{\partial t}(t,x,y)&= \frac12 \frac{\partial^2 u}{\partial y^2}(t,x,y)+b_1(x,y)\frac{\partial u}{\partial x}(t,x,y)+b_2(x,y)\frac{\partial u}{\partial y}(t,x,y)\\
&\qquad +c(x,y)u(t,x,y)+f(t,x,y) \qquad t>0,\, 0<x<\infty,\, -\infty<y<\infty \\
u(0,x,y)&=u_\circ(x,y)\qquad 0<x<\infty,\, -\infty<y<\infty  \\
u(t,0,y)&=u_\bdy(t,y).\qquad  t>0,\, -\infty<y<\infty \end{aligned}\end{equation}
We assume that the coefficient functions $b_1$, $b_2$ and $c$, the forcing function $f$, and the boundary and initial value functions $u_\bdy$ and $u_\circ$ are regular (more precise assumptions are in Section \ref{S:Hypotheses}).  The degeneracy in \eqref{E:MainPDE} comes from the lack of a $\sfrac{\partial^2}{\partial x^2}$ term; the partial differential equation (PDE) \eqref{E:MainPDE} may thus fail to have a unique smooth solution.  Let's firstly assume that
\begin{equation}\label{E:negative} b_1<0\end{equation}
at the boundary (see Assumption \ref{A:basics}),
which implies that boundary is accessible (See Remark \ref{R:complexboundaryconditions}).  In the interest of regularity, we also assume a \emph{H\"ormander} condition, that
\begin{equation}\label{E:parahor} \frac{\partial b_1}{\partial y}>0\end{equation}
(see Assumption \ref{A:basics}).  An example of $b_1$ satisfying the combined requirements of \eqref{E:negative} and \eqref{E:parahor} (i.e., strictly negative but positive $y$-derivative) is the shifted hyberbolic tangent function 
\begin{equation}\label{E:tanh} (x,y)\mapsto -2+\tanh(y).\end{equation}
An allowable collection of generalizations is given in Section \ref{S:Hypotheses} (viz. \eqref{E:tanh_pert}).  Some of the geometry of our problem is captured in Figure \ref{fig:region}.  We have a homogeneous diffusion in the $y$ variable, and a shear negative drift in the $x$ variable.


\begin{figure}[ht] \includegraphics[width=\imagewidth]{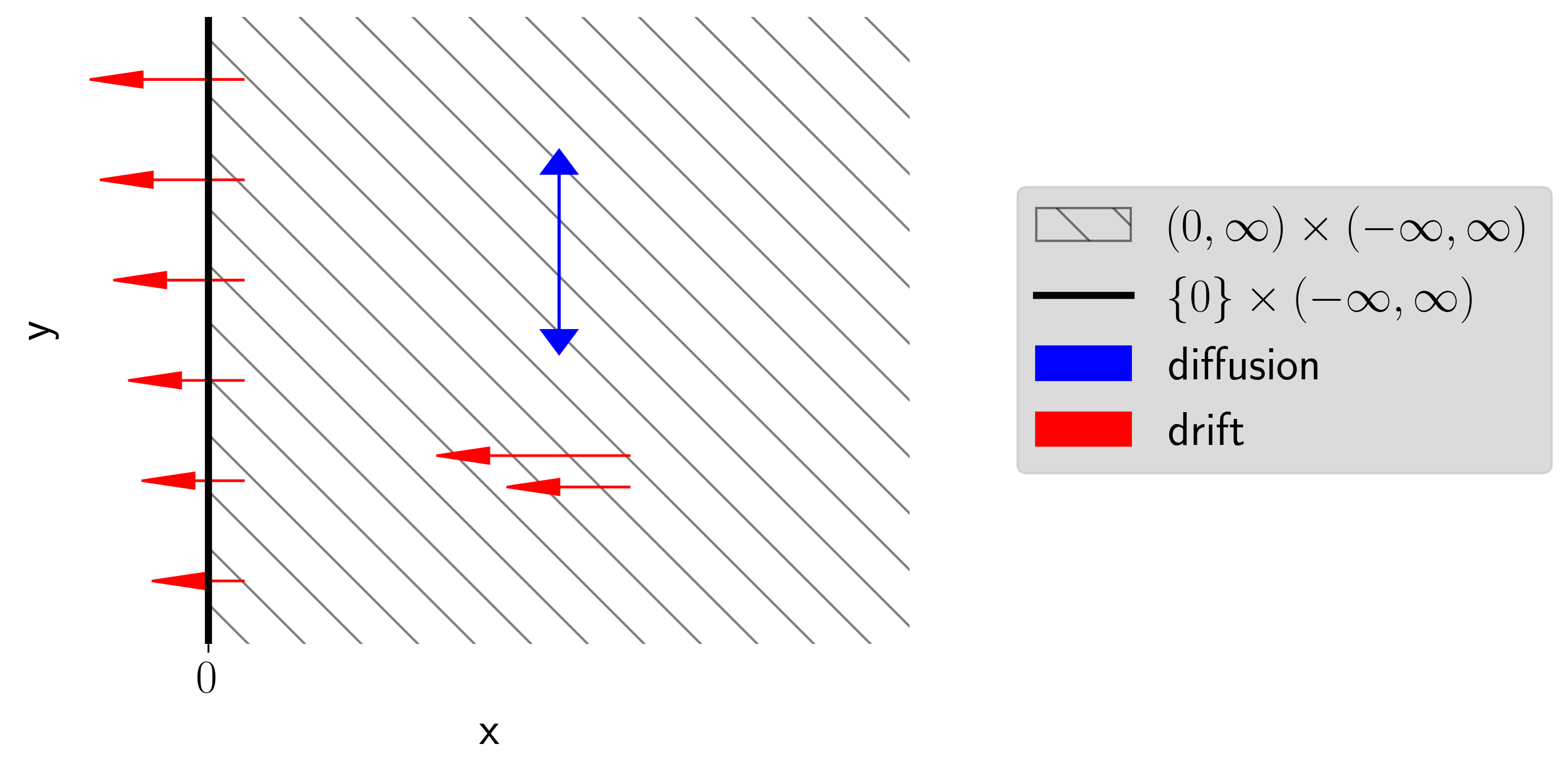}\caption{Geometry of Problem}\label{fig:region}
\end{figure}

The H\"ormander assumption \eqref{E:parahor} ensures \emph{hypoellipticity} \cite{hormander1967hypoelliptic} and thus that the solution of \eqref{E:MainPDE}, if it exists, must be smooth at interior points (assuming that $f$ is smooth).  We in particular want to understand the precise effect of the side Dirichlet boundary data $u_\bdy$.  Although simply posed, this problem seems to have been relatively unexplored (but see \cite{MR1196094,MR499684,MR0393787,MR560317,barros1973regular,garofalo2018estimates}).

\subsection{Solutions of Parabolic Equations}
The classical \cite{friedman2008partial} approach to an analogous nondegenerate boundary-value problem (which would include a $\sfrac{\partial^2}{\partial x^2}$ term) would be based on several approximate solutions.  A \lq\lq parametrix" (see references in \cite{kozhina2018parametrix}) would approximately and locally solve the nondegenerate version of \eqref{E:MainPDE} near a point in a larger spatially boundaryless domain and a \lq\lq boundary potential" (a double-layer potential for Dirichlet conditions) would approximately and locally solve the nondegenerate version of \eqref{E:MainPDE} near a point at the boundary.  One then constructs variations-of-parameters equations to write the nondegenerate version of \eqref{E:MainPDE} as integral equations against the parametrix and boundary potential.  Both the parametrix and boundary potential have Dirac-like singularities, leading to Volterra-type equations which allow one to iteratively solve for the variations of parameters representation. Often, the iteration for the parametrix in the spatially boundaryless domain is first completed (to give an exact \lq\lq fundamental solution"), and the result is then used as the basis for the boundary potential calculations.

 We would like to develop this variation of parameters approach for the degenerate PDE of \eqref{E:MainPDE}.  Lacking a term second derivative term $\sfrac{\partial^2}{\partial x^2}$ transversal to the side boundary, however, the standard double-layer potential does not work as a boundary potential.  We show that an appropriately scaled \emph{single-layer potential} is in fact the correct way to include the boundary data $u_\bdy$ in \eqref{E:MainPDE}.  Additionally, the parametrix has a more complex singularity at the origin than its nondegenerate counterpart (see Section \ref{S:Kernel}) since regularity in $x$ comes from hypoellipticity.  
 Fundamental solutions of parabolic equations often display geometric behavior in various asymptotic regimes corresponding to \lq\lq small diffusion" (or \lq\lq short-time")  \cite{aronson1959fundamental}.   In these regimes, fundamental solutions often look Gaussian, and give quantitative bounds on how singularities propagate.  For nondegenerate parabolic equations, these asymptotics can be written in terms of Riemannian geometry \cite{molchanov1975diffusion}, and for hypoelliptic parabolic equations, in terms of subRiemannian geometry \cite{bailleul2021small,barilari2017heat,brandolini2010non,kaplan1980fundamental}.  These asymptotics are are often closely connected to Varadhan-type \cite{varadhan1967behavior} large deviations estimates of a diffusion moving from one point to another in small time \cite{arous2019rare,bailleul2016large,bailleul2022diffusion,malliavin1996short,pagliarani2023yosida,pigato2018tube}, which in turn connect to control theory \cite{zelikin2004control}.  Our interest here is to give precise pre-exponential structure of the effect of the side boundary condition (cf. \cite{barilari2017curvature,habermann2018small,habermann2019small}).

\subsection{Hypoellipticity}
To understand the condition \eqref{E:parahor}, let's define vector fields
\begin{equation}\label{E:VecFieldDef} V_0(x,y)\Def b_1(x,y)\tfrac{\partial}{\partial x}+b_2(x,y)\tfrac{\partial}{\partial y} \qquad \text{and}\qquad 
V_1(x,y)\Def \tfrac{\partial}{\partial y} \end{equation}
for all $(x,y)$ in the 2-dimensional plane.  We can then write the partial differential operator in \eqref{E:MainPDE} as 
\begin{equation}\label{E:PDEandV} \tfrac12 \tfrac{\partial}{\partial y^2} + b_1\tfrac{\partial}{\partial x}+b_2\tfrac{\partial}{\partial y} =\tfrac12 V_1^2 + V_0. \end{equation}
The Lie bracket between $V_0$ and $V_1$ is
\begin{equation*} [V_0,V_1] = -\tfrac{\partial b_1}{\partial y}\tfrac{\partial}{\partial x}-\tfrac{\partial b_2}{\partial y}\tfrac{\partial}{\partial y}; \end{equation*}
the requirement \eqref{E:parahor} thus implies the classical weak parabolic H\"ormander \cite{hormander1967hypoelliptic} condition that
\begin{equation*} \textrm{Span}\lb V_1(x,y),[V_0,V_1](x,y)\rb \end{equation*}
has full rank (i.e., 2-dimensional); see \cite{bramanti2014invitation, bramanti2022hormander}.

Historically, some of the interest in and properties of \eqref{E:MainPDE} in the interior $(t,x,y)\in (0,\infty)\times (0,\infty)\times (-\infty,\infty)$ stemmed from its connection with stochastic processes and the \lq\lq Kolmogorov" diffusion \cite{anceschi2019survey, kolmogoroff1934zufallige}.  Consider the 2-dimensional stochastic system
 \begin{equation} \label{E:KolmogorovSDE} dX_t=Y_tdt \qquad \text{and}\qquad 
dY_t =  d\rvW_t \end{equation}
for $t>0$, where $\rvW$ is a standard Brownian motion defined on some underlying probability triple $(\Omega,\filt,\BP)$.  This represents the dynamics of a particle (of unit mass) driven by a white noise force\footnote{The work of \cite{nick2022near} motivated this effort.}.
Defining the vector fields $V_{1,K} = \sfrac{\partial}{\partial y}$ and $V_{0,K} = y\sfrac{\partial}{\partial x}$ 
(set $b_1(x,y)=y$ and $b_2\equiv 0$ in \eqref{E:VecFieldDef})
we have
\begin{equation}\label{E:KBracket} [V_{0,K},V_{1,K}]=-\tfrac{\partial}{\partial x}; \end{equation}
and then
\begin{equation*} \frac12 V_{1,K}^2+V_{0,K}=\tfrac12\tfrac{\partial^2}{\partial y^2}+y\tfrac{\partial}{\partial x}. \end{equation*}
This simplifies \eqref{E:VecFieldDef}, \eqref{E:PDEandV}, and \eqref{E:KBracket} to a \lq\lq minimal" case, but one in which the global negativity requirement \eqref{E:negative} doesn't hold (the function $(x,y)\mapsto y$ is not bounded from above away from zero).

For any fixed initial condition $(X_0,Y_0)=(x,y)$, we can explicitly solve \eqref{E:KolmogorovSDE}
as
\begin{equation} \label{E:processes} X^{(x,y)}_t= x+\int_{s=0}^t \{y+\rvW_s\}ds = x+y t + \int_{s=0}^t \rvW_sds \qquad \text{and}\qquad 
Y^{(x,y)}_t = y + \rvW_t\end{equation}
for $t\ge 0$.  The statistics of $(X^{(x,y)}_t,Y^{(x,y)}_t)$ are thus jointly Gaussian.  Using the calculations
\begin{equation} \label{E:covarcalc}\begin{aligned}
    \BE\left[W_t\int_{s=0}^t W_s\right]&= \int_{s=0}^t \BE[\rvW_t\rvW_s]ds = \int_{s=0}^t (t\wedge s) ds = \int_{s=0}^t s ds = \tfrac12 t^2 \\
    \BE\left[\lb \int_{s=0}^t \rvW_sds\rb^2\right]
    &= \int_{s_1=0}^t \int_{s_2=0}^t \BE\left[\rvW_{s_1}\rvW_{s_2}\right]ds_1\, ds_2 = \int_{s_1=0}^t \int_{s_2=0}^t (s_1\wedge s_2)ds_1\, ds_2 \\
    & = 2\int_{s_1=0}^t \lb \int_{s_2=0}^{s_1}s_2 ds_2\rb ds_1 = \int_{s_1=0}^t s_1^2 ds_1 = \tfrac13 t^3, \end{aligned}\end{equation}
we can see that $(X^{(x,y)}_t,Y^{(x,y)}_t)$ has mean vector and covariance matrix
\begin{equation}\label{E:statistics} \begin{pmatrix} x+yt \\ y \end{pmatrix} \qquad \text{and}\qquad \begin{pmatrix} \tfrac13 t^3 & \tfrac12 t^2 \\ \tfrac12 t^2 & t \end{pmatrix} \qquad t>0\end{equation}
Checking that the covariance matrix is invertible (it has determinant $\sfrac{t^4}{12}$), the vector $(X^{(x,y)}_t,Y^{(x,y)}_t)$ has a smooth density (one of the singular successes of \emph{Malliavin calculus} \cite{MR517250,malliavin1978stochastic} is that a smooth density exists when the dynamics of $(X^{(x,y)},Y^{(x,y)})$ are nonlinear but still hypoelliptic; see \cite{bell2012malliavin,bismut1981martingales,kusuoka1987applications,kusuoka2006hormander,kusuoka1984applications,kusuoka1985applications,norris1986simplified,stroock1981malliavin}).

\begin{remark}\label{R:complexboundaryconditions}
At least formally, we can solve the PDE
\begin{equation} \label{E:KPDE} \begin{aligned} \frac{\partial u_K}{\partial t}(t,x,y)&=\frac12 \frac{\partial^2 u_K}{\partial y^2}(t,x,y)+y\frac{\partial u_K}{\partial x}(t,x,y) \qquad (t,x,y)\in (0,\infty)\times (0,\infty)\times (-\infty,\infty) \\
u_K(0,x,y)&= u_\circ(x,y) \qquad (x,y)\in (0,\infty)\times (-\infty,\infty) \\
u_K(t,0,y)&= u_\bdy(t,y) \qquad (t,y)\in (0,\infty)\times (-\infty,\infty)
\end{aligned}\end{equation}
by using \eqref{E:processes}.  For $(x,y)\in (0,\infty)\times (-\infty,\infty)$, define
\begin{equation*} \tau^{(x,y)}\Def \inf\lb t>0: X^{(x,y)}_t<0 \rb; \end{equation*}
we should then have the Feynman-Kac representation
\begin{equation}\label{E:exitstatistics} u(t,x,y)=\BE\left[u_\bdy\left(t- \tau^{(x,y)},Y^{(x,y)}_{\tau^{(x,y)}}\right)\bOne_{\{\tau^{(x,y)}<t\}}\right]+\BE\left[u_\circ\left(X^{(x,y)}_t,Y^{(x,y)}_t\right)\bOne_{\{\tau^{(x,y)}\ge t\}}\right] \end{equation}
for all $(t,x,y)\in (0,\infty)\times (0,\infty)\times (-\infty,\infty)$.

The representation \eqref{E:exitstatistics} for the solution of \eqref{E:KPDE} motivates \eqref{E:negative}.  The system \eqref{E:processes} has diffusion only in the direction parallel to the boundary \textup{(}the y-direction\textup{)}.  The trajectories of \eqref{E:processes} can only hit the boundary via drift, and only when the drift is non-positive \textup{(}see also \cite{kogoj2017dirichlet}\textup{)}.  The requirement \eqref{E:negative} \textup{(}and more precisely Assumption \ref{A:basics}\textup{)} ensures that this drift is uniformly negative in our problem.  This is then used in the construction of the bounding function 
 \eqref{E:sidekernelboundintegrableDef} and then in Proposition \ref{P:sideintegrablekernel}.

Some classic results by McKean \cite{mckean1962winding} give information about the joint statistics of $\tau^{(0,1)}$ and $Y^{(x,y)}_{\tau^{(0,1)}}$.  See Remark \ref{R:McKean} for further information.
\end{remark}

Our main result is Theorem \ref{T:Main}, which gives a solution to \eqref{E:MainPDE} in terms of a variation of parameters formula \eqref{E:VariationOfConstants}.  Our main contribution is the kernel $\sidekernelQ$ (of \eqref{E:siderkernelQDef}) in \eqref{E:VariationOfConstants}, which captures propagation of the boundary data $u_\bdy$ into the interior of the domain of \eqref{E:MainPDE}. 

Although hypoelliptic parabolic equations have been fairly thoroughly studied, our efforts seem to fill a lacuna in the literature.  An abstract theory \cite{AIF_1969__19_1_277_0, kogoj2017dirichlet} ensures the existence and uniqueness of solutions to \eqref{E:MainPDE}.  Our interest is a framework allowing precise asymptotics.  The nonlinearities in \eqref{E:MainPDE} are in the drift terms; \cite{manfredini1997dirichlet} considers a similar problem, but with linear drift (an extension of \eqref{E:KPDE}) but a spatially varying diffusion term; the results of \cite{manfredini1997dirichlet} are existence and uniqueness via functional-analytic tools.  Fundamental solutions of strongly hypoelliptic parabolic PDE's have been well-studied \cite{bailleul2021small,barilari2017heat,brandolini2010non}, but our problem is weakly hypoelliptic (even in these strongly hypoelliptic cases, the precise effect of boundary conditions seems not to be well-understood).

\section{Assumptions and Notation }\label{S:Hypotheses}

\subsection{Spaces and standard notation}

As usual, $\R\Def (-\infty,\infty)$, and $\Rplus\Def [0,\infty)$.  Then $\Rplus$ has interior $\Rplusint \Def (0,\infty)$.

Let's fix a $\TT>0$ and define the parabolic domain 
\begin{equation*} \dom \Def [0,\TT)\times \Rplus\times \R \end{equation*}
of our problem.
We endow $\dom$ with its standard topology (inherited from $\R^3$).  The first equality of \eqref{E:MainPDE} holds on the interior $\domint = (0,\TT)\times \Rplusint \times \R$ of $\dom$.
The relative boundary of $\dom$ is then the disjoint union
\begin{equation*} \bdy \dom = \dom\setminus \domint = \left((0,\TT)\times \{0\}\times \R\right) \cup \left(\{0\}\times \Rplusint\times \R\right) \cup \left(\{0\}\times \{0\}\times \R\right),\end{equation*}
and
\begin{align*}
(0,\TT)\times \{0\}\times \R&\sim \bdy_s\Def (0,\TT)\times \R  \\
\{0\}\times \Rplusint\times \R&\sim \bdy_i\Def \Rplusint\times \R \\
\{0\}\times \{0\}\times \R&\sim \R;\end{align*}
with $\bdy_s$ being isomorphic to the \lq\lq side" boundary and $\bdy_i$ being isomorphic to the \lq\lq initial" boundary.
Let's also define
\begin{equation*} \domplus\Def \{(t,x,y)\in \dom: t>0\}= (0,\TT)\times \Rplus\times \R \end{equation*}
as the points in $\dom$ with positive time; heat kernels typically have a singularity at $t=0$.

Several spatial arguments will be simpler to understand when written in terms of the boundaryless plane $\RR\Def \R\times \R$ (as opposed to $\partial_i$).  Define also the open strip
\begin{equation*} \openstrip\Def (0,\TT)\times \R\times \R; \end{equation*}
then $\domplus\subset \openstrip$.

As usual, if $S$ is a topological space, $C_b(S)$ is the collection of bounded and continuous functions on $S$, and $\|\cdot \|$ will be the associated supremum norm on $C_b(S)$.

Finally, we will frequently use
indices in the set 
\begin{equation}\label{E:NDef} \N\Def \{0,1,2\dots\}. \end{equation}

\subsection{Regularity}

\begin{assumption}
We assume that $u_\bdy\in C_b(\bdy_s)$ and $u_\circ\in C_b(\bdy_i)$, and $f\in C_b(\domint)$.\end{assumption}

\begin{assumption}\label{A:boundedness} We assume that $b_1$, $b_2$, and $c$ are in $C_b^\infty(\RR)$ \textup{(}i.e., infinitely differentiable with all derivatives being bounded\textup{)}.  We assume that there is a $\KK[E:boundedness]>0$ such that
\begin{equation} \label{E:boundedness} \sup_{\substack{(x,y)\in \RR \\ (i,j)\in \N^2 \\ i+j\le 3}}\left|\frac{\partial^{i+j} b_1}{\partial x^i \partial y^j}(x,y)\right|\le \KK[E:boundedness], \quad \sup_{\substack{(x,y)\in \RR \\ (i,j)\in \N^2 \\ i+j\le 1}}\left|\frac{\partial^{i+j} b_2}{\partial x^i \partial y^j}(x,y)\right|\le \KK[E:boundedness], \quad \text{and}\quad \sup_{(x,y)\in \RR}\left|c(x,y)\right|\le \KK[E:boundedness]. \end{equation}
\end{assumption}
\noindent We will need one derivative of $b_1$ in \eqref{E:hypodef}, \eqref{E:heart}, and the proof of Proposition \ref{P:dirac}, two derivatives of $b_1$ in \eqref{E:VH} and \eqref{E:tgenPLstuff}, and three derivatives of $b_1$ in \eqref{E:XiPartsDef}.  We will also need one derivative of $b_2$ in \eqref{E:XiPartsDef}.

Let's next formalize \eqref{E:negative} and \eqref{E:parahor}.  Define
\begin{equation}\label{E:hypodef} \hypo(x,y)\Def \tfrac{\partial b_1}{\partial y}(x,y). \qquad (x,y)\in \RR \end{equation}
From Assumption \ref{A:boundedness}, we have that
\begin{equation}\label{E:hypoisbounded} |\hypo(x,y)|\le \KK[E:boundedness] \end{equation}
for all $(x,y)\in \RR$.
\begin{assumption}\label{A:basics} We require that 
\begin{equation} \label{E:basics} \ub \Def \inf_{y\in \R}\{-b_1(0,y)\} \end{equation}
be positive \textup{(}i.e., $b_1(0,y)\le -\ub<0$ if $y\in \R$\textup{)}.  We also require that
$\hypo(x,y)>0$ for all $(x,y)\in \RR$. \end{assumption}
\noindent For each $x\in \R$, we have (keeping in mind the assumption that $\hypo>0$)
\begin{equation} \label{E:hypotozero} 0 = \lim_{|y|\nearrow \infty}\frac{\KK[E:boundedness]}{|y|}
\ge \varliminf_{|y|\nearrow \infty}\frac{|b_1(x,y)|}{|y|}
=\varliminf_{|y|\nearrow \infty}\frac{\left|b_1(x,0)+\int_{y'=0}^y \hypo(x,y')dy'\right|}{|y|}
\ge \varliminf_{|y|\nearrow \infty}|\hypo(x,y)|; \end{equation}
our calculations must allow for globally small values of $\hypo$.
The function \eqref{E:tanh} satisfies both Assumptions \ref{A:boundedness} and \ref{A:basics} and has global behavior which is sufficiently regular for our calculations. Another example of $b_1$ which we might take is:
\begin{equation}\label{E:tanh_pert} (x,y)\mapsto  -2+\tanh(y)+\delta \Psi(x,y) \end{equation}
where $\Psi$ is infinitely smooth and with compact support and where $\delta$ is sufficiently small.

\begin{remark}
We could just as well assume that $\hypo$ be negative; a reflection $y\mapsto -y$ would then bring us back to Assumption \ref{A:basics}.  We have chosen to take $\hypo$ to be positive to more easily benefit from comparisons to \eqref{E:KPDE}.
\end{remark}

The quantity $\hypo$ will act as a prefactor (see \eqref{E:approxSDE}) of the interaction like $Y_t dt$ of \eqref{E:KolmogorovSDE}.  It thus is part of a covariance matrix (see \eqref{E:thesestatistics}) which generalizes the covariance matrix of \eqref{E:statistics}.  In light of \eqref{E:hypotozero}, this covariance matrix must globally have small points.  We shall assume that $\hypo$ controls a number of relevant derivatives of $b_1$, implying that the regularity of $b_1$ is locally of the same order as $\hypo$, particularly where $\hypo$ is small.
\begin{assumption}\label{A:hypobound}  We assume that there is a $\KK[E:hypobound]>0$ such that
\begin{equation}\label{E:hypobound} \left|\frac{\partial^{i+j} b_1}{\partial x^i \partial y^j}(x,y)\right|\le \KK[E:hypobound]|\hypo(x,y)| \end{equation}
for all $(x,y)\in \RR$ and all $(i,j)\in \N^2$ with $i+j\le 3$.
\end{assumption}
\noindent Since $\hypo$ is required to be positive (Assumption \ref{A:basics}), the functions
\begin{equation*} \bnorm_{i,j}(x,y)\Def \frac{\frac{\partial^{i+j} b_1}{\partial x^i\partial y^j}(x,y)}{\hypo(x,y)} \qquad (x,y)\in \RR \end{equation*}
are well-defined for $(x,y)\in \RR$ and $(i,j)\in \N^2$ such that $i+j\le 3$ and furthermore
\begin{equation}\label{E:betabounds} \sup_{(x,y)\in \RR}\left|\bnorm_{i,j}(x,y)\right|\le \KK[E:hypobound]. \end{equation}
for $(x,y)\in \RR$ and $(i,j)\in \N^2$ such that $i+j\le 3$.
We note that by definition \eqref{E:hypodef}, $\bnorm_{0,1}(x,y)=1$
for all $(x,y)\in \RR$.

Both \eqref{E:tanh} and \eqref{E:tanh_pert} decay to an exponential in $y$ for $y$ large, and their derivatives are again exponential. 
 Both \eqref{E:tanh} and \eqref{E:tanh_pert} thus satisfy Assumption \ref{A:hypobound}.  However, a function $b_1$ of the type
\begin{equation}\label{E:GaussianB} (x,y)\mapsto -2+\int_{s=-\infty}^y \tfrac{1}{\sqrt{2\pi}}\exp\left[-\tfrac12 s^2\right]ds \end{equation}
(i.e., a Gaussian cumulative) would \emph{not} satisfy Assumption \ref{A:hypobound}. The $y$-derivative of \eqref{E:GaussianB} is a Gaussian density; $\sfrac{\partial \ln \hypo}{\partial y}$ would in that case grow linearly (and not be bounded).

\subsection{Bounds}
Assumption \ref{A:hypobound} implies some quantitative bounds.  

Let's firstly show that Assumption \ref{A:hypobound} controls how small $\hypo$ can become in a neighborhood.  For $(x,y)\in \RR$,
\begin{align*} \frac{\partial \ln \hypo}{\partial x}(x,y) &= \frac{\frac{\partial \hypo}{\partial x}(x,y)}{\hypo(x,y)} = \frac{\frac{\partial^2 b_1}{\partial x \partial y}(x,y)}{\hypo(x,y)} = \bnorm_{1,1}(x,y)\\
\frac{\partial \ln \hypo}{\partial y}(x,y) &= \frac{\frac{\partial \hypo}{\partial y}(x,y)}{\hypo(x,y)} = \frac{\frac{\partial^2 b_1}{\partial y^2}(x,y)}{\hypo(x,y)} = \bnorm_{0,2}(x,y). \end{align*}
For any \emph{base} and \emph{evaluation} points $(x_b.y_b)$ and $(x_e,y_e)$ in $\RR$,
\begin{equation} \label{E:VH} \begin{aligned} \left|\ln \hypo(x_e,y_e)-\ln \hypo(x_b,y_b)\right|
&= \left|\int_{s=0}^1\lb (x_e-x_b)\frac{\partial \ln \hypo}{\partial x}(x_b+s(x_e-x_b),y_b+s(y_e-y_b)) \right.\right.\\
&\qquad \qquad \left.\left. + (y_e-y_b)\frac{\partial \ln \hypo}{\partial y}(x_b+s(x_e-x_b),y_b+s(y_e-y_b))\rb ds\right|\\
&= \left|\int_{s=0}^1\lb (x_e-x_b)\bnorm_{1,1}(x_b+s(x_e-x_b),y_b+s(y_e-y_b)) \right.\right.\\
&\qquad \qquad \left.\left. + (y_e-y_b)\bnorm_{0,2}(x+s(x_e-x_b),y_b+s(y_e-y_b))\rb ds\right|\\
&\le \KK[E:hypobound]\lb |x_e-x_b|+|y_e-y_b|\rb;
\end{aligned}\end{equation}
thus
\begin{equation}\label{E:hypoboundsandwich} \exp\left[-\KK[E:hypobound]\lb |x_e-x_b|+|y_e-y_b|\rb\right]\le \frac{|\hypo(x_e,y_e)|}{|\hypo(x_b,y_b)|}\le \exp\left[\KK[E:hypobound]\lb |x_e-x_b|+|y_e-y_b|\rb\right].
\end{equation}

Let's next construct some bounds on Taylor remainders, in particular using $\hypo$ to bound the remainders of $b_1$.  Fix $g\in C^\infty(\RR)$, a positive integer $d$, and base and evaluation points $(x_b,y_b)$ and $(x_e,y_e)$ in $\RR$.  Define 
\begin{equation} \label{E:TaylorRemainder}\begin{aligned} \left(\Res^{(d)}_{(x_b.y_b)}g\right)(x_e,y_e)&\Def \int_{s=0}^1\frac{(1-s)^{d-1}}{(d-1)!}\lb \sum_{d'=0}^d (x_e-x_b)^{d'}(y_e-y_b)^{d-d'}\binom{d}{d'}\right.\\
&\qquad \left. \times \frac{\partial^d g}{\partial x^{d'} \partial y^{d-d'}}\left(x_b+s(x_e-x_b),y_b+s(y_e-y_b)\right)\rb ds \end{aligned}\end{equation}
so that the Taylor expansion of $g$ with degree $d-1$ is
\begin{equation*} g(x,y) =\sum_{\substack{(i,j)\in \N^2 \\ i+j\le d-1}}\frac{1}{i!j!}\frac{\partial^{i+j}g}{\partial x^i \partial y^j}(x_b,y_b)(x_e-x_b)^i(y_e-y_b)^j + \left(\Res^{(d)}_{(x_b,y_b)}g\right)(x_e,y_e). \end{equation*}
A natural bound on $\Res^{(d)}$ is
\begin{align*} &\left|\left(\Res^{(d)}_{(x_b,y_b)}g\right)(x_e,y_e)\right| \\
&\qquad \le \frac{1}{d!}\lb \sup_{\substack{0\le s\le 1 \\ 0\le d'\le d}}\left|\frac{\partial^d g}{\partial x^{d'}\partial y^{d-d'}}\left(x_b+s(x_e-x_b),y_b+s(y_e-y_b)\right)\right|\rb\sum_{d'=0}^d \binom{d}{d'}|x_e-x_b|^{d'}|y_e-y_b|^{d-d'}\\
&\qquad =\frac{1}{d!}\lb \sup_{\substack{0\le s\le 1 \\ 0\le d'\le d}}\left|\frac{\partial^d g}{\partial x^{d'}\partial y^{d-d'}}\left(x_b+s(x_e-x_b),y_b+s(y_e-y_b)\right)\right|\rb\lb |x_e-x_b|+|y_e-y_b|\rb^d. \end{align*}
From Assumption \ref{A:boundedness}, we have that
\begin{equation} \label{E:Tbounds} \left|\left(\Res^{(1)}_{(x_b,y_b)}b_2\right)(x_e,y_e)\right|\le \KK[E:boundedness]\lb |x_e-x_b|+|y_e-y_b|\rb.  \end{equation}
Using \eqref{E:betabounds} and then \eqref{E:hypoboundsandwich}, we also have that
\begin{align*} &\left|\left(\Res^{(d)}_{(x_b,y_b)}b_1\right)(x_e,y_e)\right|\\
&\qquad \le  \tfrac{1}{d!}\KK[E:hypobound]\lb \sup_{\substack{ 0\le s\le 1 \\ 0\le d'\le d}}\left|\frac{\partial^d b_1}{\partial x^{d'}\partial y^{d-d'}}\left(x_b+s(x_e-x_b),y_b+s(y_e-y_b)\right)\right|\rb \lb |x_e-x_b|+|y_e-y_b|\rb^d \\
&\qquad \le  \tfrac{1}{d!}\KK[E:hypobound]\lb \sup_{0\le s\le 1}\hypo\left(x_b+s(x_e-x_b),y_b+s(y_e-y_b)\right)\rb\lb |x_e-x_b|+|y_e-y_b|\rb^d \\
&\qquad\le  \tfrac{1}{d!}\KK[E:hypobound]\hypo(x_b,y_b)\lb \sup_{0\le s\le 1}\exp\left[\KK[E:hypobound]\lb s|x_e-x_b|+s|y_e-y_b|\rb\right]\rb \lb |x_e-x_b|+|y_e-y_b|\rb^d \\
&\qquad=  \tfrac{1}{d!}\KK[E:hypobound]\hypo(x_b,y_b)\exp\left[\KK[E:hypobound]\lb |x_e-x_b|+|y_e-y_b|\rb\right] \lb |x_e-x_b|+|y_e-y_b|\rb^d
\end{align*}
for $d\in \{1,2,3\}$.
In other words,
\begin{equation}\label{E:Tboundsb1}\frac{\left|\left(\Res^{(d)}_{(x_b,y_b)}b_1\right)(x_e,y_e)\right|}{\hypo(x_b,y_b)}
\le \tfrac{\KK[E:hypobound]}{d!}\exp\left[\KK[E:hypobound]\lb |x_e-x_b|+|y_e-y_b|\rb\right] \lb |x_e-x_b|+|y_e-y_b|\rb^d \end{equation}
for $d\in \{1,2,3\}$.
While \eqref{E:Tbounds} is standard, the more complicated bounds of \eqref{E:Tboundsb1} will be needed to control regularity of $b_1$ in terms of $\hypo$. 

\subsection{Partial Differential Operator}
If $g$ is a real-valued function which is defined and twice-differentiable on some open subset $\openstrip'$ of $\openstrip$ (typically either $\openstrip$ or $\domint$), we define
\begin{equation}\label{E:genPdef} \begin{aligned} (\genP g)(t,x,y)&\Def \frac12 \frac{\partial^2 g}{\partial y^2}(t,x,y)+b_1(x,y)\frac{\partial g}{\partial x}(t,x,y)+b_2(x,y)\frac{\partial g}{\partial y}(t,x,y)\\
&\qquad +c(x,y)g(t,x,y)-\frac{\partial g}{\partial t}(t,x,y). \end{aligned} \qquad (t,x,y)\in \openstrip'
\end{equation}
If $u\in C(\dom)$ is twice-differentiable in $\domint$ and satisfies 
\begin{align*} \left(\genP u\right)(t,x,y)&=-f(t,x,y)\qquad (t,x,y)\in \domint \\
u(0,x,y)&= u_\circ(x,y) \qquad (x,y)\in \bdy_i\\
u(t,0,y)&= u_\bdy(t,y) \qquad (t,y)\in \bdy_s \end{align*}
then $u$ is a classical solution of \eqref{E:MainPDE} (see also \cite{manfredini1997dirichlet} for some related regularity results).

\section{Main Parts of Parametrix and Approximate Boundary Potential}\label{S:Kernel}
Let's define an approximate heat kernel for \eqref{E:MainPDE}.  We want the heat kernel to capture the smoothing effect of the H\"ormander condition \eqref{E:parahor} and be simple enough to manipulate.  

Thinking of \eqref{E:MainPDE} as a variation on the Kolmogorov equation for a density, let's fix $(\SP)\in \RR$ and consider the stochastic differential equation
\begin{equation}\label{E:mainprocess}\begin{aligned} dX_t &= -b_1(X_t,Y_t)dt \\
dY_t &= -b_2(X_t,Y_t)dt + d\rvW_t \\
(X_0,Y_0)&= (\SP).\end{aligned}\qquad t>0\end{equation}
The Kolmogorov diffusion of \eqref{E:KolmogorovSDE} suggests that we approximate $b_1$ with a linearization; for $(\SP)\in \RR$, define
\begin{equation*}
b^L_{1,\SP}(x,y)\Def b_1(\SP)+\frac{\partial b_1}{\partial y}(\SP)(y-y_\circ)=b_1(\SP)+\hypo(\SP)(y-y_\circ)\end{equation*}
for $(x,y)\in \RR$.
For $t$ small, the solution of \eqref{E:mainprocess} should satisfy 
\begin{equation} \label{E:approxSDE} \begin{aligned} dX_t &\approx -b^L_{1,\SP}(X_t,Y_t)dt =-\lb b_1(\SP)+\frac{\partial b_1}{\partial y}(\SP)(Y_t-y_\circ)\rb dt \\
dY_t &\approx -b_2(\SP)dt + d\rvW_t
\end{aligned}\qquad t>0\end{equation}
with initial condition $(X_0,Y_0)= (\SP)$.
We can explicitly solve this approximation as
\begin{equation} \label{E:approxSDEsolution} \begin{aligned} X_t&\approx x_\circ - b_1(\SP)t - \hypo(\SP) \int_{s=0}^t \lb -b_2(\SP)s + \rvW_s\rb ds \\
&= x_\circ - b_1(\SP)t +\tfrac12 \hypo(\SP)b_2(\SP)t^2 - \hypo(\SP)\int_{s=0}^t \rvW_s ds\\
Y_t &\approx y_\circ - b_2(\SP)t + \rvW_t
\end{aligned}\end{equation}
for $t\ge 0$.

\begin{remark}\label{R:ForwardBackward}
The Feynman-Kac formula \eqref{E:exitstatistics}  depends on the fact that the generator of \eqref{E:processes} is the partial differential operator on the right-hand side of the first line of \eqref{E:KPDE}.  On the other hand, we want to construct a kernel for \eqref{E:MainPDE} as a \emph{density}.  Densities evolve according to the adjoint of the generator.  The generator of \eqref{E:approxSDE} is
\begin{equation} \label{E:forwardL} \begin{aligned} \left(\genL_{\eqref{E:approxSDE},\SP}f\right)(x,y) &= \frac12 \frac{\partial ^2 f}{\partial y^2}(x,y)-\lb b_1(\SP)+\hypo(\SP)(y-y_\circ)\rb \frac{\partial f}{\partial x}(x,y)\\
&\qquad -b_2(\SP)\frac{\partial f}{\partial y}(x,y) \qquad f\in C^2(\R^2),\, (x,y)\in \R^2 \end{aligned}\end{equation}
the adjoint of which is
\begin{align*}\left(\genL^*_{\eqref{E:approxSDE},\SP}f\right)(x,y) &= \frac12 \frac{\partial ^2 f}{\partial y^2}(x,y)+\lb b_1(\SP)+\frac{\partial b_1}{\partial y}(\SP)(y-y_\circ)\rb \frac{\partial f}{\partial x}(x,y) \\
&\qquad +b_2(\SP)\frac{\partial f}{\partial y}(x,y). \qquad f\in C^2(\R^2),\, (x,y)\in \R^2 \end{align*}
This approximates the first-and second-order terms in \eqref{E:MainPDE}.  
\end{remark}

 Comparable to \eqref{E:statistics}, the mean and variance of the right-hand side of \eqref{E:approxSDEsolution} are (using \eqref{E:covarcalc})
\begin{equation}\label{E:thesestatistics} \begin{gathered} \begin{pmatrix} \mu_1(t,\SP)\\ \mu_2(t,\SP)\end{pmatrix} = \begin{pmatrix}  x_\circ - b_1(\SP)t + \frac12 \hypo(\SP)b_2(\SP)t^2 \\
y_\circ - b_2(\SP)t \end{pmatrix} \\
D(t,\SP)\Def \begin{pmatrix} \frac13 \hypo^2(\SP) t^3 &  -\frac12 \hypo(\SP) t^2 \\
-\frac12 \hypo(\SP) t^2 & t
\end{pmatrix}.
\end{gathered}\end{equation}
for $t>0$.  The determinant and inverse of $D(t,\SP)$ are then
\begin{equation*} 
 \det D(t,\SP) = \tfrac{1}{12}t^4\hypo^2(\SP) \qquad \text{and}\qquad
    D^{-1}(t,\SP)=\begin{pmatrix} \frac{12}{\hypo^2(\SP)t^3} & \frac{6}{\hypo(\SP)t^2} \\
    \frac{6}{\hypo(\SP)t^2} & \frac{4}{t}\end{pmatrix}.
\end{equation*}

We can then explicitly compute the density of the approximation (see \cite{anceschi2019survey, weber1951fundamental}) \eqref{E:approxSDEsolution}, which we can use to build a solution of \eqref{E:MainPDE}.  For $(\SP)$ and $(t,x,y)\in \openstrip$, define
\begin{equation} \label{E:cEDef} \begin{aligned} \cE_{\SP}(t,x,y)&\Def \frac12 \begin{pmatrix} x-\mu_1(t,\SP) & y-\mu_2(t,\SP) \end{pmatrix} D^{-1}(t,\SP)\begin{pmatrix} x-\mu_1(t,\SP) \\ y-\mu_2(t,\SP) \end{pmatrix} \\
&= 6\frac{\left(x-\mu_1(t,\SP)\right)^2}{\hypo^2(\SP)t^3} +6\frac{\left(x-\mu_1(t,\SP)\right)\left(y-\mu_2(t,\SP)\right)}{\hypo(\SP)t^2} +2\frac{\left(y-\mu_2(t,\SP)\right)^2}{t}\\
&= 6\XX_{\SP}^2(t,x)+6\XX_{\SP}(t,x)\YY_{\SP}(t,y)+2\YY^2_{\SP}(t,y) \end{aligned}\end{equation}
where
\begin{equation}\label{E:XXYYDef}\begin{aligned}\XX_{\SP}(t,x)&\Def \frac{x-\mu_1(t,\SP)}{\hypo(\SP)t^{3/2}} = \xx_{\SP}(t,x)-\frac12 b_2(\SP)t^{1/2} \\
\YY_{\SP}(t,y)&\Def \frac{y-\mu_2(t,\SP)}{t^{1/2}}= \yy_{\SP}(t,y)+b_2(\SP)t^{1/2}\end{aligned}\end{equation}
where in turn we define
\begin{equation} \label{E:xxyydef}  \xx_{\SP}(t,x)\Def \frac{x-x_\circ+b_1(\SP)t}{\hypo(\SP)t^{3/2}}\qquad \text{and}\qquad 
\yy_{\SP}(t,y)\Def \frac{y-y_\circ}{t^{1/2}}. \end{equation}
The terms $\xx_{\SP}$ and $\yy_{\SP}$ are the main components of $\XX_{\SP}$ and $\YY_{\SP}$ (i.e., without the effect of $b_2$).

For $(\SP)\in \RR$ and $(t,x,y)\in \openstrip$, set \cite{mckean1962winding,melnykova2020parametric,weber1951fundamental}
\begin{equation}\label{E:kernelDef} \kernel_{\SP}(t,x,y)
= \frac{1}{2\pi \sqrt{\det D(t,\SP)}}\exp\left[-\cE_{\SP}(t,x,y)\right]=\frac{\sqrt{12}}{2\pi|\hypo(\SP)|t^2}\exp\left[-\cE_{\SP}(t,x,y)\right]. \end{equation}
This will act as the first approximation of the heat kernel of \eqref{E:MainPDE}.

For future reference,
\begin{equation} \label{E:spatialderivatives}
\begin{aligned}
\frac{\partial \cE_{\SP}}{\partial x}(t,x,y)
    &= \frac{1}{\hypo(\SP)t^{3/2}}\lb 12\XX_{\SP}(t,x)+6\YY_{\SP}(t,y)\rb \\
\frac{\partial \cE_{\SP}}{\partial y}(t,x,y)
    &= \frac{1}{t^{1/2}}\lb 6\XX_{\SP}(t,x)+4\YY_{\SP}(t,y)\rb 
\end{aligned}\end{equation}
for $(\SP)\in \RR$ and $(t,x,y)\in \openstrip$.

Not surprisingly, $\kernel_{\SP}$ solves a linearized PDE corresponding to the generator of \eqref{E:approxSDE}.  For a real-valued function $g$ which is defined and twice-differentiable on a subset $\openstrip'$ of $\openstrip$, define
\begin{equation}\label{E:linearizedgenerator}\begin{aligned}
 (\genP^L_{\SP} g)(t,x,y)&\Def \frac12 \frac{\partial^2 g}{\partial y^2}(t,x,y)+b^L_{1.\SP}(x,y)\frac{\partial g}{\partial x}(t,x,y)+b_2(\SP)\frac{\partial g}{\partial y}(t,x,y)\\
&\qquad -\frac{\partial g}{\partial t}(t,x,y). \end{aligned} \qquad (t,x,y)\in \openstrip'\end{equation}
Recalling Remark \ref{R:ForwardBackward},
\begin{equation}\label{E:combinedgenerators}
 \genP^L_{\SP}= \genL^*_{\eqref{E:approxSDE},\SP}-\frac{\partial}{\partial t}
 \end{equation}

\begin{proposition}\label{P:KEvol}  Fix $(\SP)\in \RR$.  Then $\genP^L_{\SP}\kernel_{\SP}=0$ on $\openstrip$.
\end{proposition}
\begin{proof}
The SDE suggested by \eqref{E:approxSDE} (replace approximation with equality) has generator \eqref{E:forwardL}.
The density of this SDE is $\kernel_{\SP}$; keeping \eqref{E:combinedgenerators}, the claim follows.
\end{proof}
\noindent We will build upon this in Section \ref{S:Corrections}.

For each $y_\circ\in \R$ and $(t,x,y)\in \openstrip$, let's next define
\begin{equation}\label{E:sidekernelDef}\begin{aligned} \sidekernel_{y_\circ}(t,x,y)&\Def |b_1(0,y_\circ)|\kernel_{0,y_\circ}(t,x,y)\\
&=\frac{\sqrt{12}|b_1(0,y_\circ)|}{2\pi|\hypo(0,y_\circ)|t^2}\exp\left[-6\left(\frac{x-\mu_1(t,0,y_\circ)}{\hypo(0,y_\circ)t^{3/2}}\right)^2\right.\\
&\qquad \qquad \left.-6\left(\frac{x-\mu_1(t,0,y_\circ)}{\hypo(0,y_\circ)t^{3/2}}\right)\left(\frac{y-\mu_2(t,0,y_\circ)}{t^{1/2}}\right) -2\left(\frac{y-\mu_2(t,x_\circ0,y_\circ)}{t^{1/2}}\right)^2\right] \\
\end{aligned}\end{equation}
which will act as an approximate \lq\lq single layer potential" which will allow us to include the effect of $u_\bdy$ in \eqref{E:MainPDE}.  From Proposition \ref{P:KEvol}, we of course have that $\genP^L_{0,y_\circ}\sidekernel_{y_\circ}=0$ on $\openstrip$.  As with $\kernel$, we will improve upon $\sidekernel$ in Section \ref{S:Corrections}.

To more efficiently prove some needed estimates on $\kernel$ and $\sidekernel$, let's prove a Gaussian-type bound with simpler and separated dependence on $\xx_{\SP}$ and $\yy_{\SP}$.  The exponential nature of this bound will allow us to control various smaller errors.  We secondly prove a bound which allows us to explicitly bound various integrals in $x_\circ$, $y_\circ$ and $t$.

Let's start by comparing $\XX_{\SP}$ and $\YY_{\SP}$ of \eqref{E:XXYYDef} to, respectively, $\xx_{\SP}$ and $\yy_{\SP}$ of \eqref{E:xxyydef}; we have that
\begin{equation} \label{E:differences}
\left|\XX_{\SP}(t,x)-\xx_{\SP}(t,x)\right|
\le \KK[E:boundedness]t^{1/2} \qquad \text{and}\qquad 
\left|\YY_{\SP}(t,y)-\yy_{\SP}(t,y)\right|\le \KK[E:boundedness]t^{1/2}
\end{equation}
for $(\SP)\in \RR$ and $(t,x,y)\in \openstrip$.

Let's write down some standard bounds on squares.
\begin{lemma}\label{L:Young}
For any $\xi$ and $\tilde \xi$ in $\R$,
\begin{equation*} (\xi+\tilde \xi)^2\le 2\xi^2 + 2\tilde \xi^2 \qquad \text{and}\qquad (\xi+\tilde \xi)^2\ge \tfrac12 \xi^2 - \tilde \xi^2. \end{equation*}
\end{lemma}
\begin{proof} By the Cauchy–Bunyakovsky–Schwarz inequality (or alternately, convexity),
\begin{equation*} (\xi+\tilde \xi)^2 = \xi^2+2\xi\tilde \xi+\tilde \xi^2\le 2\xi^2+2\tilde \xi^2 \end{equation*}
which is the first claim.  Writing
\begin{equation*} \xi^2=(\xi+\tilde \xi-\tilde \xi)^2 \le 2(\xi+\tilde \xi)^2 + 2\tilde \xi^2 \end{equation*}
and rearranging, we get the second claim.\end{proof}

For $\alpha>0$ and $(\SP)\in \RR$, define
\begin{equation}\label{E:kernelboundDef}\begin{aligned} \kernelbound{\alpha}_{\SP}(t,x,y)&\Def \frac{1}{|\hypo(\SP)|t^2}\exp\left[-\alpha\xx^2_{\SP}(t,x)-\alpha\yy^2_{\SP}(t,y)\right]\\
&= \lb \frac{1}{|\hypo(\SP)|t^{3/2}}\exp\left[-\alpha\left(\frac{x-x_\circ+b_1(\SP)t}{\hypo(\SP)t^{3/2}}\right)^2\right]\rb \lb \frac{1}{t^{1/2}}\exp\left[-\alpha\left(\frac{y-y_\circ}{t^{1/2}}\right)^2\right]\rb \end{aligned}\end{equation} 
for $(t,x,y)\in \openstrip$.
\begin{lemma}\label{L:KandhatK}
For $(\SP)\in \RR$,
\begin{equation*}\left|\kernel_{\SP}(t,x,y)\right|\le \KK[E:KandhatKConst] \kernelbound{\sfrac{1}{6}}_{\SP}(t,x,y)  \qquad \text{and}\qquad
\left|\sidekernel_{y_\circ}(t,x,y)\right|\le \KK[E:KandhatKConst] \kernelbound{\sfrac{1}{6}}_{0,y_\circ}(t,x,y)  \end{equation*}
for $(t,x,y)\in \openstrip$, where
\begin{equation} \label{E:KandhatKConst}
\KK[E:KandhatKConst]\Def \frac{\sqrt{12}}{2\pi}\exp\left[\frac23 \KK[E:boundedness]^2\TT \right]\max\lb 1,\KK[E:boundedness]\rb. \end{equation}
\end{lemma}
\begin{proof}
Let's start with a lower bound on the quadratic form in \eqref{E:cEDef}.
Set $\eps \Def \sqrt{\sfrac{83}{135}}$.  
For $(u,v)\in \RR$,
the Cauchy–Bunyakovsky–Schwarz inequality (i.e., Young's inequality with $\eps$) implies that
\begin{equation} \label{E:qformlowerbound}\begin{aligned} 6u^2+6uv+2v^2 
&= 6u^2+2v^2+ 2\left(3\eps u\right)\left(\frac{v}{\eps}\right)
\ge 6u^2+2v^2-\left(3\eps u\right)^2-\left(\frac{v}{\eps}\right)^2\\
&=\lb 6-9\eps^2\rb u^2 + \lb 2-\frac{1}{\eps^2}\rb v^2 
= \frac{63}{135}u^2 + \frac{31}{83}v^2 
\ge \frac13\{u^2+v^2\}. \end{aligned}\end{equation}

Fix $(\SP)\in \RR$ and $(t,x,y)\in \openstrip$. 
 Then
\begin{equation*} \left|\kernel_{\SP}(t,x,y)\right|
\le \frac{\sqrt{12}}{2\pi}\frac{1}{|\hypo(\SP)|t^2}\exp\left[-\tfrac13 \XX^2_{\SP}(t,x)-\tfrac13 \YY^2_{\SP}(t,y)\right]. \end{equation*}
Combining \eqref{E:differences} and Lemma \ref{L:Young}, we have
\begin{align*}
\XX_{\SP}^2(t,x)
&\ge \frac12 \xx^2_{\SP}(t,x) - \KK[E:boundedness]^2 t\ge  \frac12 \xx^2_{\SP}(t,x) - \KK[E:boundedness]^2 \TT \\
\YY^2_{\SP}(t,y)
&\ge \frac12 \yy^2_{\SP}(t,y)-\KK[E:boundedness]^2 t \ge \frac12 \yy^2_{\SP}(t,y)-\KK[E:boundedness]^2 \TT  \end{align*}
and the bound on $\kernel$ follows.  Since $|b_1(0,y_\circ)|\le \KK[E:boundedness]$, the claim on $\sidekernel$ also follows.
\end{proof}
\noindent An alternate proof of \eqref{E:qformlowerbound} would be that the eigenvalues of
\begin{equation*} \begin{pmatrix} 6 & 3 \\ 3 & 2\end{pmatrix}\end{equation*}
are 
\begin{equation*} \tfrac12\lb 8 \pm \sqrt{52}\rb >\tfrac13;\end{equation*}
thus
\begin{equation*} \frac12 \begin{pmatrix} 12 & 6 \\ 6 & 4 \end{pmatrix} =\begin{pmatrix} 6 & 3 \\ 3 & 2\end{pmatrix}\ge  \frac13 \begin{pmatrix} 1 & 0 \\
    0 & 1\end{pmatrix}. \end{equation*}

We note that $\alpha \mapsto \kernelbound{\alpha}$ is pointwise decreasing in $\alpha$.  In particular
\begin{equation}\label{E:kernelmonotonicity} \kernelbound{\sfrac{1}{12}}_{\SP}(t,x,y)= \exp\left[\tfrac{1}{12}\lb \xx^2_{\SP}(t,x)+\yy^2_{\SP}(t,y)\rb \right]\kernelbound{\sfrac{1}{6}}_{\SP}(t,x,y) \end{equation}
for all $(\SP)\in \RR$ and $(t,x,y)\in \openstrip$.  This will allow us to bound various various pre-exponential terms which are polynomial in $\xx_{\SP}$ and $\yy_{\SP}$ while still retaining Gaussian tails.

\section{Parametrix Calculations in the Interior}\label{S:Parametrix}

We want to use $\kernel$ of \eqref{E:kernelDef} as a parametrix for \eqref{E:MainPDE} in $\domint$.  Proposition \ref{P:KEvol} suggests that $\kernel$ approximately is in the kernel of $\genP$ of \eqref{E:genPdef}.  We also need to show \emph{integrability} and that it is approximately a Dirac function as $t\searrow 0$.   We will improve upon $\kernel$ in Section \ref{S:Corrections}, so we will show appropriate integrability of $\kernelbound{\sfrac{1}{12}}$ of \eqref{E:kernelboundDef}; keep Lemma \ref{L:KandhatK} and \eqref{E:kernelmonotonicity} in mind.

A simple Gaussian calculation (reflecting the fact that $\kernel$ was defined as a density) shows that $\kernelbound{\alpha}_{\SP}(t,\cdot,\cdot)$ should be integrable for each $\alpha>0$, $(\SP)\in \RR$ and $t>0$.  In using $\kernel$ as a parametrix, however, we want to integrate against $(\SP)$.  
For $(\SP)\in \RR$ and $(t,x,y)\in \openstrip$, set
\begin{equation}\label{E:kernelboundintegrableDef} \begin{aligned} \kernelboundintegrable_{\SP}(t,x,y)&\Def \lb \bOne_{\{|x-x_\circ|\ge 2\KK[E:boundedness]\TT\}}\frac{t^{3/2}}{|x-x_\circ|^2+1}\right. \\
&\qquad \left.+ \bOne_{\{|x-x_\circ|< 2\KK[E:boundedness]\TT\}}\frac{1}{\hypo(x,y_\circ)t^{3/2}}\exp\left[-\frac{1}{\KK[E:Kashmir]}\left(\frac{x-x_\circ+b_1(x,y_\circ)t}{\hypo(x,y_\circ)t^{3/2}}\right)^2\right]\rb\\
&\qquad \qquad \times \frac{1}{\sqrt{t}}\exp\left[-\frac{(y-y_\circ)^2}{12t}\right] \end{aligned}\end{equation}
where
\begin{equation} \label{E:Kashmir} \KK[E:Kashmir] \Def 48\exp\left[4\KK[E:boundedness]\KK[E:hypobound]\TT\right]. \end{equation}
\begin{proposition}\label{P:integrablekernel}  There is a $\KK[E:integrablekernel]>0$ such that
\begin{equation} \label{E:integrablekernel}
\kernelbound{\sfrac{1}{12}}_{\SP}(t,x,y)\le \KK[E:integrablekernel]\kernelboundintegrable_{\SP}(t,x,y)\end{equation}
for all $(\SP)\in \RR$ and $(t,x,y)\in \openstrip$.\end{proposition}
\noindent Before starting the proof, let's recall that all of the terms in the Taylor series of $x\mapsto e^x$ are nonnegative for $x>0$; for any nonnegative $x>0$ and any nonnegative integer $p$, 
\begin{equation}\label{E:firstexpbound} e^x\ge \tfrac{1}{p!}x^p.\end{equation}
Thus
\begin{equation}\label{E:expbound} e^{-x}\le \frac{p!}{x^p} \end{equation}
for $x>0$ and all nonnegative integers $p$.
\begin{proof}[Proof of Proposition \ref{P:integrablekernel}]
Fix $(\SP)\in \RR$ and $(t,x,y)\in \openstrip$. 
 We divide the proof into several cases; see Figure \ref{fig:DiracDecomposition}.

 \begin{figure}[ht] \includegraphics[width=0.4\textwidth]{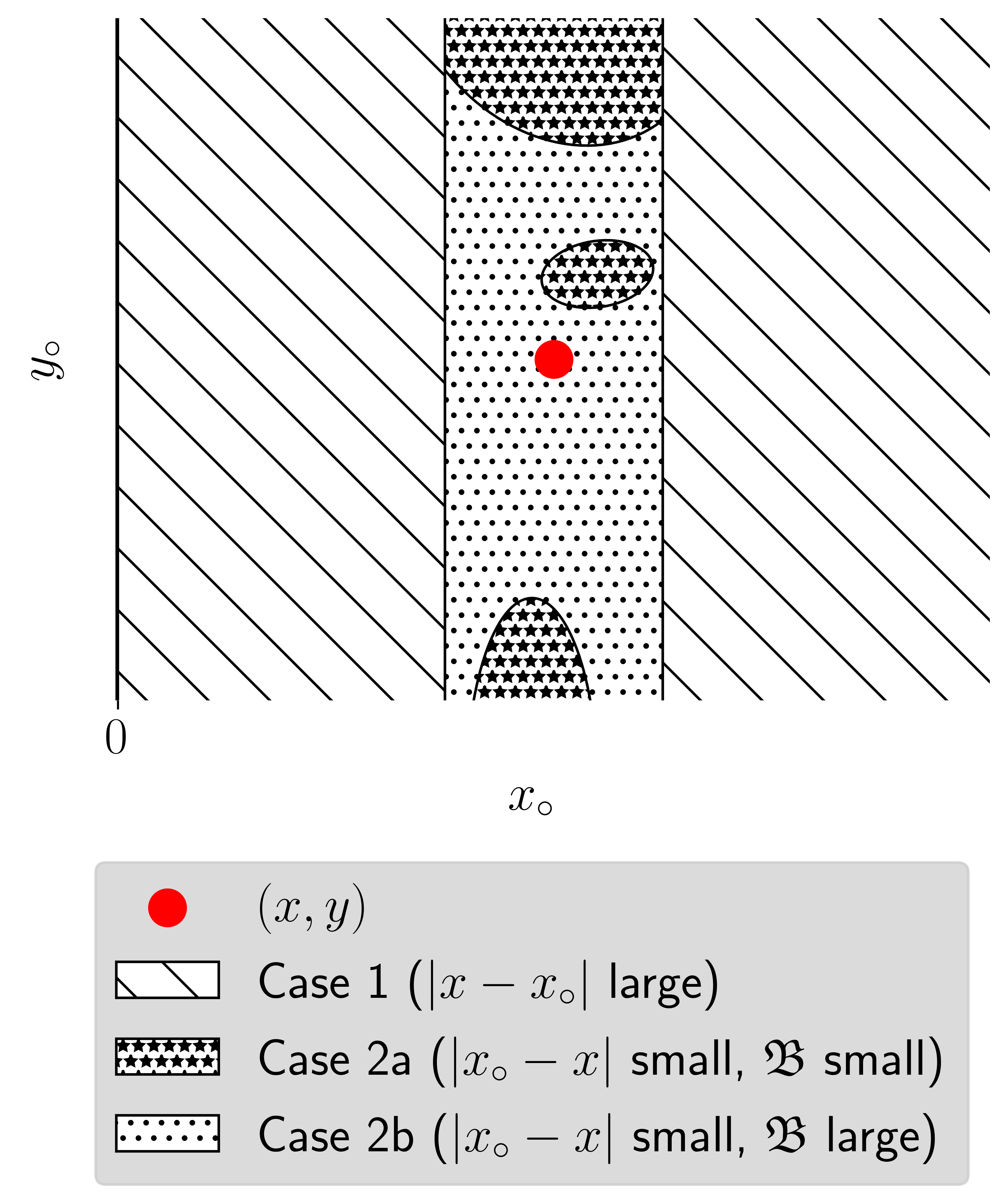}\caption{Cases in proof of Proposition  \ref{P:integrablekernel}}\label{fig:DiracDecomposition}
\end{figure}

\noindent $\bullet$ \emph{Case 1}:  Assume first that $|x-x_\circ|\ge 2\KK[E:boundedness]\TT$.  By the triangle inequality,
\begin{equation*} \left|x-x_\circ+b_1(\SP)t\right|
\ge \left|x-x_\circ\right|-|b_1(\SP)|t
=\tfrac12 |x-x_\circ|+\lb \tfrac12|x-x_\circ|-\KK[E:boundedness]\TT\rb  \ge \tfrac12|x-x_\circ|>0.\end{equation*}
Using \eqref{E:expbound} with $p=1$ and and \eqref{E:hypoisbounded}, we have that
\begin{align*} &\frac{1}{\hypo(\SP)t^{3/2}}\exp\left[-\frac{\left(x-x_\circ+b_1(\SP)t\right)^2}{12\hypo^2(\SP) t^3}\right]
\le \frac{1}{\hypo(\SP)t^{3/2}}\exp\left[-\frac{\left|x-x_\circ\right|^2}{48\hypo^2(\SP)t^3}\right] \\
&\qquad \le \frac{1}{\hypo(\SP)t^{3/2}}\frac{48\hypo^2(\SP)t^3}{\left|x-x_\circ\right|^2} 
=\frac{48\hypo(\SP)t^{3/2}}{|x-x_\circ|^2}
\le \frac{48\KK[E:boundedness]t^{3/2}}{|x-x_\circ|^2}\\
&\qquad = \frac{48\KK[E:boundedness]t^{3/2}}{1+|x-x_\circ|^2}\frac{1+|x-x_\circ|^2}{|x-x_\circ|^2}
= \frac{48\KK[E:boundedness]t^{3/2}}{1+|x-x_\circ|^2}\lb 1+\frac{1}{|x-x_\circ|^2}\rb\\
&\qquad = \lb 1+\frac{1}{2\KK[E:boundedness]\TT}\rb\frac{48\KK[E:boundedness]t^{3/2}}{1+|x-x_\circ|^2}.
\end{align*}
The claim follows in this case.

\noindent $\bullet$ \emph{Case 2}.  We now assume that 
\begin{equation}\label{E:case2} |x-x_\circ|< 2\KK[E:boundedness]\TT.\end{equation}
Using \eqref{E:hypoboundsandwich} with base point $\left(x,y_\circ\right)$ and evaluation point $(x_\circ,y_\circ)$, we then have that 
 \begin{equation}\label{E:journey} \kk_1^{-1}\le \frac{\hypo\left(x_\circ,y_\circ\right)}{\hypo\left(x,y_\circ\right)}\le \kk_1\end{equation}
 where
 \begin{equation*} \kk_1 \Def \exp\left[\KK[E:hypobound]\lb 2\KK[E:boundedness]\TT\rb \right] = \exp\left[2\KK[E:boundedness]\KK[E:hypobound]\TT\right]; \end{equation*}
then $\KK[E:Kashmir] =  48\kk_1^2$.

Using this and Lemma \ref{L:Young}, we have
\begin{equation}\label{E:physicalgraffiti}\begin{aligned} \frac{1}{12}\left(\frac{x-x_\circ+b_1(\SP)t}{\hypo(\SP)t^{3/2}}\right)^2
&\ge \frac{\left(x-x_\circ+b_1(x,y_\circ)t+ \left(b_1(\SP)-b_1(x,y_\circ)\right)t\right)^2}{12\kk_1^2\hypo^2(x,y_\circ)t^3}\\
&\ge \frac{1}{24\kk_1^2}\left(\frac{x-x_\circ+b_1(x,y_\circ)t}{\hypo^2(x,y_\circ)t^{3/2}}\right)^2-\frac{\left(b_1(\SP)-b_1(x,y_\circ)\right)^2}{12\kk_1^2\hypo^2(x,y_\circ)t}.\end{aligned}\end{equation}
We have several bounds on the last term.  Using Assumption \ref{A:boundedness}, we have
\begin{equation*}\left|b_1(\SP)-b_1(x,y_\circ)\right|\le 2\KK[E:boundedness]. \end{equation*}
Alternately, combining \eqref{E:Tboundsb1} and \eqref{E:case2}, we also have
\begin{equation} \label{E:heart}\begin{aligned} \left|b_1(\SP)-b_1(x,y_\circ)\right|
&\le \KK[E:hypobound]\exp\left[\KK[E:hypobound]\lb 2\KK[E:boundedness]\TT\rb \right]\hypo(x,y_\circ) |x-x_\circ|\\
&=\KK[E:hypobound]\kk_1\hypo(x,y_\circ)\left|x-x_\circ+b_1(x,y_\circ)t-b_1(x,y_\circ)t\right|\\
&\le \KK[E:hypobound]\kk_1\hypo(x,y_\circ)\lb \left|x-x_\circ+b_1(x,y_\circ)t\right|+\left|b_1(x,y_\circ)\right|t\rb \\
&\le \KK[E:hypobound]\kk_1\hypo(x,y_\circ)\left|x-x_\circ+b_1(x,y_\circ)t\right|+\KK[E:hypobound]\kk_1\KK[E:boundedness]\hypo(x,y_\circ)t. \end{aligned}\end{equation}
We thus have the two bounds (the second one depending on Lemma \ref{L:Young} and some algebraic rearrangement)
\begin{equation} \label{E:HousesoftheHoly}\begin{aligned}
\frac{\left(b_1(\SP)-b_1(x,y_\circ)\right)^2}{12\kk_1^2\hypo^2(x,y_\circ)t}
&\le \frac{4\KK[E:boundedness]^2}{12\kk_1^2\left(\hypo(x,y_\circ)\sqrt{t}\right)^2}
= \frac{\KK[E:boundedness]^2}{3\kk_1^2}\left(\hypo(x,y_\circ)\sqrt{t}\right)^{-2}\\
\frac{\left(b_1(\SP)-b_1(x,y_\circ)\right)^2}{12\kk_1^2\hypo^2(x,y_\circ)t}
&\le 2\KK[E:hypobound]^2\kk_1^2\hypo^2(x,y_\circ)\frac{\left(x-x_\circ+b_1(x,y_\circ)t\right)^2}{12\kk_1^2\hypo^2(x,y_\circ)t}
+2\KK[E:boundedness]^2\KK[E:hypobound]^2\kk_1^2\frac{\hypo^2(x,y_\circ)t^2}{12\hypo^2(x,y_\circ)t} \\
&\le \lb 8\KK[E:hypobound]^2\kk_1^2 t\rb \left(\hypo(x,y_\circ)\sqrt{t}\right)^2 \left(\frac{1}{48\kk_1^2}\right)\left(\frac{x-x_\circ+b_1(x,y_\circ)t}{\hypo(x,y_\circ)t^{3/2}}\right)^2
+\tfrac16\KK[E:boundedness]^2\KK[E:hypobound]^2\kk_1^2t \\
&\le \lb 8\KK[E:hypobound]^2\kk_1^2 \TT\rb \left(\hypo(x,y_\circ)\sqrt{t}\right)^2 \left(\frac{1}{48\kk_1^2}\right)\left(\frac{x-x_\circ+b_1(x,y_\circ)t}{\hypo(x,y_\circ)t^{3/2}}\right)^2
+\tfrac16\KK[E:boundedness]^2\KK[E:hypobound]^2\kk_1^2\TT. 
\end{aligned}\end{equation}

\noindent $\bullet$ \emph{Case 2a}.  In addition to \eqref{E:case2}, we now assume that
\begin{equation*} \hypo(x,y_\circ)\sqrt{t} < \frac{1}{\sqrt{8\KK[E:hypobound]^2\kk_1^2 \TT}}. \end{equation*}
Combining \eqref{E:physicalgraffiti} and the second part of \eqref{E:HousesoftheHoly}, we have
\begin{equation*}\frac{1}{12}\left(\frac{x-x_\circ+b_1(\SP)t}{\hypo(\SP)t^{3/2}}\right)^2
\ge \frac{1}{48\kk_1^2}\left(\frac{x-x_\circ+b_1(\SP)t}{\hypo(x.y_\circ)t^{3/2}}\right)^2-\tfrac16\KK[E:boundedness]^2\KK[E:hypobound]^2\kk_1^2\TT.\end{equation*}
Again using \eqref{E:journey}, we have that
\begin{align*}&\frac{1}{\hypo(\SP)t^{3/2}}\exp\left[-\frac{\left(x-x_\circ+b_1(\SP)t\right)^2}{12\hypo^2(\SP) t^3}\right]\\
&\qquad \le \frac{\kk_1\exp\left[\tfrac16\KK[E:boundedness]^2\KK[E:hypobound]^2\kk_1^2\TT\right]}{\hypo(x,y_\circ)t^{3/2}}\exp\left[-\frac{1}{48\kk_1^2}\left(\frac{x-x_\circ+b_1(\SP)t}{\hypo(x.y_\circ)t^{3/2}}\right)^2\right]
\end{align*}
which implies the claim in this case.

\noindent $\bullet$ \emph{Case 2b}.  In addition to \eqref{E:case2}, we now assume that
\begin{equation*} \hypo(x,y_\circ)\sqrt{t} \ge \frac{1}{\sqrt{8\KK[E:hypobound]^2\kk_1^2 \TT}}. \end{equation*}
Combining \eqref{E:physicalgraffiti} and the first part of \eqref{E:HousesoftheHoly}, we have
\begin{equation*}\frac{1}{12}\left(\frac{x-x_\circ+b_1(\SP)t}{\hypo(\SP)t^{3/2}}\right)^2
\ge \frac{1}{24\kk_1^2}\left(\frac{x-x_\circ+b_1(\SP)t}{\hypo(x.y_\circ)t^{3/2}}\right)^2-\frac{\KK[E:boundedness]^2}{3\kk_1^2}\left(8\KK[E:hypobound]^2\kk_1^2 \TT\right).\end{equation*}
Again using \eqref{E:journey}, we have that
\begin{align*}&\frac{1}{\hypo(\SP)t^{3/2}}\exp\left[-\frac{\left(x-x_\circ+b_1(\SP)t\right)^2}{12\hypo^2(\SP) t^3}\right]\\
&\qquad \le \frac{\kk_1\exp\left[\tfrac{8}{3} \KK[E:boundedness]^2\KK[E:hypobound]^2\TT\right]}{\hypo(x,y_\circ)t^{3/2}}\exp\left[-\frac{1}{24\kk_1^2}\left(\frac{x-x_\circ+b_1(\SP)t}{\hypo(x.y_\circ)t^{3/2}}\right)^2\right]
\end{align*}
which implies the claim in this final case.
\end{proof}

The definition \eqref{E:kernelboundintegrableDef} suggests several transformations.
\begin{remark}\label{R:transformationone}
For fixed $(t,x,y)\in \openstrip$, the change of variables
\begin{equation*} (u_\circ,v_\circ)=\left(x-x_\circ,\frac{y-y_\circ}{t^{1/2}}\right) \end{equation*}
from $(x_\circ,y_\circ)$ to $(u_\circ,v_\circ)$ has inverse
\begin{equation*} (x_\circ,y_\circ) = \left(x-u_\circ,y-t^{1/2}v_\circ\right) \end{equation*}
which has Jacobian $dx_\circ\, dy_\circ = t^{1/2}du_\circ\, dv_\circ$.
We note that
\begin{equation*}t^{1/2}\kernelboundintegrable_{x-u_\circ,y-t^{1/2}v_\circ}(t,x,y)\bOne_{\{|u_\circ|\ge 2\KK[E:boundedness]\TT\}}\le \frac{t^{3/2}}{u_\circ^2+1}\exp\left[-\frac{v_\circ^2}{12}\right]. \end{equation*}
We also have that
\begin{equation} \label{E:integrabilityone}\int_{(u_\circ,v_\circ)\in \RR}\frac{t^{3/2}}{u_\circ^2+1}\exp\left[-\frac{v_\circ^2}{12}\right]du_\circ\, dv_\circ = t^{3/2}\pi\sqrt{12\pi}. \end{equation}
\end{remark}

\begin{remark}\label{R:transformationtwo}
For fixed $(t,x,y)\in \openstrip$, the change of variables
\begin{equation*} (u_\circ,v_\circ) = \left(\frac{x-x_\circ+b_1(x,y_\circ)t}{\hypo(x,y_\circ)t^{3/2}},\frac{y-y_\circ}{t^{1/2}}\right) \end{equation*}
from $(x_\circ,y_\circ)$ to $(u_\circ,v_\circ)$ has inverse
\begin{equation*} (x_\circ,y_\circ)=\left(x+\zeta_1(t,x,y,u_\circ,v_\circ),y-t^{1/2}v_\circ\right) \end{equation*}
where
\begin{equation}\label{E:zetadef} \zeta_1(t,x,y,u_\circ,v_\circ)=b_1\left(x,y-t^{1/2}v_\circ\right)t- \hypo\left(x,y-t^{1/2}v_\circ\right)t^{3/2}u_\circ. \end{equation}
The determinant of the Jacobian of this inverse transformation \textup{(}which determines the transformation of volume elements\textup{)} is
\begin{equation*} dx_\circ\, dy_\circ = \left|\det \begin{pmatrix} -\hypo(x,y-t^{1/2}v_\circ)t^{3/2} & \sfrac{\partial \zeta_1}{\partial v_\circ}(t,x,y,u_\circ,v_\circ) \\ 0 & -t^{1/2}\end{pmatrix}\right|du_\circ\, dv_\circ = t^2\hypo\left(x,y-t^{1/2}v_\circ \right) du_\circ\, dv_\circ. \end{equation*}
We note that
\begin{equation*}t^2\hypo\left(x,y-t^{1/2}v_\circ \right)\kernelboundintegrable_{x+\zeta_1(t,x,y,u_\circ,v_\circ),y-t^{1/2}v_\circ}(t,x,y)\bOne_{\{|\zeta_1(t,x,y,u_\circ,v_\circ)|< 2\KK[E:boundedness]\TT\}}\le \exp\left[-\frac{u_\circ^2}{\KK[E:Kashmir]}-\frac{v_\circ^2}{12}\right]. \end{equation*}
We also have that
\begin{equation} \label{E:integrabilitytwo}\int_{(u_\circ,v_\circ)\in \RR}\exp\left[-\frac{u_\circ^2}{\KK[E:Kashmir]}-\frac{v_\circ^2}{12}\right]du_\circ\, dv_\circ = \sqrt{\KK[E:Kashmir]\pi}\sqrt{12\pi}. \end{equation}
\end{remark}
\begin{proposition}\label{P:integrablekernelbound}
We have that
\begin{equation}\label{E:integrablekernelbound} \KK[E:integrablekernelbound]\Def \sup_{(t,x,y)\in \openstrip}\int_{(\SP)\in \RR}\kernelboundintegrable_{\SP}(t,x,y)dx_\circ\, dy_\circ \end{equation}
is finite.\end{proposition}
\begin{proof}
Fix $(t,x,y)\in \openstrip$.  From Remarks \ref{R:transformationone} and \ref{R:transformationtwo},
\begin{align*}
&\int_{(x_\circ,y_\circ)\in \RR}\kernelboundintegrable_{\SP}(t,x,y) \bOne_{\{|x-x_\circ|\ge 2\KK[E:boundedness]\TT\}}dx_\circ\, dy_\circ \\
&\qquad = \int_{(u_\circ,v_\circ)\in \RR}t^{1/2}\kernelboundintegrable_{x-u_\circ,y-t^{1/2}v_\circ}(t,x,y)\bOne_{\{|u_\circ|\ge 2\KK[E:boundedness]\TT\}}du_\circ\, dv_\circ \\
&\int_{(x_\circ,y_\circ)\in \RR}\kernelboundintegrable_{\SP}(t,x,y) \bOne_{\{|x-x_\circ|< 2\KK[E:boundedness]\TT\}}dx_\circ\, dy_\circ\\
&\qquad = \int_{(u_\circ,v_\circ)\in \RR} t^2\hypo\left(x,y-t^{1/2}v_\circ \right)\kernelboundintegrable_{x+\zeta_1(t,x,y,u_\circ,v_\circ),y-t^{1/2}v_\circ}(t,x,y)\bOne_{\{|\zeta_1(t,x,y,u_\circ,v_\circ)|< 2\KK[E:boundedness]\TT\}}du_\circ\, dv_\circ.
\end{align*}
The claim follows from \eqref{E:integrabilityone} and \eqref{E:integrabilitytwo}.
\end{proof}

We also have that $\kernel$ approaches a Dirac measure in small time.
\begin{proposition}\label{P:dirac}
Fix $g\in C_b([0,\TT)\times \Rplusint\times \R)$.  For any $(x^*,y^*)\in \bdy_i$, we have that
\begin{equation*} \lim_{\substack{(t,x,y)\to (0,x^*,y^*) \\ (t,x,y)\in \domint}}\int_{(\SP)\in \bdy_i}\kernel_{\SP}(t,x,y)g(t,\SP)dx_\circ\, dy_\circ = g(0,x^*,y^*). \end{equation*}
\end{proposition}
\begin{proof}  Based on the definition \eqref{E:kernelboundintegrableDef} of $\kernelboundintegrable$, let's write
\begin{equation} \label{E:decomp}\begin{aligned}
    &\int_{(\SP)\in \bdy_i}\kernel_{\SP}(t,x,y)g(t,\SP)dx_\circ\, dy_\circ\\
    &\qquad=\int_{(\SP)\in \bdy_i}\bOne_{\{|x-x_\circ|\ge  2\KK[E:boundedness]\TT\}}\kernel_{\SP}(t,x,y)g(t,\SP)dx_\circ\, dy_\circ\\
    &\qquad \qquad +\int_{(\SP)\in \bdy_i}\bOne_{\{|x-x_\circ|<  2\KK[E:boundedness]\TT\}}\kernel_{\SP}(t,x,y)g(t,\SP)dx_\circ\, dy_\circ. 
\end{aligned}\end{equation}

Using Remark \ref{R:transformationone}, we have
\begin{align*} &\left|\int_{(\SP)\in \bdy_i}\bOne_{\{|x-x_\circ|\ge  2\KK[E:boundedness]\TT\}}\kernel_{\SP}(t,x,y)g(t,\SP)dx_\circ\, dy_\circ\right|\\
&\qquad \le \KK[E:KandhatKConst]\KK[E:integrablekernel]\|g\| \int_{(u_\circ,v_\circ)\in \RR}t^{1/2}\kernelboundintegrable_{x-u_\circ,y-t^{1/2}v_\circ}(t,x,y)\bOne_{\{|u_\circ|\ge 2\KK[E:boundedness]\TT\}}du_\circ\, dv_\circ\\
&\qquad \le \KK[E:KandhatKConst]\KK[E:integrablekernel]\|g\| t^{3/2}\pi\sqrt{12\pi},\end{align*}
implying that
\begin{equation}\label{E:nullpart} \lim_{\substack{(t,x,y)\to (0,x^*,y^*)\\
(t,x,y)\in \domint}}\int_{(\SP)\in \bdy_i}\bOne_{\{|x-x_\circ|\ge  2\KK[E:boundedness]\TT\}}\kernel_{\SP}(t,x,y)g(t,\SP)dx_\circ\, dy_\circ=0.\end{equation}

Let's use Remark \ref{R:transformationtwo} to understand the asymptotics of the second term on the right of \eqref{E:decomp}.  We have that
\begin{equation*} \int_{(\SP)\in \bdy_i}\bOne_{\{|x-x_\circ|<  2\KK[E:boundedness]\TT\}}\kernel_{\SP}(t,x,y)g(t,\SP,x,y)dx_\circ\, dy_\circ
= \int_{(u_\circ,v_\circ)\in \RR}\Phi_g(t,x,y,u_\circ,v_\circ)du_\circ\, dv_\circ \end{equation*}
where
\begin{align*} \Phi_g(t,x,y,u_\circ,v_\circ)&=t^2\hypo\left(x,y-t^{1/2}v_\circ \right)\kernel_{x+\zeta_1(t,x,y,u_\circ,v_\circ),y-t^{1/2}v_\circ}(t,x,y)\bOne_{\{|\zeta_1(t,x,y,u_\circ,v_\circ)|<2\KK[E:boundedness]\TT\}}\\
&\qquad \times \bOne_{\{x+\zeta_1(t,x,y,u_\circ,v_\circ)>0\}}g\left(t,x+\zeta_1(t,x,y,u_\circ,v_\circ),y-t^{1/2}v_\circ\right). \end{align*}
From Proposition \ref{P:integrablekernel} and recalling \eqref{E:kernelmonotonicity}, we have that
\begin{align*} &\left|\Phi_g(t,x,y,u_\circ,v_\circ)\right|\\
&\qquad \le \KK[E:KandhatKConst]\|g\|t^2\hypo\left(x,y-t^{1/2}v_\circ \right)\KK[E:integrablekernel]\kernelboundintegrable_{x+\zeta_1(t,x,y,u_\circ,v_\circ),y-t^{1/2}v_\circ}(t,x,y)\bOne_{\{|\zeta(t,x,y,u_\circ,v_\circ)|<2\KK[E:boundedness]\TT\}}\\
&\qquad \le \KK[E:KandhatKConst]\KK[E:integrablekernel]\|g\|\exp\left[-\tfrac{1}{\KK[E:Kashmir]}u_\circ^2-\tfrac{1}{12}v_\circ^2\right].
\end{align*}
Using \eqref{E:integrabilitytwo}, we can appeal to the dominated convergence theorem.

Explicitly,
\begin{align*}
\Phi_g(t,x,y,u_\circ,v_\circ)
&=\bOne_{\{x+\zeta_1(t,x,y,u_\circ,v_\circ)>0\}}\bOne_{\{|\zeta_1(t,x,y,u_\circ,v_\circ)|<2\KK[E:boundedness]\TT\}}\\
&\qquad \times \frac{\hypo\left(x,y-t^{1/2}v_\circ \right)}{\hypo\left(x+\zeta_1(t,x,y,u_\circ,v_\circ),y-t^{1/2}v_\circ\right)}\\
&\qquad \times \frac{\sqrt{12}}{2\pi}\exp\left[-6\XX_{x+\zeta_1(t,x,y,u_\circ,v_\circ),y-t^{1/2}v_\circ}(t,x)\right.\\
&\qquad \qquad \left. -6\XX_{x+\zeta_1(t,x,y,u_\circ,v_\circ),y-t^{1/2}v_\circ}(t,x)\YY_{x+\zeta_1(t,x,y,u_\circ,v_\circ),y-t^{1/2}v_\circ}(t,y)\right.\\
&\qquad \qquad \left.-2\YY_{x+\zeta_1(t,x,y,u_\circ,v_\circ),y-t^{1/2}v_\circ}(t,y)\right].
\end{align*}

Let's now note from the explicit formula \eqref{E:zetadef} that
\begin{equation} \label{E:zetaasymp}\left|\zeta_1(t,x,y,u_\circ,v_\circ)\right|\le \KK[E:boundedness]t\lb 1+t^{1/2}|u_\circ|\rb \end{equation}
for all $(t,x,y)\in \domint$ and $(u_\circ,v_\circ)\in \RR$.
Fix now $(u_\circ,v_\circ)$ in $\RR$.  From \eqref{E:zetaasymp},
\begin{equation*} \lim_{\substack{(t,x,y)\to (0,x^*,y^*)\\
(t,x,y)\in \domint}}\bOne_{\{|\zeta_1(t,x,y,u_\circ,v_\circ)|<2\KK[E:boundedness]\TT\}}=1\qquad \text{and}\qquad 
\lim_{\substack{(t,x,y)\to (0,x^*,y^*)\\
(t,x,y)\in \domint}}\bOne_{\{x+\zeta_1(t,x,y,u_\circ,v_\circ)>0\}}=1. \end{equation*}

Then
\begin{equation} \label{E:zetasmall}\lim_{\substack{(t,x,y)\to (0,x^*,y^*)\\
(t,x,y)\in \domint}}\zeta_1(t,x,y,u_\circ,v_\circ)=0. \end{equation}
From \eqref{E:XXYYDef} and \eqref{E:xxyydef}, we also have that
\begin{align*} &\XX_{x+\zeta_1(t,x,y,u_\circ,v_\circ),y-t^{1/2}v_\circ}(t,x)\\
&\qquad =\frac{-b_1\left(x,y-t^{1/2}v_\circ\right)t+ \hypo\left(x,y-t^{1/2}v_\circ\right)t^{3/2}u_\circ+b_1\left(x+\zeta_1(t,x,y,u_\circ,v_\circ),y-t^{1/2}v_\circ\right)t}{\hypo\left(x+\zeta_1(t,x,y,u_\circ,v_\circ),y-t^{1/2}v_\circ\right)t^{3/2}}\\
&\qquad\qquad -\tfrac12b_2(x+\zeta_1(t,x,y,u_\circ,v_\circ),y-t^{1/2}v_\circ)t^{1/2}\\
&\qquad=\frac{\hypo(x,y-t^{1/2}v_\circ)}{\hypo\left(x+\zeta_1(t,x,y,u_\circ,v_\circ),y-t^{1/2}v_\circ\right)}u_\circ\\
&\qquad\qquad +\frac{1}{\hypo\left(x+\zeta_1(t,x,y,u_\circ,v_\circ),y-t^{1/2}v_\circ\right)}\\
&\qquad \qquad \qquad \times \int_{s=0}^1\frac{\partial b_1}{\partial x}\left(x+s\zeta_1(t,x,y,u_\circ,v_\circ),y-t^{1/2}v_\circ\right)t^{-1/2}\zeta_1(t,x,y,u_\circ,v_\circ)ds\\
&\qquad\qquad -\tfrac12b_2(x+\zeta_1(t,x,y,u_\circ,v_\circ),y-t^{1/2}v_\circ)t^{1/2}\\
&\YY_{x+\zeta_1(t,x,y,u_\circ,v_\circ),y-t^{1/2}v_\circ}(t,y)\\
&\qquad=v_\circ+b_2(x+\zeta_1(t,x,y,u_\circ,v_\circ),y-t^{1/2}v_\circ)t^{1/2}. \end{align*}
Note that \eqref{E:zetaasymp}
implies that 
\begin{equation*} \lim_{\substack{(t,x,y)\to (0,x^*,y^*)\\
(t,x,y)\in \domint}}t^{-1/2}\zeta_1(t,x,y,u_\circ,v_\circ)=0,\end{equation*}
which is in fact stronger than \eqref{E:zetasmall}.
Collecting things together and taking limits, we have 
\begin{equation*} \lim_{\substack{(t,x,y)\to (0,x^*,y^*)\\
(t,x,y)\in \domint}}\Phi_g(t,x,y,u_\circ,v_\circ)=\frac{\sqrt{12}}{2\pi}\exp\left[-6u_\circ-6u_\circ v_\circ -2v_\circ^2\right]g(0,x^*,y^*) \end{equation*}
for each $(u_\circ,v_\circ)\in \RR$.
Since
\begin{equation}\label{E:unitintegral} \int_{(u_\circ,v_\circ)\in \RR}\frac{\sqrt{12}}{2\pi}\exp\left[-\frac12 \begin{pmatrix} u_\circ & v_\circ\end{pmatrix}\begin{pmatrix} 12 & 6 \\ 6 & 4 \end{pmatrix}\begin{pmatrix} u_\circ \\ v_\circ \end{pmatrix}\right] du_\circ\, dv_\circ=1, \end{equation}
we indeed get that
\begin{equation}\label{E:nonnullpart}\lim_{\substack{(t,x,y)\to (0,x^*,y^*)\\
(t,x,y)\in \domint}}\int_{(u_\circ,v_\circ)\in \RR}\Phi_g(t,x,y,u_\circ,v_\circ)du_\circ dv_\circ=g(0,x^*,y^*).\end{equation}

Combine \eqref{E:decomp}, \eqref{E:nullpart}, and \eqref{E:nonnullpart} to get the claim.
\end{proof}

\section{Sides}
Let's next understand $\sidekernel$ of \eqref{E:sidekernelDef}.  Some of our calculations will parallel the calculations of Section \ref{S:Parametrix}.  Since these calculations are specifically oriented towards the side boundary of $\dom$, we restrict $(t,x,y)$ to be in $\domplus$ (see \cite{anceschi2019survey} for 

\begin{remark}[Laplace asymptotic arguments]
Let's understand the main idea of some asymptotic arguments by simplifying $ds\, dy_\circ$ integral of $\kernel_{0,y_\circ}(s,x,y)$ \textup{(}i.e, the integral of $\kernel$ against the side boundary of $\dom$\textup{)}.  Fix $g\in C_b([0,1])$ and consider the integral
\begin{equation} \label{E:coreI} \int_{s=0}^1 s^{-3/2}\exp\left[-\frac{(x-s)^2}{s^3}\right]g(s)ds. \qquad x>0 \end{equation}
Informally, the integrand corresponds to replacing $b_1$ by $-1$ and $\hypo$ by $1$ in $\kernel_{0,y_\circ}(t,x,y)$ and then integrating out \textup{(}i.e., suppressing the dependence on\textup{)} $y_\circ$.  We are interested in behavior of \eqref{E:coreI} as $x\searrow 0$.

The theory of Laplace-type asymptotics suggest that we start by looking at the minimum of
\begin{equation*} t\mapsto \frac{(x-s)^2}{s^3} \end{equation*}
which is, of course $x$.  For $s$ near $x$, the $ds$ integral should look approximately Gaussian.  Of course for $s$ near $x$, we should also have $s^3\approx x^3$.  As $x\searrow 0$, we might more specifically look at $s\in (x/2,2x)$, which is a small neighborhood of $x$ which scales like $x$.  We might expect that
\begin{align*} &\int_{s=x/2}^{2x}s^{-3/2}\exp\left[-\frac{(x-s)^2}{s^3}\right]g(s)ds
\approx \int_{s=x/2}^{2x}x^{-3/2}\exp\left[-\frac{(x-s)^2}{x^3}\right]g(s)ds\\
&\approx \int_{r=- x^{-1/2}}^{\tfrac12x^{-1/2}}e^{-r^2}g(x-x^{3/2}r)dr\approx \sqrt{\pi} g(x) \approx \sqrt{\pi} g(0) \end{align*}
as $x\searrow 0$.  

If $s>2x$, then $s$ is \lq\lq large" compared to $x$, so we might expect that
\begin{align*} &\int_{s=2x}^1 s^{-3/2}\exp\left[-\frac{(x-s)^2}{s^3}\right]g(s)ds
\approx \int_{s=2x}^1 s^{-3/2}\exp\left[-\frac{s^2}{s^3}\right]g(s)ds \\
&\qquad \approx \int_{s=2x}^1 s^{-3/2}e^{1/2}g(s)ds 
\approx \int_{s=0}^1 s^{-3/2}e^{-1/s}g(s)ds \end{align*}
as $x\searrow 0$. We note that $s\mapsto s^{-3/2}e^{-1/2}$ is bounded on $(0,\infty)$.  See also Remark \ref{R:McKean}.

If $s<x/2$, then $x>2s$ and $x$ is \lq\lq large" compared to $s$.  The estimates of Proposition \ref{P:sideintegrablekernel} will imply that 
\begin{equation*} \sup_{\substack{ 0<s < x/2 \\ x>0}}s^{-3/2}\exp\left[-\frac{x^2}{s^3}\right]<\infty. \end{equation*}
This implies that
\begin{equation*}\lim_{x\searrow 0}\int_{s=0}^{x/2} s^{-3/2}\exp\left[-\frac{(x-s)^2}{s^3}\right]g(s)ds
=0.\end{equation*}

Collecting our estimates together, we might expect that 
\begin{equation}\label{E:newidea}
    \int_{s=0}^1 s^{-3/2}\exp\left[-\frac{(x-s)^2}{s^3}\right]g(s)ds
    \approx \sqrt{\pi} g(0) + \int_{s=0}^1 s^{-3/2}e^{-1/s}g(s)ds
\end{equation}
as $x\searrow 0$.

Informally, the integrand of \eqref{E:coreI} looks \textup{(}as $x\searrow 0$\textup{)} like the sum of a Dirac measure at $x$ and an integral against a bounded function.  See Figure \ref{fig:Laplace}.  Our estimates about integrals against $\sidekernel$ will make the equivalent of these calculations rigorous.

Several other asymptotics of integrals of the McKean distribution are given in \cite{goldman1971first}.

\begin{figure}[ht] \includegraphics[width=0.4\textwidth]{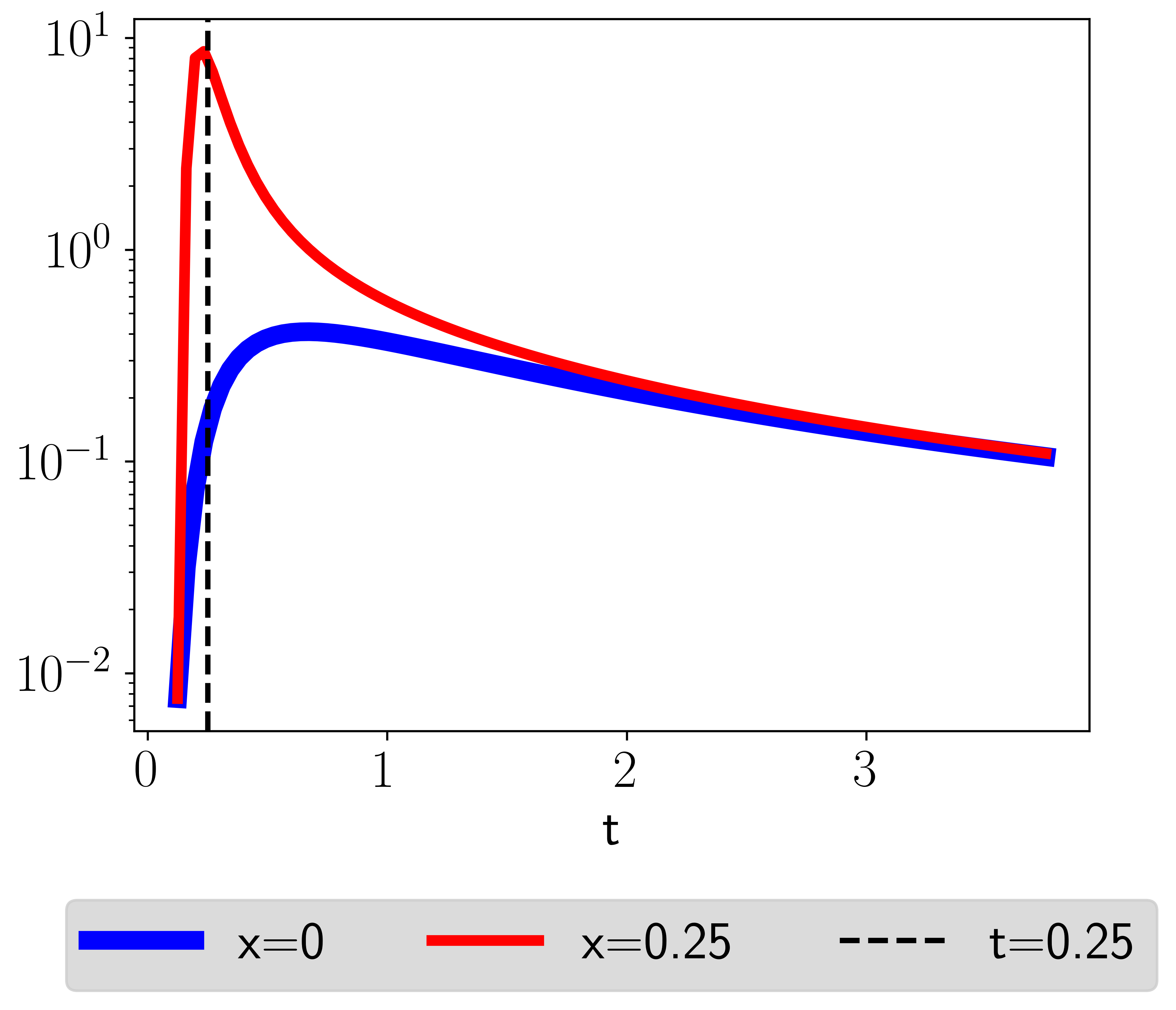}\caption{Integrand of \eqref{E:coreI}}\label{fig:Laplace}
\end{figure}
\end{remark}

\begin{remark}
The central calculation for a standard double-layer potential is 
\begin{equation*} \lim_{x\searrow 0}\int_{s=0}^1 \frac{x}{\sqrt{2\pi s^3}}\exp\left[-\frac{x^2}{2s}\right]g(s)ds = g(0) \end{equation*}
for any $g\in C_b([0,1])$ \textup{(}see also \cite{moss1989boundary}\textup{)}. 
The asymptotics of \eqref{E:newidea} thus contain a term not present in standard diffusive boundary potential theory.  A novel term \textup{(}the second term in the recursion for $\tilde \psi^\bdy_{n+1}$ of \eqref{E:explicitpsi}\textup{)} will consequently appear in the Volterra recursion for the solution of \eqref{E:MainPDE}.
\end{remark}

Let's start our formal analysis with an analogue of \eqref{E:kernelboundintegrableDef}. 
For $y_\circ\in \R$ and $(t,x,y)\in \domplus$, let's here define
\begin{equation}\label{E:sidekernelboundintegrableDef} \begin{aligned} \sidekernelboundintegrable_{y_\circ}(t,x,y)
&\Def \lb \bOne_{\{t\le  \sfrac{x}{2\KK[E:boundedness]}\}}
+\bOne_{\{t\ge  \sfrac{2x}{\ub}\}}
\rb \frac{1}{\sqrt{t}}\exp\left[-\frac{(y-y_\circ)^2}{12t}\right]\\
&\qquad +\bOne_{\{\sfrac{x}{2\KK[E:boundedness]}< t< \sfrac{2x}{\ub}\}}\frac{|b_1(0,y_\circ)|}{\hypo(0,y_\circ)x^{3/2}}\exp\left[-\frac{\ub^3}{96}\frac{\left(x+b_1(0,y_\circ)t\right)^2}{\hypo^2(0,y_\circ) x^3}\right]\\
&\qquad \qquad \times \frac{1}{\sqrt{x}}\exp\left[-\frac{\ub}{24}\frac{(y-y_\circ)^2}{x}\right].
\end{aligned}\end{equation}
\begin{proposition}\label{P:sideintegrablekernel}  There is a $\KK[E:sideintegrablekernel]>0$ such that
\begin{equation} \label{E:sideintegrablekernel}
\kernelbound{\sfrac{1}{12}}_{0,y_\circ}(t,x,y)\le \KK[E:sideintegrablekernel]\sidekernelboundintegrable_{y_\circ}(t,x,y)\end{equation}
for all $y_\circ\in \R$ and $(t,x,y)\in \domplus$.\end{proposition}
\begin{proof}  Fix $y_\circ\in \R$ and $(t,x,y)\in \domplus$. 
 We consider several cases.  First, however, let's use \eqref{E:expbound} with $p=3$ to see that
\begin{equation}\label{E:threehalvesbound} e^{-x}=\sqrt{e^{-2x}}\le \sqrt{\frac{3!}{(2x)^3}}=\frac{\sqrt{3}}{2x^{3/2}} \end{equation}
for $x>0$.  

\noindent $\bullet$ \emph{Case 1}:  Assume that 
\begin{equation*} t\le  \frac{x}{2\KK[E:boundedness]}. \end{equation*}
Then
\begin{equation*} x+b_1(0,y_\circ)t\ge x-\KK[E:boundedness]t\ge x-\frac{x\KK[E:boundedness]}{2\KK[E:boundedness]}=\frac{x}{2}>0. \end{equation*}
Using \eqref{E:threehalvesbound}, we then have
\begin{align*} &\frac{|b_1(0,y_\circ)|}{\hypo(0,y_\circ)t^{3/2}}\exp\left[-\frac{\left(x+b_1(0,y_\circ)t\right)^2}{12\hypo^2(0,y_\circ) t^3}\right]
 \le \frac{|b_1(0,y_\circ)|}{\hypo(0,y_\circ)t^{3/2}}\exp\left[-\frac{\left(x/2\right)^2}{12\hypo^2(0,y_\circ) t^3}\right] \\
& \qquad \le  \frac{\sqrt{3}}{2}\frac{|b_1(0,y_\circ)|}{\hypo(0,y_\circ)t^{3/2}}\left(\frac{48\hypo^2(0,y_\circ)t^3}{x^2}\right)^{3/2}
=\frac{48^{3/2}\sqrt{3}}{2}\hypo^2(0,y_\circ)|b_1(0,y_\circ)|\left(\frac{t}{x}\right)^3\\
&\qquad \le \frac{48^{3/2}\sqrt{3}}{2}\KK[E:boundedness]^3\left(\frac{1}{2\KK[E:boundedness]}\right)^3.
\end{align*}
The claim follows in this case.

\noindent $\bullet$ \emph{Case 2}:  Now assume that 
\begin{equation*} t\ge  \frac{2x}{\ub}. \end{equation*}
In this case,
\begin{equation*} -b_1(0,y_\circ)t -x \ge \ub t-\tfrac12 \ub t= \tfrac12 \ub t. \end{equation*}
Again using \eqref{E:threehalvesbound}, we have
\begin{align*} &\frac{|b_1(0,y_\circ)|}{\hypo(0,y_\circ)t^{3/2}}\exp\left[-\frac{\left(x+b_1(0,y_\circ)t\right)^2}{12\hypo^2(0,y_\circ) t^3}\right]
 \le \frac{|b_1(0,y_\circ)|}{\hypo(0,y_\circ)t^{3/2}}\exp\left[-\frac{\left(\frac12 \ub t\right)^2}{12\hypo^2(0,y_\circ) t^3}\right] \\
 &\qquad \le \frac{|b_1(0,y_\circ)|}{\hypo(0,y_\circ)t^{3/2}}\exp\left[-\frac{\ub^2}{48\hypo^2(0,y_\circ) t}\right] 
 \le  \frac{\sqrt{3}}{2}\frac{|b_1(0,y_\circ)|}{\hypo(0,y_\circ)t^{3/2}}\left(\frac{48\hypo^2(0,y_\circ) t}{\ub^2}\right)^{3/2}\\
&\qquad =\frac{48^{3/2}\sqrt{3}}{2}\frac{\hypo^2(0,y_\circ)|b_1(0,y_\circ)|}{\ub^3}
\le \frac{48^{3/2}\sqrt{3}}{2}\frac{\KK[E:boundedness]^3}{\ub^3}.
\end{align*}
The claim follows in this case.

\noindent $\bullet$ \emph{Case 3}:  Finally, assume that 
\begin{equation*}\frac{x}{2\KK[E:boundedness]}< t< \frac{2x}{\ub}. \end{equation*}
Here
\begin{equation*}
    \frac{\left(x+b_1(0,y_\circ)t\right)^2}{12\hypo^2(0,y_\circ)t^3}+\frac{\left(y-y_\circ\right)^2}{12t} \ge \frac{\left(x+b_1(0,y_\circ)t\right)^2}{12\hypo^2(0,y_\circ)\left(\sfrac{2x}{\ub}\right)^3}+\frac{\left(y-y_\circ\right)^2}{12\left(\sfrac{2x}{\ub}\right)}.
\end{equation*}
The claim easily follows.
\end{proof}

As with Remarks \ref{R:transformationone} and \ref{R:transformationtwo},   \eqref{E:sidekernelboundintegrableDef} suggests several transformations.
\begin{remark}\label{R:transformationthree}
For fixed $y\in \R$, the transformation
\begin{equation*} (r,v_\circ) = \left(s,\frac{y-y_\circ}{\sqrt{s}} \right)\end{equation*}
from $(s,y_\circ)$ to $(r,v_\circ)$ has inverse
\begin{equation*} (s,y_\circ)=\left(r,y-r^{1/2}v_\circ\right) \end{equation*}
which has Jacobian $ds\, dy_\circ = r^{1/2}dr\, dv_\circ$.  We note that
\begin{equation*} r^{1/2}\sidekernelboundintegrable_{y-r^{1/2}v_\circ}(r,x,y)\lb \bOne_{\{r\le  \sfrac{x}{2\KK[E:boundedness]}\}}
+\bOne_{\{r\ge  \sfrac{2x}{\ub}\}}
\rb \le \exp\left[-\tfrac{1}{12}v_\circ^2\right]. \end{equation*}
We also have that
\begin{equation}\label{E:integrabilitythree} \int_{r\in (0,\TT)}\int_{v_\circ\in \R}\exp\left[-\tfrac{1}{12}v_\circ^2\right] dr\, dv_\circ = \TT\sqrt{12\pi}. \end{equation}
\end{remark}

\begin{remark}\label{R:transformationfour}
For fixed $(x,y)\in \Rplusint\times \R$, the transformation
\begin{equation*} (r,v_\circ)=\left(\frac{x+b_1(0,y_\circ)s}{\hypo(0,y_\circ)x^{3/2}},\frac{y-y_\circ}{x^{1/2}}\right)=\left(\frac{x-\left|b_1(0,y_\circ)\right|s}{\hypo(0,y_\circ)x^{3/2}},\frac{y-y_\circ}{x^{1/2}}\right) \end{equation*}
from $(s,y_\circ)$ to $(r,v_\circ)$ has inverse
\begin{equation*}(s,y_\circ)=\left(\zeta_2(x,y,r,v_\circ),y-x^{1/2}v_\circ\right)\end{equation*}
where
\begin{equation} \label{E:zeta2def}\zeta_2(x,y,r,v_\circ)=\frac{x-\hypo\left(0,y-x^{1/2}v_\circ\right)x^{3/2}r}{\left|b_1\left(0,y-x^{1/2}v_\circ\right)\right|}=\frac{x}{\left|b_1\left(0,y-x^{1/2}v_\circ\right)\right|}\lb 1-\hypo\left(0,y-x^{1/2}v_\circ\right)x^{1/2}r\rb.
\end{equation}
The Jacobian of this inverse transformation is
\begin{equation*} ds\, dy_\circ = \left|\det\begin{pmatrix}-\frac{\hypo(0,y-x^{1/2}v_\circ)x^{3/2}}{b_1(0,y-x^{1/2}v_\circ)} & \sfrac{\partial \zeta_2}{\partial v_\circ}(x,y,r,v_\circ) \\ 0 & -x^{1/2} \end{pmatrix}\right| dr\, dv_\circ 
 = \frac{\hypo\left(0,y-x^{1/2}v_\circ\right)}{\left|b_1\left(0,y-x^{1/2}v_\circ\right)\right|}x^2 dr\, dv_\circ. \end{equation*}
We note that
\begin{align*} &\frac{\hypo\left(0,y-x^{1/2}v_\circ\right)}{\left|b_1\left(0,y-x^{1/2}v_\circ\right)\right|}x^2\sidekernelboundintegrable_{y-r^{1/2}v_\circ}(\zeta_2(x,y,r,v_\circ),x,y)\bOne_{\{\sfrac{x}{2\KK[E:boundedness]}< \zeta_2(x,y,r,v_\circ)< \sfrac{2x}{\ub}\wedge \TT\}}\\
&\qquad \le \exp\left[-\tfrac{\ub^3}{96}r^2-\tfrac{\ub}{24}v_\circ^2\right]. \end{align*}
We also have that
\begin{equation}\label{E:integrabilityfour} \int_{r\in \R}\int_{v_\circ\in \R}\exp\left[-\tfrac{\ub^3}{96}r^2-\tfrac{\ub}{24}v_\circ^2\right]dr\, dv_\circ=  \sqrt{\tfrac{96\pi}{\ub^3}}\sqrt{\tfrac{24\pi}{\ub}}.\end{equation}
\end{remark}
Let's now prove integrability (similar to Proposition \ref{P:integrablekernelbound}).
\begin{proposition}\label{P:integrablesidekernelbound}  We have that
\begin{equation} \label{E:integrablesidekernelbound} \KK[E:integrablesidekernelbound] \Def \sup_{(t,x,y)\in \domplus}\int_{s=0}^t \int_{y_\circ\in \R}\sidekernelboundintegrable_{\SP}(s,x,y)dy_\circ\, ds \end{equation}
is finite.
\end{proposition}
\begin{proof}
Fix $(t,x,y)\in \domplus$.  From Remarks \ref{R:transformationthree} and \ref{R:transformationfour}, we have that
\begin{align*} &\int_{s=0}^t \int_{y_\circ\in \R}\sidekernelboundintegrable_{\SP}(s,x,y)\bOne_{\{s\le  \sfrac{x}{2\KK[E:boundedness]}\}}
+\bOne_{\{s\ge  \sfrac{2x}{\ub}\}}dy_\circ\, ds \\
&\qquad =  \int_{r=0}^t\int_{v_\circ\in \R}r^{1/2}\sidekernelboundintegrable_{y-r^{1/2}v_\circ}(r,x,y)\lb \bOne_{\{r\le  \sfrac{x}{2\KK[E:boundedness]}\}}
+\bOne_{\{r\ge  \sfrac{2x}{\ub}\}}
\rb dr\, dv_\circ \\
&\int_{s=0}^t \int_{y_\circ\in \R}\sidekernelboundintegrable_{\SP}(s,x,y)\bOne_{\{\sfrac{x}{2\KK[E:boundedness]}< t< \sfrac{2x}{\ub}\}}dy_\circ\, ds \\
&\qquad =  \int_{r\in \R}\int_{v_\circ\in \R}\frac{\hypo\left(0,y-x^{1/2}v_\circ\right)}{\left|b_1\left(0,y-x^{1/2}v_\circ\right)\right|}x^2\sidekernelboundintegrable_{y-r^{1/2}v_\circ}(\zeta_2(x,y,r,v_\circ),x,y)\bOne_{\{\sfrac{x}{2\KK[E:boundedness]}< \zeta_2(x,y,r,v_\circ)< \sfrac{2x}{\ub}\wedge t\}}dr\, dv_\circ. \end{align*}
The claim follows from \eqref{E:integrabilitythree} and \eqref{E:integrabilityfour}.
\end{proof}

We also have the analogue of Proposition \ref{P:dirac}; i.e., a \lq\lq jump" boundary condition.
For $y_\circ\in \R$ and $(t,x,y)\in \domplus$, we explicitly have that
\begin{align*}
    \sidekernel_{y_\circ}(s,0,y)&= \frac{\sqrt{12}|b_1(0,y_\circ)|}{2\pi \hypo(0,y_\circ)t^2}\exp\left[-6\left(\frac{x+b_1(0,y_\circ)t}{\hypo(0,y_\circ)t^{3/2}}-\tfrac12b_2(0,y_\circ)t\right)^2\right.\\
&\qquad \qquad \left.-6\left(\frac{x+b_1(0,y_\circ)t}{\hypo(0,y_\circ)t^{3/2}}-\tfrac12b_2(0,y_\circ)t\right)\left(\frac{y-y_\circ}{t^{1/2}}+b_2(0,y_\circ)t\right) -2\left(\frac{y-y_\circ}{t^{1/2}}+b_2(0,y_\circ)t\right) ^2\right], 
\end{align*}
which is similar to the integrand of \eqref{E:coreI}.
\begin{proposition}\label{P:jumpboundary}
Fix $g\in C_b([0,\TT)\times \Rplus\times \R\times \R)$.  For any $(t^*,y^*)\in \bdy_s$, we have that
\begin{equation} \label{E:jumpboundary}\begin{aligned} &\lim_{\substack{(t,x,y)\to (t^*,0,y^*)\\ (t,x,y)\in \dom^\circ}}\int_{s=0}^t \int_{y_\circ\in \R}\sidekernel_{y_\circ}(s,x,y)g(s,x,y,y_\circ)ds\, dy_\circ = g(0,0,y^*,y^*)\\
&\qquad +\int_{s=0}^{t^*} \int_{y_\circ\in \R}\sidekernel_{y_\circ}(s,0,y)g(s,0,y^*,y_\circ)ds\, dy_\circ. \end{aligned}\end{equation}
\end{proposition}
\begin{proof}
Fix $(t,x,y)\in \domint$.  Let's use Remarks \ref{R:transformationthree} and \ref{R:transformationfour} to write
\begin{align*}
&\int_{s=0}^t \int_{y_\circ\in \R}\sidekernel_{y_\circ}(s,x,y)g(s,x,y,y_\circ)ds\, dy_\circ \\
&\qquad = 
\int_{s=0}^t \int_{y_\circ\in \R}\lb \bOne_{\{s\le  \sfrac{x}{2\KK[E:boundedness]}\}}+\bOne_{\{s\ge  \sfrac{2x}{\ub}\}}\rb \sidekernel_{y_\circ}(s,x,y)g(s,x,y,y_\circ)ds\, dy_\circ\\
&\qquad \qquad + \int_{s=0}^t \int_{y_\circ\in \R}\bOne_{\{\sfrac{x}{2\KK[E:boundedness]}< s< \sfrac{2x}{\ub}\}}\sidekernel_{y_\circ}(s,x,y)g(s,x,y,y_\circ)ds\, dy_\circ\\
&\qquad = \int_{r=0}^{\TT}\int_{v_\circ\in \R}\Phi^{(1)}_g(t,x,y,r,v_\circ) dr\, dv_\circ+ \int_{r\in \R}\int_{v_\circ\in \R} \Phi^{(2)}_g(t,x,y,r,v_\circ) dr\, dv_\circ
\end{align*}
where
\begin{align*}
\Phi^{(1)}_g(t,x,y,r,v_\circ) &=  r^{1/2}\sidekernel_{y-r^{1/2}v_\circ}(r,x,y)\lb \bOne_{\{r\le  \sfrac{x}{2\KK[E:boundedness]}\}}+\bOne_{\{r\ge  \sfrac{2x}{\ub}\}}\rb \bOne_{\{0<r<t\}} \\
&\qquad \times g(r,x,y,y-r^{1/2}v_\circ)\\
\Phi^{(2)}_g(t,x,y,r,v_\circ)
&= \frac{\hypo\left(0,y-x^{1/2}v_\circ\right)}{\left|b_1\left(0,y-x^{1/2}v_\circ\right)\right|}x^2\sidekernel_{y-x^{1/2}v_\circ}\left(\zeta_2(x,y,r),x,y\right)\bOne_{\lb \frac{x}{2\KK[E:boundedness]}< \zeta_2(x,y,r)< \frac{2x}{\ub}\wedge t\rb } \\
&\qquad \times  g\left(\zeta_2(x,y,r),x,y,y-x^{1/2}v_\circ\right).
\end{align*}

Let's first write the bound
\begin{equation*} \left|\Phi^{(1)}_g(t,x,y,r,v_\circ)\right|
\le \KK[E:KandhatKConst]\KK[E:sideintegrablekernel]\|g\| r^{1/2} \sidekernelboundintegrable_{y-r^{1/2}v_\circ}(r,x,y)\lb \bOne_{\{0<r\le  \sfrac{x}{2\KK[E:boundedness]}\}}+\bOne_{\{r\ge  \sfrac{2x}{\ub}\}}\rb. \end{equation*}
From the comments of Remark \ref{R:transformationthree}, we can use dominated convergence on the integral of $\Phi^{(1)}_g$.  For almost every $r$ and $v_\circ$, $(t,x,y)\mapsto \Phi^{(1)}_g(t,x,y,r,v_\circ)$ is continuous; thus
\begin{align*} \lim_{\substack{(t,x,y)\to (t^*,0,y^*)\\ (t,x,y)\in \dom^\circ}}\int_{r=0}^{\TT} \int_{v_\circ\in \R} \Phi^{(1)}_g(t,x,y,r,v_\circ) dv_\circ \, dr
&=\int_{r=0}^{\TT}\int_{v_\circ\in \R} \Phi^{(1)}_g(t^*,x^*,y^*,r,v_\circ) dv_\circ \, dr\\
& = \int_{s=0}^{t^*} \int_{y_\circ\in \R} \sidekernel_{y_\circ}(s,0,y)g(s,0,y,y_\circ)ds\, dy_\circ.
\end{align*}

Let's next write
\begin{align*} \left|\Phi^{(2)}_g(t,x,y,r,v_\circ)\right|
&\le \KK[E:KandhatKConst]\KK[E:sideintegrablekernel]\|g\|\frac{\hypo\left(0,y-x^{1/2}v_\circ\right)}{\left|b_1\left(0,y-x^{1/2}v_\circ\right)\right|}x^2\sidekernelboundintegrable_{y-r^{1/2}v_\circ}(\zeta_2(x,y,r,v_\circ),x,y)\\
&\qquad \times \bOne_{\{\sfrac{x}{2\KK[E:boundedness]}< \zeta_2(x,y,r,v_\circ)< \sfrac{2x}{\ub}\wedge \TT \}}.\end{align*}
From the comments of Remark \ref{R:transformationfour}, we can use dominated convergence on the integral of $\Phi^{(2)}_g$.  

Formalizing the second representation of \eqref{E:zeta2def}, let's now define
\begin{equation*} \tilde \zeta_2(x,y,r,v_\circ)\Def \frac{ 1-\hypo\left(0,y-x^{1/2}v_\circ\right)x^{1/2}r}{\left|b_1\left(0,y-x^{1/2}v_\circ\right)\right|}.\end{equation*}
For every $r$ and $v_\circ$, 
\begin{equation}\label{E:zeta2asymp} \lim_{\substack{(t,x,y)\to (t^*,0,y^*)\\ (t,x,y)\in \dom^\circ}}\tilde \zeta_2(x,y,r,v_\circ)=\frac{1}{|b_1(0,y)|}. \end{equation}

Fix $r$ and $v_\circ$ in $\R$.  
Then
\begin{equation*}
    \bOne_{\lb \frac{x}{2\KK[E:boundedness]}< \zeta_2(x,y,r,v_\circ)< \frac{2x}{\ub}\wedge t\rb } = \bOne_{\lb \frac{x}{2\KK[E:boundedness]}< \zeta_2(x,y,r,v_\circ)< \frac{2x}{\ub}\rb }\bOne_{\{\zeta_2(x,y,r,v_\circ)< t\}}.
\end{equation*}
We have that
\begin{equation*}
\bOne_{\{\zeta_2(x,y,r,v_\circ)< t\}}
= \bOne_{\{\tilde \zeta_2(x,y,r,v_\circ)< t/x\}}
\end{equation*}
and \eqref{E:zeta2asymp} thus implies that
\begin{equation*}  \lim_{\substack{(t,x,y)\to (t^*,0,y^*)\\ (t,x,y)\in \dom^\circ}}\bOne_{\{0< \zeta_2(x,y,r,v_\circ)< t\}}=1. \end{equation*}
By rearrangement
\begin{align*}
\bOne_{\lb \frac{x}{2\KK[E:boundedness]}< \zeta_2(x,y,r,v_\circ)< \frac{2x}{\ub}\rb }
& = \bOne_{\lb \frac{1}{2\KK[E:boundedness]}< \tilde \zeta_2(x,y,r,v_\circ)< \frac{2}{\ub}\rb }\\
& = \bOne_{\lb -\frac{2}{\ub}< \frac{\hypo\left(0,y-x^{1/2}v_\circ\right)x^{1/2}r-1}{\left|b_1\left(0,y-x^{1/2}v_\circ\right)\right|}< -\frac{1}{2\KK[E:boundedness]}\rb }\\
&= \bOne_{\lb -\frac{2\left|b_1\left(0,y-x^{1/2}v_\circ\right)\right|}{\ub}< \hypo\left(0,y-x^{1/2}v_\circ\right)x^{1/2}r-1< -\frac{\left|b_1\left(0,y-x^{1/2}v_\circ\right)\right|}{2\KK[E:boundedness]}\rb }\\
&= \bOne_{\lb -2\frac{\left|b_1\left(0,y-x^{1/2}v_\circ\right)\right|}{\ub}+1< \hypo\left(0,y-x^{1/2}v_\circ\right)x^{1/2}r< 1-\frac12 \frac{\left|b_1\left(0,y-x^{1/2}v_\circ\right)\right|}{\KK[E:boundedness]}\rb }.
\end{align*}
By Assumptions \ref{A:boundedness} and \ref{A:basics}, we have that
\begin{equation*}  \frac{\left|b_1\left(0,y-x^{1/2}v_\circ\right)\right|}{\KK[E:boundedness]}\le 1  \qquad \text{and}\qquad
\frac{\left|b_1\left(0,y-x^{1/2}v_\circ\right)\right|}{\ub}
\ge 1\end{equation*}
so
\begin{equation*} \bOne_{\lb \frac{x}{2\KK[E:boundedness]}<\zeta_2(x,y,r,v_\circ)< \frac{2x}{\ub}\rb }\ge \bOne_{\lb -1< \hypo\left(0,y-x^{1/2}v_\circ\right)x^{1/2}r< \frac12 \rb } \end{equation*}
and thus
\begin{equation*} 1\ge \varlimsup_{\substack{(t,x,y)\to (t^*,0,y^*)\\ (t,x,y)\in \dom^\circ}}\bOne_{\lb \frac{x}{2\KK[E:boundedness]}< \zeta_2(x,y,r,v_\circ)< \frac{2x}{\ub}\rb }\ge \varliminf_{\substack{(t,x,y)\to (t^*,0,y^*)\\ (t,x,y)\in \dom^\circ}}\bOne_{\lb -1< \hypo\left(0,y-x^{1/2}v_\circ\right)x^{1/2}r< \frac12 \rb } =1. \end{equation*}
We also have that
\begin{align*}
    &\frac{\hypo\left(0,y-x^{1/2}v_\circ\right)}{\left|b_1\left(0,y-x^{1/2}v_\circ\right)\right|}x^2\sidekernel_{y-x^{1/2}v_\circ}\left(\zeta_1(x,y,r,v_\circ),x,y\right)\\
&\qquad = \left(\frac{x}{\zeta_2(x,y,r,v_\circ)}\right)^2 \frac{\sqrt{12}}{2\pi}\exp\left[-6\XX^2_{0,y-x^{1/2}v_\circ}\left(\zeta_2(x,y,r,v_\circ),x\right)\right.\\
&\qquad \qquad \left. -6\XX_{0,y-x^{1/2}v_\circ}\left(\zeta_2(x,y,r,v_\circ),x\right)\YY_{0,y-x^{1/2}v_\circ}\left(\zeta_2(x,y,r,v_\circ),y\right)-2\YY^2_{0,y-x^{1/2}v_\circ}\left(\zeta_2(x,y,r,v_\circ),y\right)\right].
\end{align*}
Using the second representation of \eqref{E:zeta2def}, we have
\begin{align*}
\frac{x}{\zeta_2(x,y,r,v_\circ)}
&=\frac{1}{\tilde \zeta_2(x,y,r,v_\circ)}\\
\XX_{0,y-x^{1/2}v_\circ}\left(\zeta_2(x,y,r,v_\circ),x\right)
&= \frac{x+b_1(0,y-x^{1/2}v_\circ)\zeta_2(x,y,r,v_\circ)}{\hypo(0,y-x^{1/2}v_\circ) \zeta^{3/2}_2(x,y,r,v_\circ)} -\frac12b_2\left(0,y-x^{1/2}v_\circ\right)x^{1/2}\tilde \zeta^{1/2}_2(x,y,r,v_\circ)\\
&= \frac{x-\left|b_1(0,y-x^{1/2}v_\circ)\right|\lb \frac{x-\hypo\left(0,y-x^{1/2}v_\circ\right)x^{3/2}r}{\left|b_1\left(0,y-x^{1/2}v_\circ\right)\right|}\rb}{\hypo(0,y-x^{1/2}v_\circ)x^{3/2}\tilde \zeta^{3/2}_2(x,y,r,v_\circ)} \\
&\qquad -\frac12b_2\left(0,y-x^{1/2}v_\circ\right)x^{1/2}\tilde \zeta^{1/2}_2(x,y,r,v_\circ)\\
&= \frac{r}{\tilde \zeta_2^{3/2}(x,y,r,v_\circ)}-\frac12b_2\left(0,y-x^{1/2}v_\circ\right)x^{1/2}\tilde \zeta^{1/2}_2(x,y,r,v_\circ)\\
\YY_{0,y-x^{1/2}v_\circ}\left(\zeta_2(x,y,r,v_\circ),y\right)
& =\frac{x^{1/2}v_\circ}{\zeta^{1/2}_2(x,y,r,v_\circ)}+b_2\left(0,y-x^{1/2}v_\circ\right)\zeta^{1/2}_2(x,y,r,v_\circ)\\
&=\frac{v_\circ}{\tilde \zeta^{1/2}_2(x,y,r,v_\circ)} +b_2\left(0,y-x^{1/2}v_\circ\right)x^{1/2}\tilde \zeta^{1/2}_2(x,y,r,v_\circ).
\end{align*}
Taking limits and using \eqref{E:zeta2asymp},
\begin{align*} \lim_{\substack{(t,x,y)\to (t^*,0,y^*)\\ (t,x,y)\in \dom^\circ}}\frac{x}{\zeta_2(x,y,r,v_\circ)}&=|b_1(0,y^*)| \\
\lim_{\substack{(t,x,y)\to (t^*,0,y^*)\\ (t,x,y)\in \dom^\circ}}
\XX_{0,y-x^{1/2}v_\circ}\left(\zeta_2(x,y,r,v_\circ),x\right)&=r|b_1(0,y^*)|^{3/2}\\
\lim_{\substack{(t,x,y)\to (t^*,0,y^*)\\ (t,x,y)\in \dom^\circ}}
\YY_{0,y-x^{1/2}v_\circ}\left(\zeta_2(x,y,r,v_\circ),y\right)&=v_\circ |b_1(0,y^*)|^{1/2} \end{align*}
and, combining things,
\begin{align*}&\lim_{\substack{(t,x,y)\to (t^*,0,y^*)\\ (t,x,y)\in \dom^\circ}}
\Phi^{(2)}_g(x,y,r,v_\circ)\\
&\qquad = |b_1(0,y^*)|^2\frac{\sqrt{12}}{2\pi}\exp\left[-6\left(r|b_1(0,y^*)|^{3/2}\right)^2\right.\\
&\qquad \qquad \left.- 6\left(r|b_1(0,y^*)|^{3/2}\right)\left(v_\circ |b_1(0,y^*)|^{1/2}\right)^2 -2 \left(v_\circ |b_1(0,y^*)|^{1/2}\right)^2\right]g(0,0,y^*,y^*). 
\end{align*}
Similarly to \eqref{E:unitintegral},
\begin{align*}&\int_{(r,v_\circ)\in \R^2}|b_1(0,y^*)|^2\frac{\sqrt{12}}{2\pi}\exp\left[-6\left(r|b_1(0,y^*)|^{3/2}\right)^2\right.\\
&\qquad \left. - 6\left(r|b_1(0,y^*)|^{3/2}\right)\left(v_\circ |b_1(0,y^*)|^{1/2}\right) -2 \left(v_\circ |b_1(0,y^*)|^{1/2}\right)^2\right]dr\, dv_\circ \\
&\qquad = \int_{(r,v_\circ)\in \RR}\frac{\sqrt{12}}{2\pi}\exp\left[-\frac12 \begin{pmatrix} r & v_\circ\end{pmatrix}\begin{pmatrix} 12 & 6 \\ 6 & 4 \end{pmatrix}\begin{pmatrix} r \\ v_\circ \end{pmatrix}\right] dr\, dv_\circ=1.
\end{align*}
Dominated convergence now that
\begin{equation*} \lim_{\substack{(t,x,y)\to (t^*,0,y^*)\\ (t,x,y)\in \dom^\circ}}\int_{r\in \R}\int_{v_\circ\in \R} \Phi^{(2)}_g(t,x,y,r,v_\circ) dr\, dv_\circ=g(0,0,y^*,y^*). \end{equation*}
Combining our arguments, the claim follows.
\end{proof}

\begin{remark}\label{R:McKean}
The kernel \eqref{E:sidekernelDef} and its space-time integral \textup{(}as in Proposition \ref{P:jumpboundary}\textup{)} is closely connected to the work of \cite{mckean1962winding} \textup{(}see also \cite{masoliver1995exact,masoliver1996exact}\textup{)}.  In \cite{mckean1962winding}, McKean constructs an implicit formula \textup{(}involving Bessel functions\textup{)} for the distribution of the time $\tau^{(0,1)}$ and  place $Y^{(0,1)}_{\tau^{(0,1)}}$ at which the process  \eqref{E:KolmogorovSDE} \textup{(}starting from $(X_0,Y_0)=(0,1)$\textup{)} hits $x=0$ \textup{(}see also \cite{lefebvre1989first,lefebvre1989aap}\textup{)}.
 Informally, Proposition \ref{P:jumpboundary} says that the integral \eqref{E:jumpboundary} against $\sidekernel$ tends to the sum of a Dirac measure and an integral against the measure suggested by the Green function \textup{(}expression \textup{(}2\textup{)}\textup{)} of \cite{mckean1962winding}.

Our work, and Proposition \ref{P:jumpboundary}, focusses in a sense on the boundary layer of \eqref{E:sidekernelDef} near $x=0$.   The work of \cite{mckean1962winding} in a sense focusses more on the behavior at the boundary.  Further exploration of the connection between our PDE analysis and the probabilistic analysis of \cite{mckean1962winding} might be of independent interest.  In this regard, we note that \cite{mckean1962winding} did not give an explicit description of the statistics of $\tau^{(0,1)}$. 
\end{remark}

\section{Corrections}\label{S:Corrections}
Let's try to improve upon $\kernel$ and $\sidekernel$ of \eqref{E:kernelDef} and \eqref{E:sidekernelDef} by considering a \emph{corrected} kernel
\begin{equation}\label{E:kernelQDef} \kernelQ_{\SP}(t,x,y)= \kernel_{\SP}(t,x,y)\lb 1+\QCorrector_{\SP}(t,x,y)\rb \qquad (t,x,y)\in \openstrip\end{equation}
where $\QCorrector_{\SP}$ is to be determined such that $\genP \kernelQ_{\SP}$ is sufficiently \lq\lq small" on $\openstrip$.  More specifically, we will need good bounds, given in Theorem \ref{T:collected}, on $\genP \kernelQ$ for the Volterra recursion of Section \ref{S:Volterra}.  The function $\QCorrector_{\SP}$ will be given in \eqref{E:zwounds}-\eqref{E:zwoundsprime} as the solution of the linear equation representing the projection of $\genP \kernelQ\equiv 0$ onto a carefully-constructed low-dimensional (namely, 6-dimensional) vector space of space-time polynomials.

A number of techniques for asymptotically expanding fundamental solutions of heat equations (in small time) have been developed.  These can proceed from semigroup theory \cite{vassilevich2003heat} or from probability theory.  In the case of probability theory, one can consider expansions around Varadhan-type \cite{varadhan1967behavior} small-time asymptotics of a heat kernel \cite{arous1989developpement,azencott1984densite,ben1988developpement,franchi2014small,Leandre19924576}
or from stochastic flows (or their analytic equivalent) and parametrices \cite{bally2015probabilistic,baudoin2012stochastic,MR2802040,taylor1986parametrix}.  Invariably, a fair amount of bookkeeping is involved to keep track of degrees in the expansion \cite{chen1957integration}.  This becomes particularly challenging when noise and smoothness propagate through the H\"ormander Lie brackets (see references in \cite{baudoin2012stochastic}).  Here, our goal is to correct $\kernel$ and $\kernelQ$ so that we can bound a Volterra-type recursion for a variation of parameters formula.  More precisely, we seek bounds similar to the kernel of \eqref{E:kernelboundDef} (and recall in particular the complexities of the upper bounds of \eqref{E:kernelboundintegrableDef} and \eqref{E:sidekernelboundintegrableDef}). To do so, we shall simply project the PDE for \eqref{E:kernelQDef} onto a space of polynomials.

Using Proposition \ref{P:KEvol}, we can understand $\genP \kernel_{\SP}$ of \eqref{E:linearizedgenerator}; for $(\SP)\in \RR$,
\begin{align*} \genP \kernel_{\SP}(t,x,y)&=\genP^L_{\SP}\kernel_{\SP}(t,x,y)+\lb b_1(x,y)-b_{1,\SP}^L(x,y)\rb \frac{\partial \kernel_{\SP}}{\partial x}(t,x,y)\\
&\qquad +\lb b_2(x,y)-b_2(\SP)\rb \frac{\partial \kernel_{\SP}}{\partial y}(t,x,y) + c(x,y)\kernel_{\SP}(t,x,y)\\
&= \Xi_{\SP}(t,x,y)\kernel_{\SP}(t,x,y)\end{align*}
for $(t,x,y)\in \openstrip$, where we have defined
\begin{equation*} \Xi_{\SP}(t,x,y) = -\lb b_1(x,y)-b^L_{1,\SP}(x,y)\rb \frac{\partial \cE_{\SP}}{\partial x}(x,y) - \lb b_2(x,y)-b_2(\SP)\rb \frac{\partial \cE_{\SP}}{\partial y}(x,y)+c(x,y)\end{equation*}
for $(t,x,y)\in \openstrip$.  Thus
\begin{equation} \label{E:ErrorPDE}\begin{aligned} \genP \kernelQ_{\SP} &= \left(\genP\kernel_{\SP}\right)\lb 1+Q_{\SP}\rb + \kernel_{\SP}\left(\genP\lb 1+Q_{\SP}\rb\right)-c\kernel_{\SP}\lb 1+Q_{\SP}\rb \\
&\qquad + \left(\tfrac{\partial}{\partial y} \kernel_{\SP}\right)\left(\tfrac{\partial}{\partial y}\lb 1+Q_{\SP}\rb\right) \\
&= \kernel_{\SP}\Xi_{\SP} \lb 1+Q_{\SP}\rb +\kernel_{\SP}\lb c+\genP \QCorrector_{\SP}\rb - c\kernel_{\SP}\lb 1+Q_{\SP}\rb \\
&\qquad - \kernel_{\SP}\left(\tfrac{\partial \cE_{\SP}}{\partial y}\right)\left(\tfrac{\partial \QCorrector_{\SP}}{\partial y}\right) \\
&= \kernel_{\SP}\lb \tgenP_{\SP} \QCorrector_{\SP}+\Xi_{\SP} \rb \end{aligned}\end{equation}
on $\openstrip$, where, if $g$ is a real-valued function which is defined and twice-differentiable on some open subset $\openstrip'$ of $\openstrip$,
\begin{equation} \label{E:tgenPDef}\begin{aligned} \left(\tgenP_{\SP} g\right)(t,x,y)&\Def \left(\genP g\right)(t,x,y) - c(x,y)g(t,x,y)-\left(\frac{\partial \cE_{\SP}}{\partial y}\right)(t,x,y)\frac{\partial g}{\partial y}(t,x,y)\\
&\qquad +\Xi_{\SP}(t,x,y)g(t,x,y)\end{aligned}\end{equation}
for $(t,x,y)\in \openstrip'$.

In the standard diffusive case, the exponential part of the fundamental solution (analogous to \eqref{E:cEDef}) is given by the square of a distance divided by time.  Reversing this and keeping \eqref{E:XXYYDef} in mind, let's rewrite spatial dependence in terms of $\xx_{\SP}$ and $\yy_{\SP}$ of \eqref{E:xxyydef};
\begin{equation} \label{E:reversedcoordinates} x-x_\circ = \hypo(\SP)\xx_{\SP}(t,x)t^{3/2}-b_1(\SP)t \qquad \text{and}\qquad 
y-y_\circ= \yy_{\SP}(t,y)t^{1/2} \end{equation}
for $(\SP)\in \RR$ and $(t,x,y)\in \openstrip$.  We want to expand
$\Xi_{\SP}$ as a polynomial in $\xx_{\SP}$, $\yy_{\SP}$ and $t$, and then find a similar polynomial expansion for a solution of 
\begin{equation} \label{E:genPgoal} \tgenP_{\SP} \QCorrector_{\SP}\approx -\Xi_{\SP}\end{equation} 
on $\openstrip$.  Our main result about the size of this error will be Proposition \ref{P:CorrectorWhy}.

\subsection{Setup}
Let's start to construct a space of polynomials.  
With $\N$ as in \eqref{E:NDef}, define the index set
\begin{equation*} \Index\Def \N\times \N\times (\Z/2) \end{equation*}
and for $(p,q,r)\in \Index$, define
\begin{equation*} \Mono{p}{q}{r}(t,x,y)\Def \left(\frac{x-x_\circ+b_1(\SP)t}{\hypo(\SP)}\right)^p(y-y_\circ)^q t^r. \qquad (t,x,y)\in \openstrip \end{equation*}
Then
\begin{equation} \label{E:monoequiv} \begin{aligned} \Mono{p}{q}{r}(t,x,y)&= \xx^p_{\SP}(t,x)\yy^q_{\SP}(t,y)t^{r+(3p+q)/2} \\
\xx^p_{\SP}(t,x)\yy^q_{\SP}(t,y)t^r&=\Mono{p}{q}{r-(3p+q)/2}(t,x,y) \end{aligned}\end{equation}
for $(p,q,r)\in \Index$ and $(t,x,y)\in \openstrip$.
The set
\begin{equation*} \Monoid_{\SP}\Def \lb \Mono{p}{q}{r}: (p,q,r)\in \Index\rb \subset C^\infty(\openstrip), \end{equation*}
is a \emph{monoid} under multiplication in $C^\infty(\openstrip)$; 
$\Mono{p_1}{q_1}{r_1}\cdot \Mono{p_2}{q_2}{r_2}=\Mono{p_1+p_2}{q_1+q_2}{r_1+r_2}$
for $(p_1,q_1,r_1)$ and $(p_2,q_2,r_2)$ in $\Index$ (and with $\Mono{0}{0}{0}$ being the identity element).  Define now the algebra
\begin{equation*} \Algebra\Def \lb \sum_{(p,q,r)\in \Index'}c_{p,q,r}\Mono{p}{q}{r}: \Index'\subset \Index, |\Index'|<\infty, \{c_{p,q,r}\}_{(p,q,r)\in \Index'}\subset \R\rb \subset C^\infty(\openstrip) \end{equation*}
which corresponds to finite polynomials in $\xx_{\SP}(t,x)$ and $\yy_{\SP}(t,y)$ and a finite Laurent series in $t^{1/2}$.  Then (recalling \eqref{E:spatialderivatives})
\begin{align*} x-x_\circ &= \hypo(\SP)\Mono{1}{0}{0}(t,x,y)-b_1(\SP)\Mono{0}{0}{1}(t,x,y) \\
y-y_\circ &= \Mono{0}{1}{0}(t,x,y) \\
t&= \Mono{0}{0}{1}(t,x,y) \\
\frac{\partial \cE_{\SP}}{\partial x}(t,x,y)
&= \frac{1}{\hypo(\SP)}t^{-3/2}\lb 12\xx_{\SP}(t,x)+6\yy_{\SP}(t,y)\rb \\
    &= \frac{1}{\hypo(\SP)}\lb 12\Mono{1}{0}{-3}(t,x,y)+6\Mono{0}{1}{-2}(t,x,y)\rb \\
\frac{\partial \cE_{\SP}}{\partial y}(t,x,y) 
    &= t^{-1/2}\lb 6\xx_{\SP}(t,x)+4\yy_{\SP}(t,y)+b_2(\SP)t^{1/2}\rb \\
    &= \lb 6\Mono{1}{0}{-2}(t,x,y)+4\Mono{0}{1}{-1}(t,x,y)+b_2(\SP)\Mono{0}{0}{1/2}(t,x,y)\rb\\
    b_{1,\SP}^L(x,y)&= b_1(\SP)+\frac{\partial b_1}{\partial y}(\SP)\Mono{0}{1}{0}(t,x,y) =b_1(\SP)+\hypo(\SP)\Mono{0}{1}{0}(t,x,y)  
    \end{align*}
for $(t,x,y)\in \openstrip$.

We want to approximate \eqref{E:genPgoal} with an equation on $\Algebra$.  For $g\in C^\infty(\RR)$ and a nonnegative integer $d$, define
\begin{align*} \left(\Poly^{(d)}_{\SP}g\right)(t,x,y)
&\Def \sum_{\substack{(i,j)\in \N^2 \\ i+j\le d}}\frac{1}{i!j!}\frac{\partial^{i+j}g}{\partial x^i \partial y^j}(\SP)\left(x-x_\circ\right)^i\left(y-y_\circ\right)^j\\
&=\sum_{\substack{(i,j)\in \N^2 \\ i+j\le d}}\frac{1}{i!j!}\frac{\partial^{i+j}g}{\partial x^i \partial y^j}(\SP)\left(\hypo(\SP)\Mono{1}{0}{0}(t,x,y)-b_1(\SP)\Mono{0}{0}{1}(t,x,y)\right)^i\\
&\qquad \qquad \times \left(\Mono{0}{1}{0}(t,x,y)\right)^j
\end{align*}
for $(\SP)\in \RR$ and $(t,x,y)\in \openstrip$. Then $\Poly^{(d)}_{\SP}g\in \Algebra$ and
\begin{equation*} g(x,y)=\left(\Poly^{(d)}_{\SP}g\right)(t,x,y)+\left(\Res^{(d+1)}_{\SP}g\right)(x,y) \end{equation*}
for $(t,x,y)\in \openstrip$, where $\Res^{(d+1)}_{\SP}g$ is as in \eqref{E:TaylorRemainder}.

Let's next define a \emph{degree} map which will help us organize our calculations.  Setting
\begin{equation}\label{E:degdef} \degspace \Def \N \times (\Z/2), \end{equation}
let's define $\degree:\Monoid_{\SP}\to \degspace$ as
\begin{equation*} \degree\left(\Mono{p}{q}{r}\right) \Def \left(p+q,r+\tfrac12(3p+q)\right) \end{equation*}
for $\Mono{p}{q}{r}\in \Monoid_{\SP}$.
We then have that
\begin{equation} \label{E:degreeequality}\begin{aligned} \degree\left(\Mono{p_1}{q_1}{r_1}\Mono{p_2}{q_2}{r_2}\right) &= \degree\left(\Mono{p_1+p_2}{q_1+q_2}{r_1+r_2}\right)\\
&= \left(p_1+p_2+q_1+q_2,(r_1+r_2)+\tfrac12 \left(3(p_1+p_2)+(q_1+q_2)\right)\right)\\
&=\left(p_1+q_1,r_1+\tfrac12(3p_1+q_1)\right)+\left(p_2+q_2,r_2+\tfrac12(3p_2+q_2)\right)\\
&=\degree\left(\Mono{p_1}{q_1}{r_1}\right)+\degree\left(\Mono{p_2}{q_2}{r_2}\right) \end{aligned}\end{equation}
for any $(p_1,q_1,r_1)$ and $(p_2,q_2,r_2)$ in $\Index$, implying that $\degree$ does act as a degree map.  
Recalling \eqref{E:monoequiv}, $\degree$ maps $\xx_{\SP}^p \yy_{\SP}^q t^r$ to
\begin{equation*} (p+q,\left(r-(3p+q)/2+(3p+q)/2)=(p+q,r\right). \end{equation*}
The first component of $\degree$ keeps track of the spatial regularity, and the second component keeps track of the order as $t\searrow 0$, if we consider an asymptotic regime where $\xx_{\SP}(t,x)$ and $\yy_{\SP}(t,y)$ are both of order 1.

Let's rewrite $\Algebra$ as a graded vector space; we have that
\begin{equation*} \Algebra= \bigoplus_{(d,s)\in \degspace}\Vspace^{d,s} \end{equation*}
where, for each for each $(d,s)\in \N\times (\Z/2)$,
\begin{equation}\label{E:vectorspaces}\Vspace^{d,s}\Def \lb \sum_{\substack{(p,q,r)\in \Index \\ \degree(p,q,r)=(d,s)}}c_{p,q,r}\Mono{p}{q}{r}: \{c_{p,q,r}\}_{\substack{(p,q,r)\in \Index \\ \degree(p,q,r)=(d,s)}}\subset \R\rb = \lb \sum_{p=0}^d c_p \Mono{p}{d-p}{s-p-d/2}: \{c_p\}_{p=0}^d\subset \R\rb 
\end{equation}
is the vector space of linear combinations of elements of $\Monoid_{\SP}$ of degree $(d,s)$.  We note that $\Vspace^{d,s}$ has dimension $d+1$.  We also interpret products of $\Vspace^{d,s}$'s as subsets of $\Algebra$ (i.e., finite linear combinations of elements of $\Mono{p}{q}{r}$'s).  In particular,
$\Vspace^{d_1,s_1}\cdot \Vspace^{d_2,s_2}\subset \Vspace^{d_1+d_2,s_1+s_2}$
for $(d_1,s_1)$ and $(d_2,s_2)$ in $\Index$; the product of an element of $\Vspace^{d_1,s_1}$ and an element of $\Vspace^{d_2,s_2}$ is an element of $\Vspace^{d_1+d_2,s_1+s_2}$ (recall \eqref{E:degreeequality}).

Let's define \emph{main} $\Proj^M$ and \emph{error} $\Proj^E$ projection operators as the projections of $\Algebra$, respectively, onto 
\begin{equation*} \Vspace^M\Def \bigoplus_{\substack{(d,s)\in \degspace\\ s\le \sfrac12}}\Vspace^{d,s} \qquad \text{and}\qquad \Vspace^E\Def \bigoplus_{\substack{(d,s)\in \degspace\\ s\ge 0}}\Vspace^{d,s}.\end{equation*}
We want to project \eqref{E:genPgoal} onto an equation on $\Vspace^M$.  This will allow us to control the small-time error of \eqref{E:genPgoal}.

\subsection{Solution and Error Bounds}
Let's start with the right-hand side of \eqref{E:genPgoal}.  Let's write
\begin{equation}\label{E:XiAsSum}\Xi_{\SP} =\Xi^L_{\SP}+ \Xi^{NL}_{\SP}\end{equation}
on $\openstrip$, where
\begin{equation} \label{E:XiPartsDef} \begin{aligned}
\Xi^L_{\SP}(t,x,y)&= -\frac{ \left(\Poly_{\SP}^{(2)}b_1\right)(t,x,y)-b^L_{1,\SP}(x,y)}{\hypo(\SP)} \frac{\partial \cE_{\SP}}{\partial x}(t,x,y) \\
&=-\sum_{\substack{(i,j)\in \N^2 \\ i+j\le 2 \\ (i,j)\not\in \{(0,0),(0,1)\}}}\frac{1}{i!j!}\bnorm_{i,j}(\SP)\left(\hypo(\SP)\Mono{1}{0}{0}(t,x,y)-b_1(\SP)\Mono{0}{0}{1}(t,x,y)\right)^i\\
&\qquad \qquad \times \left(\Mono{0}{1}{0}(t,x,y)\right)^j\lb 12\Mono{1}{0}{-3}(t,x,y)+6\Mono{0}{1}{-2}(t,x,y)\rb\\
\Xi^{NL}_{\SP}(t,x,y)
&= -\left(\Res^{(3)}_{\SP}b_1\right)(x,y) \frac{\partial \cE_{\SP}}{\partial x}(t,x,y) - \left(\Res^{(1)}_{\SP}b_2\right)(t,x,y)\frac{\partial \cE_{\SP}}{\partial y}(t,x,y)+c(x,y)
\end{aligned}\end{equation}
for $(t,x,y)\in \openstrip$.
Explicitly, $\Xi^L_{\SP}\in \Algebra$, and 
\begin{equation*} \Xi^L_{\SP}=\sum_{(d,s)\in D_\Xi}\Xi^{*,d,s}_{\SP} \end{equation*}
on $\openstrip$, where
\begin{equation*} D_\Xi \Def \lb (1,- \sfrac{1}{2}),(3,- \sfrac{1}{2}),(2,0),(1,\sfrac{1}{2}),(3,\sfrac{1}{2}),(2,1),(3,\sfrac{3}{2})\rb\end{equation*}
and $\Xi^{*,d,s}_{\SP}\in \Vspace^{d,s}$ for $(d,s)\in D_\Xi$.    
Projecting on $\Vspace^M$, we have 
\begin{equation}\label{E:XiMainError}
\Proj^M\Xi^M_{\SP} =\Xi^{*,1,- \sfrac{1}{2}}_{\SP}+\Xi^{*,3,- \sfrac{1}{2}}_{\SP},
\end{equation}
where
\begin{align*}
\Xi^{*,3,- \sfrac{1}{2}}_{\SP}(t,x,y)&=-6\bnorm_{0,2}(\SP)\Mono{1}{2}{-3}(t,x,y)-3\bnorm_{0,2}(\SP)\Mono{0}{3}{-2}(t,x,y)\\
\Xi^{*,1,- \sfrac{1}{2}}_{\SP}(t,x,y)&= 12\bnorm_{1,0}(\SP)b_1(\SP)\Mono{1}{0}{-2}(t,x,y)+6\bnorm_{1,0}(\SP)b_1(\SP)\Mono{0}{1}{-1}(t,x,y)
\end{align*}
for $(t,x,y)\in \openstrip$.

Let's bound the errors in approximating $\Xi_{\SP}$ by $\Proj^M\Xi^L_{\SP}$.  We are interested in small-time ($t\searrow 0$) asymptotics comparable to powers of $t$, and we will bound spatial behavior by a term which can be dominated by the exponential decay of \eqref{E:kernelboundDef}. To formalize this, let's fix
\begin{equation*} \nu\Def \sfrac{1}{36}; \end{equation*}
then 
\begin{equation}\label{E:nuselectwhy} \sfrac16-3\nu=\sfrac{1}{12}. \end{equation}

\begin{lemma}\label{L:XiLEbound}
We have that
\begin{equation} \label{E:XiLEbound} \KK[E:XiLEbound] \Def \sup\lb \left|\Proj^E\Xi^L_{\SP}(t,x,y)\right|\exp\left[-\nu \xx_{\SP}^2(t,x) - \nu \yy_{\SP}^2(t,y)\right]: (\SP)\in \RR,\,  (t,x,y)\in \openstrip\rb \end{equation}
is finite.
\end{lemma}
\begin{proof}
For each $(\SP)\in \RR$, define
\begin{equation}\label{E:Pvecdef} P(\SP)\Def \left(\bnorm_{1,0}(\SP),\bnorm_{1,1}(\SP),\bnorm_{2,0}(\SP),\bnorm_{0,2}(\SP),\hypo(\SP),b_1(\SP)\right)\in \R^6.\end{equation}
Then we can write
\begin{align*} \Proj^E\Xi^L_{\SP} &= \sum_{\substack{(d,s)\in D_\Xi \\ (d,s)\not \in \{(1,-\sfrac12),(3,-\sfrac12)\}\\
0\le p\le d}}c_{p,d-p,s-p-d/2}\left(P(\SP)\right)\Mono{p}{d-p}{s-p-d/2} 
\end{align*}
where the $c_{p,q,r}$'s are continuous.  Bounding the various terms in \eqref{E:Pvecdef}, we have that
\begin{equation*}
\lb P(\SP): \SP\in \RR\rb \subset R\Def \left[-\KK[E:hypobound],\KK[E:hypobound]\right]^4\times \left[-\KK[E:boundedness],\KK[E:boundedness]\right]^2, \end{equation*}
Since $R$ is a compact subset of $\R^6$ and the $c_{p,q,r}$'s are continuous,
\begin{equation*}\kk\Def \sup\lb \left|c_{p,q,s-p-d/2}(\rho)\right|: (d,s)\in D_\Xi,\, (d,s)\not \in \{(1,-\sfrac12),(3,-\sfrac12)\},\, 0\le p\le d,\, \rho\in R\rb \end{equation*}
is finite.  Thus
\begin{equation*} \left|\Proj^E\Xi^L_{\SP}(t,x,y)\right|\le \kk \Gamma\left(t,\xx_{\SP}(t,x),\yy_{\SP}(t,y)\right) \end{equation*}
where
\begin{equation*}\Gamma(t,X,Y)
\Def \sum_{\substack{(d,s)\in D_\Xi \\ (d,s)\not \in \{(1,-\sfrac12),(3,-\sfrac12)\}\\
0\le p\le d}}|X|^p|Y|^{d-p}t^s. \end{equation*}
Since
\begin{equation*} \sup_{\substack{ X,Y\in \R \\ 0<t< \TT}}\left|\Gamma(t,X,Y)\right|\exp\left[-\nu X^2 - \nu Y^2 \right]\end{equation*}
is finite, the claim follows.
\end{proof}

To bound $\Xi^{NL}_{\SP}$ of \eqref{E:XiPartsDef}, we will need a refinement of \eqref{E:Tbounds} and \eqref{E:Tboundsb1}.
\begin{lemma}\label{L:TaylorBounds} We have that
\begin{align} \label{E:wheel1} \KK[E:wheel1] &\Def \sup_{\substack{(\SP)\in \RR \\ (t,x,y)\in \openstrip\\ d\in \{1,2,3\}}}\frac{\left|\left(\Res^{(d)}_{\SP}b_1\right)(x,y)\right|}{\hypo(\SP)}t^{-d/2}\exp\left[-\nu \xx_{\SP}^2(t,x) - \nu \yy_{\SP}^2(t,y)\right] \\
\label{E:wheel2} \KK[E:wheel2]&\Def \sup_{\substack{(\SP)\in \RR \\ (t,x,y)\in \openstrip}}\left|\left(\Res^{(1)}_{\SP}b_2\right)(x,y)\right|t^{-1/2}\exp\left[-\nu \xx_{\SP}^2(t,x) - \nu \yy_{\SP}^2(t,y)\right] \end{align}
are finite.
\end{lemma}
\begin{proof}  Fix $(\SP)\in \RR$ and $(t,x,y)$ in $\openstrip$.  Starting with \eqref{E:reversedcoordinates}, let's first calculate that
\begin{align*}
     |x-x_\circ|+|y-y_\circ|
    &\le  \left|\xx_{\SP}(t,x)\right| \left|\tfrac{\partial b_1}{\partial y}(\SP)\right| t^{3/2} + |b_1(\SP)| t+\left|\yy_{\SP}(t,y)\right|t^{1/2}\\
    &\le t^{1/2} \lb \KK[E:boundedness]\left|\xx_{\SP}(t,x)\right| \TT + \left|\yy_{\SP}(t,y)\right|+\KK[E:boundedness]\TT^{1/2}\rb.
    \end{align*}
Using \eqref{E:firstexpbound}, we then have that
    \begin{equation*}
    \lb |x-x_\circ|+|y-y_\circ|\rb^d \le t^{d/2}d!\exp\left[\KK[E:boundedness]\left|\xx_{\SP}(t,x)\right| \TT + \left|\yy_{\SP}(t,y)\right|+\KK[E:boundedness]\TT^{1/2}\right]\end{equation*}
for $d\in \{1,2,3\}$.

We then get from \eqref{E:Tbounds} and \eqref{E:Tboundsb1} that 
\begin{align*}
\left|\left(\Res^{(1)}_{\SP}b_2\right)(x,y)\right|&\le \KK[E:boundedness]t^{1/2}\exp\left[\KK[E:boundedness]\left|\xx_{\SP}(t,x)\right| \TT + \left|\yy_{\SP}(t,y)\right|+\KK[E:boundedness]\TT^{1/2}\right]\\
\left|\frac{\left(\Res^{(d)}_{\SP}b_1\right)(x,y)}{\hypo(\SP)}\right|
&\le \KK[E:hypobound]t^{d/2}\exp\left[\left(1+\KK[E:hypobound]\TT^{1/2}\right)\lb \KK[E:boundedness]\left|\xx_{\SP}(t,x)\right|\TT+\left|\yy_{\SP}(t,y)\right|+ \KK[E:boundedness]\TT^{1/2}\rb\right].
\end{align*}
Using Young's inequality with $\eps$,
\begin{align*}&\left(1+\KK[E:hypobound]\TT^{1/2}\right)\lb \KK[E:boundedness]\left|\xx_{\SP}(t,x)\right|\TT+\left|\yy_{\SP}(t,y)\right|+ \KK[E:boundedness]\TT^{1/2}\rb\\
&\qquad =  \KK[E:boundedness]\left(1+\KK[E:hypobound]\TT^{1/2}\right)\TT\left|\xx_{\SP}(t,x)\right|+\left(1+\KK[E:hypobound]\TT^{1/2}\right)\left|\yy_{\SP}(t,y)\right|+ \KK[E:boundedness]\left(1+\KK[E:hypobound]\TT^{1/2}\right)\TT^{1/2}\\
&\qquad \le \nu \xx_{\SP}^2(t,x)+\nu\yy_{\SP}^2(t,y)+\frac{4}{\nu}\KK[E:boundedness]^2\left(1+\KK[E:hypobound]\TT^{1/2}\right)^2\TT^2+\frac{4}{\nu}\left(1+\KK[E:hypobound]\TT^{1/2}\right)^2\\
&\qquad \qquad +\KK[E:boundedness]\left(1+\KK[E:hypobound]\TT^{1/2}\right)\TT^{1/2}.
\end{align*}
The claim follows.
\end{proof}

\begin{lemma}\label{L:XiNLbound} We have that
\begin{equation} \label{E:XiNLbound}  \KK[E:XiNLbound] \Def \sup_{\substack{(\SP)\in \RR \\ (t,x,y)\in \openstrip}}\left|\Xi^{NL}_{\SP}(t,x,y)\right|\exp\left[-2\nu \xx_{\SP}^2(t,x) - 2\nu \yy_{\SP}^2(t,y)\right]\end{equation}
is finite.
\end{lemma}
\begin{proof}  
Fix $(t,x,y)\in \openstrip$.  Using Lemma \ref{L:TaylorBounds}, 
\begin{equation*} \left|\Xi^{NL}_{\SP}(t,x,y)\right|\le \Gamma\left(t,\xx_{\SP}(t,x),\yy_{\SP}(t,y))\right) \end{equation*}
where
\begin{align*}
\Gamma(t,X,Y)&\Def   \KK[E:wheel1] t^{3/2}\exp\left[\nu X^2+\nu Y^2\right] t^{-3/2}
\lb 12\left|X\right|+6\left|Y\right|\rb \\
&\qquad + \KK[E:wheel2]t^{1/2}\exp\left[\nu X+\nu Y^2\right] t^{-1/2}\lb 6\left|X\right|+4\left|Y\right|+\KK[E:boundedness]t^{1/2} \rb +\KK[E:boundedness]\\
&=\KK[E:wheel1]\exp\left[\nu X^2+\nu Y^2\right] 
\lb 12\left|X\right|+6\left|Y\right|\rb \\
&\qquad + \KK[E:wheel2]\exp\left[\nu X^2+\nu Y^2\right] \lb 6\left|X\right|+4\left|Y\right|+\KK[E:boundedness]t^{1/2} \rb +\KK[E:boundedness].
\end{align*}
Since
\begin{equation*} \sup_{\substack{ X,Y\in \R \\ 0<t< \TT}}\left|\Gamma(t,X,Y)\right|\exp\left[-2\nu X^2-2\nu  Y^2 \right] \end{equation*}
is finite, the claim follows.
\end{proof}
\noindent These asymptotics guided our choice of the degree to which we carried out the Taylor expansion of $b_1$ (2 is the smallest degree of the Taylor expansion of $b_1$ in \eqref{E:XiPartsDef} for which we get the asymptotics of Lemma \ref{L:XiNLbound}).

Let's next look at how $\tgenP_{\SP}$ of \eqref{E:tgenPDef} acts on elements of $\Algebra$.   We start by constructing some partial differential operators particularly suited to the $\Mono{p}{q}{r}$'s.  For $(t,x,y)\in \openstrip$, define
\begin{equation*}  (\VV^T_{\SP} g)(t,x,y)=\frac{\partial g}{\partial t}(t,x,y)-b_1(\SP)\frac{\partial g}{\partial x}(t,x,y) \qquad \text{and}\qquad 
(\VV^X_{\SP} g)(t,x,y)=\hypo(\SP)\frac{\partial g}{\partial x}(t,x,y)\end{equation*}
for $g\in \Algebra$.
For $(p,q,r)\in \Index$, we then have
\begin{gather*} \VV^T_{\SP} \Mono{p}{q}{r}=\begin{cases} r\Mono{p}{q}{r-1}&  \text{if $r>0$} \\
0 &\text{if $r\not =0$}\end{cases} \qquad \text{and}\qquad
\VV^X_{\SP} \Mono{p}{q}{r}=\begin{cases} p\Mono{p-1}{q}{r} &\text{if $p>0$} \\
0 &\text{if $p=0$}\end{cases} \\
\qquad \text{and}\qquad \frac{\partial \Mono{p}{q}{r}}{\partial y} =\begin{cases} q\Mono{p}{q-1}{r}&\text{if $q>0$} \\
0 &\text{if $q=0$}\end{cases}\end{gather*}
on $\openstrip$, which is similar to standard differentiation of monomials.
Defining
\begin{equation}\label{E:zerospaces} \Vspace^{d,s}\Def \{0\}\subset \Algebra \end{equation}
for $d<0$ and $s\in \Z/2$, \eqref{E:degdef} then implies that
\begin{equation}\label{E:differentialinclusionsb} \VV^T_{\SP}:\Vspace^{d,s}\to \Vspace^{d,s-1}, \quad \VV^X_{\SP}:\Vspace^{d,s}\to \Vspace^{d-1,s-\sfrac{3}{2}}, \quad \text{and}\quad 
\frac{\partial}{\partial y}:\Vspace^{d,s}\to \Vspace^{d-1,s-\sfrac{1}{2}} \end{equation}
for $(d,s)\in \degspace$.
For $g\in \Algebra$, we can write
\begin{align*} \left(\tgenP_{\SP} g\right)(t,x,y) &= \frac12\frac{\partial^2 g}{\partial y^2}(t,x,y)+ \frac{b_1(x,y)-b_1(\SP)}{\hypo(\SP)}\left(\hypo(\SP)\frac{\partial g}{\partial x}(t,x,y)\right)\\
&\qquad +\lb b_2(x,y)-\frac{\partial \cE_{\SP}}{\partial y}(t,x,y)\rb \frac{\partial g}{\partial y}(t,x,y)\\
&\qquad -\lb \frac{\partial g}{\partial t}(t,x,y)-b_1(\SP)\frac{\partial g}{\partial x}(t,x,y)\rb
+\Xi_{\SP}(t,x,y)g(t,x,y) \\
&= \frac12\frac{\partial^2 g}{\partial y^2}(t,x,y)+ \frac{b_1(x,y)-b_1(\SP)}{\hypo(\SP)}(\VV^X_{\SP}g)(t,x,y)\\
&\qquad +\lb b_2(x,y)-\frac{\partial \cE_{\SP}}{\partial y}(t,x,y)\rb \frac{\partial g}{\partial y}(t,x,y)\\
&\qquad -\left(\VV^T_{\SP}g\right)(t,x,y)
+\Xi_{\SP}(t,x,y)g(t,x,y) 
\end{align*}
for $g\in \Algebra$ and $(t,x,y)\in \openstrip$.
Iterating the last inclusion of \eqref{E:differentialinclusionsb}, we have that
\begin{equation}\label{E:opa} \frac{\partial^2}{\partial y^2}:\Vspace^{d,s} \to \Vspace^{d-2,s-1}\end{equation}
for $(d,s)\in \degspace$.

Let's carry out a decomposition of $\tgenP_{\SP}$ similar to that of \eqref{E:XiAsSum}.  For $g\in \Algebra$, define
\begin{equation} \label{E:tgenPLstuff}\begin{aligned}
\left(\tgenP_{\SP}^L g\right)(t,x,y)  &= \frac12\frac{\partial^2 g}{\partial y^2}(t,x,y)+ \frac{\left(\Poly^{(1)}_{\SP}b_1\right)(t,x,y)-b_1(\SP)}{\hypo(\SP)}(\VV^X_{\SP}g)(t,x,y)\\
&\qquad -\frac{\partial \cE_{\SP}}{\partial y}(t,x,y) \frac{\partial g}{\partial y}(t,x,y)\\
&\qquad -\left(\VV^T_{\SP}g\right)(t,x,y)
+\left(\Proj^M \Xi^L_{\SP}\right)(t,x,y)g(t,x,y) \\
\left(\tgenP_{\SP}^{NL} g\right)(t,x,y) &= \frac{\left(\Res^{(2)}_{\SP}b_1\right)(t,x,y)}{\hypo(\SP)}(\VV^X_{\SP}g)(t,x,y)+b_2(x,y) \frac{\partial g}{\partial y}(t,x,y)\\
&\qquad +\lb \left(\Proj^E\Xi^L_{\SP}(t,x,y)\right)(t,x,y)+\Xi^{NL}_{\SP}(t,x,y)\rb g(t,x,y) 
\end{aligned}\end{equation}
for $(t,x,y)\in \openstrip$.  For $g\in \Algebra$, we then have that
\begin{equation*} \tgenP_{\SP}g= \tgenP_{\SP}^Lg+\tgenP_{\SP}^{NL}g\end{equation*}
on $\openstrip$.
Decomposing $\tgenP_{\SP}^L$ into its actions on the $\Vspace^{d,s}$'s, we have
\begin{equation} \label{E:tgenPDecomp}\tgenP_{\SP}^L=\sum_{(d',s')\in D_{\genP}} \tgenP_{\SP}^{*,d',s'}\end{equation}
where
\begin{equation*} D_{\genP}\Def \lb (-2,-1),(0,-1),(1,-\sfrac12),(3,-\sfrac12),(-1,-\sfrac12),(-1,0),(0,0) \rb \end{equation*}
and where
\begin{equation} \label{E:GenPActionsParts} \tgenP_{\SP}^{*,d',s'}:\Vspace^{d,s}\to \Vspace^{d+d',s+s'}. \end{equation}

We want to invert $\tgenP_{\SP}^L$ on $\Proj^M\Xi^L_{\SP}$ of \eqref{E:XiMainError}.  We note that $\Proj^M \Xi^L_{\SP}$ takes values in $\Vspace^{1,-\sfrac12}\oplus \Vspace^{3,-\sfrac12}$.  Let's focus on the two operators 
\begin{equation*} \tgenP_{\SP}^{*,0,-1} =\Mono{0}{1}{0}\VV^X_{\SP}-\lb 6\Mono{1}{0}{-2}+4\Mono{0}{1}{-1}\rb\frac{\partial}{\partial y} - \VV^T_{\SP}\quad \text{and}\quad 
\tgenP_{\SP}^{*,-2,-1}=\frac12 \frac{\partial^2}{\partial y^2}.
\end{equation*}
Keeping \eqref{E:zerospaces},  \eqref{E:differentialinclusionsb}, and \eqref{E:opa} in mind, we have that
\begin{equation*}\tgenP_{\SP}^{*,0,-1}+\tgenP_{\SP}^{*,-2,-1}:\Vspace^{1,\sfrac12}\oplus \Vspace^{3,\sfrac12}\to \Vspace^{1,-\sfrac12}\oplus \Vspace^{3,-\sfrac12} \end{equation*}
and $\tgenP_{\SP}^{*,0,-1}+\tgenP_{\SP}^{*,-2,-1}$ has block representation
\begin{equation} \label{E:triangularmatrix}
\begin{blockarray}{ccc}
\Vspace^{1,\sfrac{1}{2}}&\Vspace^{3,\sfrac{1}{2}}&&\\ 
\begin{block}{[cc]c}
\tgenP_{\SP}^{*,0,-1}&\tgenP_{\SP}^{*,-2,-1}&\Vspace^{1,-\sfrac{1}{2}}\\ 
0&\tgenP_{\SP}^{*,0,-1}&\Vspace^{3,-\sfrac{1}{2}}&\\
\end{block}\end{blockarray}.\end{equation}

Let's explicitly write \eqref{E:triangularmatrix} in the basis of the $\Mono{p}{q}{r}$'s of \eqref{E:vectorspaces}.
Fixing $d\in \{1,3\}$, the action of $\tgenP_{\SP}^{*,0,-1}$ on a generic element of $\Vspace^{d,\sfrac12}$ can be written as
\begin{align*}
&\tgenP_{\SP}^{*,0,-1}\left(\sum_{p=0}^d c_p \Mono{p}{d-p}{\sfrac12-p-d/2}\right) \\
    &\qquad = \sum_{p=1}^d c_p p \Mono{p-1}{d-p+1}{\sfrac12-p-d/2} - 6\sum_{p=0}^{d-1} c_p (d-p)\Mono{p+1}{d-p-1}{\sfrac12-p-d/2-2}\\
    &\qquad \qquad - 4\sum_{p=0}^{d-1} c_p (d-p) \Mono{p}{d-p}{\sfrac12-p-d/2-1} -\sum_{p=0}^d c_p (\sfrac12-p-d/2)\Mono{p}{d-p}{\sfrac12-p-d/2-1} \\
    &\qquad = \sum_{p=0}^{d-1} c_{p+1} (p+1) \Mono{p}{d-p}{-\sfrac12-p-d/2} - 6\sum_{p=1}^d c_{p-1} (d-p+1)\Mono{p}{d-p}{-\sfrac12-p-d/2}\\
    &\qquad \qquad - 4\sum_{p=0}^{d-1} c_p (d-p) \Mono{p}{d-p}{-\sfrac12-p-d/2}-\sum_{p=0}^d c_p (\sfrac12-p-d/2)\Mono{p}{d-p}{-\sfrac12-p-d/2} \\
    &\qquad = \lb -\tfrac12(7d+1) c_0+c_1\rb \Mono{0}{d}{-\sfrac12-d/2} \\
    &\qquad \qquad +\sum_{0<0<d-1} \lb -6(d+1-p)c_{p-1}-\left( \tfrac12(7d+1)-5p\right)c_p+(1+p)c_{p+1}\rb \Mono{p}{d-p}{-\sfrac12-p-d/2}\\
     &\qquad \qquad + \lb -6c_{d-1}+ \tfrac12(3d-1)c_d\rb \Mono{d}{0}{-\sfrac12-3d/2}.
\end{align*}
Similarly, the action of  $\tgenP_{\SP}^{*,-2,-1}$ on a generic element of $\Vspace^{3,\sfrac12}$ can be written as
\begin{equation*} \tgenP_{\SP}^{*,-2,-1}\left(c_0 \Mono{0}{3}{-1}+c_1 \Mono{1}{2}{-2}+c_2 \Mono{2}{1}{-3}+c_3 \Mono{3}{0}{-4}\right) = 3c_0\Mono{0}{1}{-1}+c_1\Mono{1}{0}{-2}. \end{equation*}
Combining these, the action of $\tgenP_{\SP}^{*,0,-1}+\tgenP_{\SP}^{*,-2,-1}$ on a generic element of $\Vspace^{1,\sfrac12}\oplus \Vspace^{3,\sfrac12}$ can be written as the matrix 
\begin{equation}\label{E:bigblock}
\mathbf{P}\Def \begin{blockarray}{ccccccc}
    \Mono{0}{1}{0} & \Mono{1}{0}{-1} & \Mono{0}{3}{-1} & \Mono{1}{2}{-2} & \Mono{2}{1}{-3} & \Mono{3}{0}{-4} \\ 
    \begin{block}{[cccccc]c}
    -4 & 1 & 3 & 0 & 0 & 0 & \Mono{0}{1}{-1}\\
    -6 & 1 & 0 & 1 & 0 & 0 & \Mono{1}{0}{-2}\\
    0 &  0 &-11 & 1 & 0 & 0 & \Mono{0}{3}{-2}\\
    0 & 0 &-18 & -6 & 2 & 0 & \Mono{1}{2}{-3}\\
    0 &  0 & 0 & -12 & -1 & 3 & \Mono{2}{1}{-4}\\
    0 &  0 & 0 & 0 & -6 & 4& \Mono{3}{0}{-5}.\\
    \end{block}
    \end{blockarray}
\end{equation}

\begin{lemma}
The operator $\tgenP_{\SP}^{*,0,-1}+\tgenP_{\SP}^{*,-2,-1}$ is invertible on $\Vspace^{1,-\sfrac12}\oplus\Vspace^{3,-\sfrac12}$.
\end{lemma}
\begin{proof} 
Since the matrix $\mathbf{P}$ of \eqref{E:bigblock} is upper triangular, its determinant is
\begin{equation*} \det \begin{pmatrix} -4 & 1 \\ -6 & 1 \end{pmatrix}\times \det \begin{pmatrix} -11 & 1 & 0 & 0 \\ -18 & -6 & 2 & 0 \\ 0 & -12 & -1 & 3 \\ 0 & 0 & -6 & 4 \end{pmatrix} =2\times 120=240\not = 0 \end{equation*}
and the claim follows.
\end{proof}

Let's thus solve
\begin{equation} \label{E:zwounds}\begin{pmatrix} \chi_{0,1,0}(\alpha_1,\alpha_2) \\ \chi_{1,0,-1}(\alpha_1,\alpha_2) \\ \chi_{0,3,-1}(\alpha_1,\alpha_2) \\ \chi_{1,2,-2}(\alpha_1,\alpha_2) \\ \chi_{2,1,-3}(\alpha_1,\alpha_2) \\ \chi_{3,0,-4}(\alpha_1,\alpha_2) \end{pmatrix} = -\mathbf{P}^{-1}\begin{pmatrix} 6\alpha_1 \\ 12\alpha_2\\
-3\alpha_2 \\
-6\alpha_2 \\
0\\
0
\end{pmatrix} \end{equation}
and then define
\begin{equation} \begin{aligned}\label{E:zwoundsprime}\QCorrector_{\SP}(t,x,y) &= \sum_{p=0}^1 \chi_{p,1-p,-p}\left(\bnorm_{1,0}(\SP)b_1(\SP),\bnorm_{0,2}(\SP)\right) \Mono{p}{1-p}{-p}(t,x,y)\\
&\qquad +\sum_{p=0}^3 \chi_{p,3-p,-1-p}\left(\bnorm_{1,0}(\SP)b_1(\SP),\bnorm_{0,2}(\SP)\right) \Mono{p}{3-p}{-1-p}(t,x,y).\end{aligned}\end{equation}
Tracing through our notation, we then have that
\begin{equation*}\left(\tgenP_{\SP}^{*,0,-1}+\tgenP_{\SP}^{*,-2,-1}\right)\QCorrector_{\SP} 
=-\Xi^{*,1,-\sfrac12}_{\SP}-\Xi^{*,3,-\sfrac12}_{\SP}. \end{equation*}
Using \eqref{E:tgenPDecomp}, we can write
\begin{equation} \label{E:throwaway}\tgenP_{\SP}^L\QCorrector_{\SP} = \left(\tgenP_{\SP}^{*,0,-1}+\tgenP_{\SP}^{*,-2,-1}\right)\QCorrector_{\SP} + \sum_{\substack{(d',s')\in D_{\genP} \\ s'\ge -\sfrac12}} \tgenP_{\SP}^{*,d',s'}\QCorrector_{\SP}.
\end{equation}
Recalling \eqref{E:GenPActionsParts}, the fact that $\QCorrector_{\SP}\in \Vspace^{1,\sfrac12}\oplus \Vspace^{3,\sfrac12}$,  implies that
$\tgenP_{\SP}^{*,d',s'}\QCorrector_{\SP}\in \Vspace^E$
for all $(d',s')\in D_{\genP}$ such that $s\ge -\sfrac12$.
Thus
\begin{equation}\label{E:projectedequation}
\Proj^M\tgenP_{\SP}^L\QCorrector_{\SP} = \left(\tgenP_{\SP}^{*,0,-1}+\tgenP_{\SP}^{*,-2,-1}\right)\QCorrector_{\SP}= -\Proj^M\Xi^L_{\SP}. \end{equation}
which is the desired projection of \eqref{E:genPgoal}.
\begin{remark}
In constructing and then solving \eqref{E:projectedequation}, we have transformed \eqref{E:genPgoal} to the coordinate system adapted to the small-time asymptotics of $\kernel$ of \eqref{E:kernelDef}.
\end{remark}

Let's start to bound various errors.  Since each of the $\chi_{p,q,r}$'s of \eqref{E:zwounds} is continuous in $\alpha_1$ and $\alpha_2$.  compactness implies that
\begin{equation}\label{E:QCoefficientsBounds} \KK[E:QCoefficientsBounds]\Def \sup\lb \left|\chi_{p,d-p,\sfrac12-p-d/2}(\alpha_1,\alpha_2)\right|:d\in \{1,3\}, \, 0\le p\le d,\, |\alpha_1|\le \KK[E:hypobound]\KK[E:boundedness],\, |\alpha_2|\le \KK[E:hypobound]\rb  \end{equation}
is finite.

\begin{lemma}\label{L:Qsize}  We have that
\begin{equation}\label{E:Qsize}
\KK[E:Qsize]\Def \sup\lb \left|\QCorrector_{\SP}(t,x,y)\right|t^{-1/2}\exp\left[-\nu \xx_{\SP}^2(t,x) - \nu \yy_{\SP}^2(t,y)\right]: (\SP)\in \RR,\, (t,x,y)\in \openstrip\rb
\end{equation}
is finite.
\end{lemma}
Thus
\begin{equation}\label{E:Qatzero} \lim_{\substack{(t,x,y)\to (0,x^*,y^*) \\ (t,x,y)\in \openstrip}}\QCorrector_{\SP}(t,x,y)=0\end{equation}
for all $(x^*,y^*)\in \RR$.
\begin{proof}[Proof of Lemma \ref{L:Qsize}]
We have that
\begin{equation*} \left|\QCorrector_{\SP}(t,x,y)\right|\le \KK[E:QCoefficientsBounds] t^{1/2}\Gamma\left(t,\xx_{\SP}(t,x),\yy_{\SP}(t,y)\right) \end{equation*}
where
\begin{equation*} \Gamma(X,Y)
\Def \sum_{p=0}^1 |X|^p |Y|^{1-p}+\sum_{p=0}^3 |X|^p |Y|^{3-p}.\end{equation*}
Since
\begin{equation*} \sup\lb \left|\Gamma(X,Y)\right|\exp\left[-\nu X^2-\nu  Y^2 \right]: X,Y\in \R\rb \end{equation*}
is finite, the claim follows.
\end{proof}

\begin{lemma}\label{L:ProjectedGenerator}
We have that
\begin{equation} \label{E:ProjectedGenerator} \KK[E:ProjectedGenerator]\Def \sup_{\substack{(\SP)\in \RR\\ (t,x,y)\in \openstrip}}  \left|\left(\Proj^E\tgenP_{\SP}^L\QCorrector_{\SP}\right)(t,x,y)\right|\exp\left[-\nu\xx^2_{\SP}(t,x)-\nu\yy^2_{\SP}(t,y)\right] \end{equation}
is finite.
\end{lemma}
\begin{proof}
The proof is similar to that of Lemma \ref{L:XiLEbound}.  Expanding upon \eqref{E:throwaway}, we have that
\begin{equation*} \Proj^E\tgenP_{\SP}^L\QCorrector_{\SP} \in 
 \textrm{Span}\lb  \Vspace^{d+d',\sfrac12+s'} : (d',s')\in D_{\genP}, s'\ge -\sfrac12,d\in \{1,3\}\rb \subset \bigoplus_{\substack{0\le d\le 6 \\ s\in \{0,\sfrac12\}}}\Vspace^{d,s}.\end{equation*}
Let's also write that
\begin{align*}\frac{\left(\Poly^{(1)}_{\SP}b_1\right)(t,x,y)-b_1(\SP)}{\hypo(\SP)}&=\bnorm_{1,0}(\SP)\left(\hypo(\SP)\Mono{1}{0}{0}(t,x,y)-b_1(\SP)\Mono{0}{0}{1}(t,x,y)\right)\\
&\qquad + \bnorm_{0,1}(\SP)\Mono{0}{1}{0}(t,x,y).\end{align*}
For each $(\SP)\in \RR$, define
\begin{equation}\label{E:Pvecdeff} P(\SP)\Def \left(\bnorm_{1,0}(\SP),\bnorm_{0,1}(\SP),\bnorm_{0,2}(\SP),\hypo(\SP),b_1(\SP)\right)\in \R^5.\end{equation}
Then we can write
\begin{align*} \Proj^E\tgenP_{\SP}^L\QCorrector_{\SP} &= \sum_{\substack{0\le d\le 6 \\ s\in \{0,\sfrac12\}\\ 0\le p\le d}}c_{p,d-p,s-p-d/2}\left(P(\SP)\right)\Mono{p}{d-p}{s-p-d/2} 
\end{align*}
where the $c_{p,q,r}$'s are continuous.  Bounding the various terms in \eqref{E:Pvecdeff}, we have that 
\begin{equation*}
\lb P(\SP): \SP\in \RR\rb \subset R\Def \left[-\KK[E:hypobound],\KK[E:hypobound]\right]^3\times \left[-\KK[E:boundedness],\KK[E:boundedness]\right]^2.\end{equation*}
Since $R$ is a compact subset of $\R^5$ and the $c_{p,q,r}$'s are continuous,
\begin{equation*}\kk\Def \sup\lb \left|c_{p,q,s-p-d/2}(\rho)\right|: 0\le d\le 6,\, s\in \{0,\sfrac12\},\, 
0\le p\le d,\, \rho\in R\rb \end{equation*}
is finite.  Thus
\begin{equation*} \left|\Proj^E\tgenP_{\SP}^L\QCorrector_{\SP}(t,x,y)\right|\le \kk \Gamma\left(t,\xx_{\SP}(t,x),\yy_{\SP}(t,y)\right) \end{equation*}
where
\begin{equation*}\Gamma(t,X,Y)
\Def \sum_{\substack{0\le d\le 6 \\ s\in \{0,\sfrac12\} \\ 0\le p\le d}}|X|^p|Y|^{d-p}t^s. \end{equation*}
Since
\begin{equation*} \sup \lb \left|\Gamma(t,X,Y)\right|\exp\left[-\nu X^2 - \nu Y^2 \right]:  X,Y\in \R,\,  0<t< \TT\rb \end{equation*}
is finite, the claim follows.

\end{proof}
\begin{lemma}\label{L:TruncateGenerator}
We have that
\begin{equation} \label{E:TruncateGenerator} \KK[E:TruncateGenerator]\Def \sup\lb \left|\left(\tgenP_{\SP}^{NL}\QCorrector_{\SP}\right)(t,x,y)\right|\exp\left[-3\nu\xx^2_{\SP}(t,x)-3\nu\yy^2_{\SP}(t,y)\right]: (\SP)\in \RR,\ (t,x,y)\in \openstrip \rb \end{equation}
is finite.
\end{lemma}
\begin{proof}
Let's first compute that
\begin{align*} \VV^X_{\SP}\QCorrector_{\SP}&= \chi_{1,0,-1}\left(\bnorm_{1,0}(\SP)b_1(\SP),\bnorm_{0,2}(\SP)\right) \Mono{0}{0}{-1}\\
&\qquad +\sum_{p=1}^3 \chi_{p,3-p,-1-p}\left(\bnorm_{1,0}(\SP)b_1(\SP),\bnorm_{0,2}(\SP)\right) p\Mono{p-1}{3-p}{-1-p}\\
\frac{\partial}{\partial y}\QCorrector_{\SP}&=  \chi_{0,1,0}\left(\bnorm_{1,0}(\SP)b_1(\SP),\bnorm_{0,2}(\SP)\right) \Mono{0}{0}{0}\\
&\qquad +\sum_{p=0}^2 \chi_{p,3-p,-1-p}\left(\bnorm_{1,0}(\SP)b_1(\SP),\bnorm_{0,2}(\SP)\right) (3-p)\Mono{p}{2-p}{-1-p}.
\end{align*}
Then
\begin{align*}
\left|\VV^X_{\SP}\QCorrector_{\SP}(t,x,y)\right|&\le t^{-1}\KK[E:QCoefficientsBounds]\Gamma_1\left(\xx_{\SP}(t,x),\yy_{\SP}(t,y)\right)\\
\left|\frac{\partial}{\partial y}\QCorrector_{\SP}(t,x,y)\right|&\le \KK[E:QCoefficientsBounds]\Gamma_2\left(\xx_{\SP}(t,x),\yy_{\SP}(t,y)\right)
\end{align*}
where
\begin{equation*} 
\Gamma_1\left(X,Y\right)\le 1+\sum_{p=1}^3 p|X|^{p-1} |Y|^{3-p}\qquad \text{and}\qquad
\Gamma_2\left(X,Y\right)=1+\sum_{p=0}^2(3-p)|X|^p|Y|^{2-p}.
\end{equation*}
Combining Lemmas \ref{L:XiLEbound} and  \ref{L:XiNLbound}, we also have that
\begin{equation*} \left|\left(\Proj^E\Xi^L_{\SP}(t,x,y)\right)(t,x,y)+\Xi^{NL}_{\SP}(t,x,y)\right|\le \lb \KK[E:XiLEbound] + \KK[E:XiNLbound]\rb \exp\left[2\nu \xx_{\SP}^2(t,x) + 2\nu \yy_{\SP}^2(t,y)\right] \end{equation*}
and thus
\begin{align*} \left|\left(\tgenP_{\SP}^{NL} \QCorrector_{\SP}\right)(t,x,y)\right|
&\le  \KK[E:wheel1] t\exp\left[\nu\xx^2_{\SP}(t,x)+\nu\yy^2_{\SP}(t,y)\right]\\
&\qquad \qquad \times t^{-1}\KK[E:QCoefficientsBounds]\Gamma_1\left(\xx_{\SP}(t,x),\yy_{\SP}(t,y)\right)\\
&\qquad +\KK[E:boundedness]\KK[E:QCoefficientsBounds]\Gamma_2\left(\xx_{\SP}(t,x),\yy_{\SP}(t,y)\right)\\
&\qquad  +  \lb \KK[E:XiLEbound] + \KK[E:XiNLbound]\rb \KK[E:Qsize]t^{1/2}\exp\left[3\nu \xx_{\SP}^2(t,x) + 3\nu \yy_{\SP}^2(t,y)\right]\\
&=\Gamma\left(\xx_{\SP}(t,x),\yy_{\SP}(t,y)\right)
\end{align*}
where
\begin{align*} \Gamma(X,Y) &=  \KK[E:wheel1] \KK[E:QCoefficientsBounds]\Gamma_1\left(X,Y\right)\exp\left[\nu X^2+\nu Y^2\right]+\KK[E:boundedness]\KK[E:QCoefficientsBounds]\Gamma_2\left(X,Y\right)\\
&\qquad  +  \lb \KK[E:XiLEbound] + \KK[E:XiNLbound]\rb \KK[E:Qsize]\TT^{1/2}\exp\left[3\nu \xx_{\SP}^2(t,x) + 3\nu \yy_{\SP}^2(t,y)\right].
\end{align*}
Since
\begin{equation*} \sup\lb \left|\Gamma(X,Y)\right|\exp\left[-3\nu X^2-3\nu  Y^2 \right]:  X,Y\in \R,\, 0<t< \TT\rb \end{equation*}
is finite, the claim follows.\end{proof}
\noindent Collecting things together,
\begin{proposition}\label{P:smallgenerator} We have that
\begin{equation}\label{E:smallgenerator} \KK[E:smallgenerator]\Def \sup_{\substack{(\SP)\in \RR \\ (t,x,y)\in \openstrip}} \left|\left(\tgenP_{\SP}\QCorrector_{\SP}\right)(t,x,y)+\Xi_{\SP}(t,x,y)\right|\exp\left[-3\nu\xx^2_{\SP}(t,x)-3\nu\yy^2_{\SP}(t,y)\right] \end{equation}
is finite.
\end{proposition}
\begin{proof}

We have
\begin{align*} \tgenP_{\SP}\QCorrector_{\SP}+\Xi_{\SP}
&=\tgenP_{\SP}^L\QCorrector_{\SP}+\tgenP_{\SP}^{NL}\QCorrector_{\SP} + \Xi^L_{\SP}+\Xi^{NL}_{\SP} \\
&=\Proj^M\tgenP_{\SP}^L\QCorrector_{\SP}+\Proj^E\tgenP_{\SP}^L\QCorrector_{\SP}+\tgenP_{\SP}^{NL}\QCorrector_{\SP} \\
&\qquad + \Proj^M\Xi^L_{\SP}+\Proj^E\Xi^L_{\SP}+\Xi^{NL}_{\SP}
\end{align*}
on $\openstrip$.  Combine \eqref{E:projectedequation}, Lemmas \ref{L:XiLEbound}, \ref{L:XiNLbound}, \ref{L:ProjectedGenerator}, and \ref{L:TruncateGenerator}.  The claim follows.
\end{proof}

Similarly to \eqref{E:sidekernelDef}, let's construct a boundary potential out of $\kernelQ$ of \eqref{E:kernelQDef}; define
\begin{equation}\label{E:siderkernelQDef} \sidekernelQ_{y_\circ}(t,x,y)\Def |b_1(0,y_\circ)|\kernelQ_{0,y_\circ}(t,x,y)
= |b_1(0,y_\circ)|\kernel_{0,y_\circ}(t,x,y)\lb 1+\QCorrector_{\SP}\rb \end{equation}
for $y_\circ\in \R$ and $(t,x,y)\in \domplus$.

\begin{proposition}\label{P:CorrectorWhy}  There is a $\KK[E:CorrectorWhy]>0$ such that
\begin{equation}\label{E:CorrectorWhy} \begin{aligned}\left|\kernelQ_{\SP}(t,x,y)\right|&\le \KK[E:CorrectorWhy] \kernelbound{\sfrac{1}{12}}_{\SP}(t,x,y)\qquad (\SP)\in \RR,\, (t,x,y)\in \openstrip \\
\left|\genP\kernelQ_{\SP}(t,x,y)\right|&\le \KK[E:CorrectorWhy] \kernelbound{\sfrac{1}{12}}_{\SP}(t,x,y)\qquad (\SP)\in \RR,\, (t,x,y)\in \openstrip\\
\left|\sidekernelQ_{y_\circ}(t,x,y)\right|&\le \KK[E:CorrectorWhy] \kernelbound{\sfrac{1}{12}}_{0,y_\circ}(t,x,y)\qquad y_\circ \in \R,\, (t,x,y)\in \domplus\\
\left|\left(\genP\sidekernelQ_{y_\circ}\right)(t,x,y)\right|&\le \KK[E:CorrectorWhy] \kernelbound{\sfrac{1}{12}}_{0,y_\circ}(t,x,y).\qquad y_\circ\in \R,\, (t,x,y)\in \domint\end{aligned}\end{equation}
\end{proposition}
\begin{proof}
For $(t,x,y)\in \openstrip$,
\begin{align*} \left|\kernelQ_{\SP}(t,x,y)\right| &\le 
\left|\kernel_{\SP}(t,x,y)\right| \lb 1+ \left|\QCorrector_{\SP}(t,x,y)\right|\rb \\
&\le \KK[E:KandhatKConst]\kernelbound{\sfrac{1}{6}}_{\SP}(t,x,y)\lb 1+\KK[E:Qsize]t^{\sfrac12}\exp\left[\nu \xx^2_{\SP}(t,x)+\nu \yy^2_{\SP}(t,y)\right]\rb \\
&\le \KK[E:KandhatKConst]\kernelbound{\sfrac{1}{12}}_{\SP}(t,x,y)\lb 1+\KK[E:Qsize]\TT^{\sfrac12}\rb
\end{align*}
where we have used \eqref{E:kernelmonotonicity}, \eqref{E:nuselectwhy}, and Lemma \ref{L:Qsize}.
This gives the claimed bound on $\kernelQ_{\SP}$.  The bound on $\sidekernelQ$ follows from the formula \eqref{E:siderkernelQDef}.

Using \eqref{E:ErrorPDE}, \eqref{E:nuselectwhy}, and Proposition \ref{P:smallgenerator}, we next compute that
\begin{align*} \left|\left(\genP \kernelQ_{\SP}\right)(t,x,y)\right|
&= \left|\kernel_{\SP}(t,x,y)\right|\left| \tgenP_{\SP} \QCorrector_{\SP}(t,x,y)+\Xi_{\SP}(t,x,y)\right|\\
&\le \KK[E:KandhatKConst]\kernelbound{\sfrac{1}{6}}_{\SP}(t,x,y)\KK[E:smallgenerator]\exp\left[3\nu \xx^2_{\SP}(t,x)+3\nu \yy^2_{\SP}(t,y)\right]\\
&\le \KK[E:KandhatKConst]\KK[E:smallgenerator]\kernelbound{\sfrac{1}{12}}_{\SP}(t,x,y)
\end{align*}
for $(t,x,y)\in \openstrip$, 
 implying the claimed bound on $\genP\kernelQ_{\SP}$.  Since
\begin{equation*} \left(\genP \sidekernelQ_{y_\circ}\right)(t,x,y)=\left|b_1(0,y_\circ)\right|\left(\genP \kernelQ_{0,y_\circ}\right)(t,x,y)\end{equation*}
for $(t,x,y)\in \domint$, the claimed bound on $\genP \sidekernelQ_{y_\circ}$ also holds.
\end{proof}

\subsection{Integrability and Continuity}
Let's collect our various bounds together
\begin{theorem}\label{T:collected}  We have that
\begin{align}
\label{E:AAA} \KK[E:AAA]&\Def 
 \sup_{(t,x,y)\in \domplus}\int_{(\SP)\in \RR}\left|\kernelQ_{\SP}(t,x,y)\right|dx_\circ\, dy_\circ \\
 \label{E:CCC} \KK[E:CCC]&\Def 
 \sup_{(t,x,y)\in \domplus}\int_{s=0}^t \int_{y_\circ\in \R}\left|\sidekernelQ_{\SP}(s,x,y)\right| ds\, dy_\circ \\
 \label{E:BBB} \KK[E:BBB]&\Def  \sup_{(t,x,y)\in \domint}\int_{(\SP)\in \RR}\left|\genP \kernelQ_{\SP}(t,x,y)\right|dx_\circ\, dy_\circ \\
\label{E:DDD} \KK[E:DDD]&\Def  \sup_{(t,x,y)\in \domint}\int_{s=0}^t \int_{y_\circ\in \R}\left|\genP \sidekernelQ_{y_\circ}(s,x,y)\right|ds\, dy_\circ
\end{align}
are finite.  We also have that
\begin{equation} \label{E:EEE} \KK[E:EEE]\Def \sup_{(t,0,y)\in \domplus}\int_{y_\circ\in \R}\left|\sidekernelQ_{y_\circ}(t,0,y)\right|dy_\circ \end{equation}
is finite.
\end{theorem}
\begin{proof}
Combine Propositions \ref{P:integrablekernel}, \ref{P:integrablekernelbound}, \ref{P:sideintegrablekernel},  \ref{P:integrablesidekernelbound}, and \ref{P:CorrectorWhy} to get the first four bounds.  To get the last bound, we have
\begin{equation*} \left|\sidekernelQ_{y_\circ}(t,0,y)\right|
\le \KK[E:CorrectorWhy] \kernelbound{\sfrac{1}{12}}_{0,y_\circ}(t,0,y)
\le \KK[E:CorrectorWhy]\KK[E:sideintegrablekernel]\sidekernelboundintegrable_{y_\circ}(t,0,y)\le \KK[E:CorrectorWhy]\KK[E:sideintegrablekernel]\frac{1}{\sqrt{t}}\exp\left[-\frac{(y-y_\circ)^2}{12t}\right];\end{equation*}
thus
\begin{equation*} \int_{y_\circ\in \R}\left|\kernelQ_{y_\circ}(t,0,y)\right|dy_\circ\le \KK[E:CorrectorWhy]\KK[E:sideintegrablekernel]\int_{y_\circ\in \R}\frac{1}{\sqrt{t}}\exp\left[-\frac{(y-y_\circ)^2}{12t}\right]dy_\circ = \sqrt{12\pi}\KK[E:CorrectorWhy]\KK[E:sideintegrablekernel] \end{equation*}
implying the last claim.
\end{proof}

We also have extensions of Proposition \ref{P:dirac} and
\ref{P:jumpboundary}.  Recall here \eqref{E:Qatzero}.
\begin{proposition}\label{P:diracQ}
Fix $g\in C_b([0,\TT)\times \Rplusint\times \R)$.  For any $(x^*,y^*)\in \bdy_i$, we have that
\begin{equation*}  \lim_{\substack{(t,x,y)\to (0,x^*,y^*) \\ (t,x,y)\in \domint}}\int_{(\SP)\in \bdy_i}\kernelQ_{\SP}(t,x,y)g(t,\SP)dx_\circ\, dy_\circ = g(0,x^*,y^*). \end{equation*}
\end{proposition}
\begin{proof}
Let's write
\begin{align*} \left|\kernel_{\SP}(t,x,y)\QCorrector_{\SP}(t,x,y)\right|&\le \KK[E:KandhatKConst] \kernelbound{\sfrac{1}{6}}_{\SP}(t,x,y)\KK[E:Qsize]t^{1/2}\exp\left[\nu \xx^2_{\SP}(t,x)+\nu \yy^2_{\SP}(t,y)\right]\\
&\le \KK[E:KandhatKConst] \KK[E:Qsize]\kernelbound{\sfrac{1}{12}}_{\SP}(t,x,y)t^{\sfrac12}
\le \KK[E:KandhatKConst] \KK[E:Qsize]\KK[E:integrablekernel]\kernelboundintegrable_{\SP}(t,x,y)t^{\sfrac12}
\end{align*}
where we have used \eqref{E:kernelmonotonicity} and \eqref{E:nuselectwhy}.  Thus
\begin{align*}
&\varlimsup_{\substack{(t,x,y)\to (0,x^*,y^*) \\ (t,x,y)\in \domint}}\left|\int_{(\SP)\in \RR}\kernel_{\SP}(t,x,y)\QCorrector_{\SP}(t,x,y)g(t,\SP)dx_\circ\, dy_\circ\right|\\
&\qquad \le \varlimsup_{\substack{(t,x,y)\to (0,x^*,y^*) \\ (t,x,y)\in \domint}}\KK[E:KandhatKConst] \KK[E:integrablekernel]\KK[E:Qsize]\|g\| t^{\sfrac12}
\int_{(\SP)\in \RR}\kernelboundintegrable_{\SP}(t,x,y)dx_\circ\, dy_\circ \\
&\qquad \le \varlimsup_{\substack{(t,x,y)\to (0,x^*,y^*) \\ (t,x,y)\in \domint}}\KK[E:KandhatKConst] \KK[E:integrablekernel]\KK[E:Qsize]\|g\|\KK[E:integrablekernelbound] t^{\sfrac12}=0.
\end{align*}

Combining this with Proposition \ref{P:dirac}, we get the claim.
\end{proof}
\begin{proposition}\label{P:jumpboundaryQ}
Fix $g\in C_b([0,\TT)\times \R)$.  For any $(t^*,y^*)\in \bdy_s$, we have that
\begin{equation*}  \lim_{\substack{(t,x,y)\to (t^*,0,y^*)\\ (t,x,y)\in \dom^\circ}}\int_{s=0}^t \int_{y_\circ\in \R}\sidekernelQ_{y_\circ}(s,x,y)g(s,y_\circ)ds\, dy_\circ = g(0,y^*) +\int_{s=0}^{t^*} \int_{y_\circ\in \R}\sidekernelQ_{y_\circ}(s,0,y^*)g(s,y_\circ)ds\, dy_\circ. \end{equation*}
\end{proposition}
\begin{proof}
Proposition \ref{P:jumpboundary} directly implies that
\begin{equation} \label{E:basicint} \lim_{\substack{(t,x,y)\to (t^*,0,y^*)\\ (t,x,y)\in \dom^\circ}}\int_{s=0}^t \int_{y_\circ\in \R}\sidekernel_{y_\circ}(s,x,y)g(s,y_\circ)ds\, dy_\circ = g(0,y^*) +\int_{s=0}^{t^*} \int_{y_\circ\in \R}\sidekernel_{y_\circ}(s,0,y^*)g(s,y_\circ)ds\, dy_\circ. \end{equation}
To isolate the effect of $\QCorrector_{\SP}$, let's introduce the truncation function
\begin{equation*} c_L(z)\Def \begin{cases} L &\text{if $z>L$} \\
z&\text{if $|z|\le L$} \\
-L &\text{if $z<-L$} \end{cases}
\end{equation*}
for each $L>0$.  Then
\begin{align*}
&\left|\int_{s=0}^t \int_{y_\circ\in \R}\sidekernel_{y_\circ}(s,x,y)\QCorrector_{0,y_\circ}(s,x,y)g(s,y_\circ)ds\, dy_\circ\right.\\
&\qquad \qquad \left.-\int_{s=0}^{t^*} \int_{y_\circ\in \R}\sidekernel_{y_\circ}(s,0,y^*)\QCorrector_{0,y_\circ}(s,0,y^*)g(s,y_\circ)ds\, dy_\circ \right|\\
&\qquad \le \left|\int_{s=0}^t \int_{y_\circ\in \R}\sidekernel_{y_\circ}(s,x,y)c_L\left(\QCorrector_{0,y_\circ}(s,x,y)\right)g(s,y_\circ)ds\, dy_\circ\right.\\
&\qquad \qquad \qquad \left.-\int_{s=0}^{t^*} \int_{y_\circ\in \R}\sidekernel_{y_\circ}(s,0,y^*)c_L\left(\QCorrector_{0,y_\circ}(s,0,y^*)\right)g(s,y_\circ)ds\, dy_\circ \right|\\
&\qquad \qquad + \left|\int_{s=0}^t \int_{y_\circ\in \R}\sidekernel_{y_\circ}(s,x,y)\lb \QCorrector_{0,y_\circ}(s,x,y)-c_L\left(\QCorrector_{0,y_\circ}(s,x,y)\right)\rb g(s,y_\circ)ds\, dy_\circ\right|\\
&\qquad \qquad + \left|\int_{s=0}^{t^*} \int_{y_\circ\in \R}\sidekernel_{y_\circ}(s,0,y^*)\lb \QCorrector_{0,y_\circ}(s,0,y^*)-c_L\left(\QCorrector_{0,y_\circ}(s,0,y^*)\right)\rb g(s,y_\circ)ds\, dy_\circ\right|.
\end{align*}

In light of \eqref{E:Qatzero}, Proposition \ref{P:jumpboundary} implies that 
\begin{align*}
    &\lim_{\substack{(t,x,y)\to (t^*,0,y^*)\\ (t,x,y)\in \dom^\circ}}\int_{s=0}^t \int_{y_\circ\in \R}\sidekernel_{y_\circ}(s,x,y)c_L\left(\QCorrector_{0,y_\circ}(s,x,y)\right)g(s,y_\circ)ds\, dy_\circ\\
&\qquad =\int_{s=0}^{t^*} \int_{y_\circ\in \R}\sidekernel_{y_\circ}(s,0,y^*)c_L\left(\QCorrector_{0,y_\circ}(s,0,y^*)\right)g(s,y_\circ)ds\, dy_\circ.
\end{align*}

For every $z\in \R$,
\begin{align*} \left|z-c_L(z)\right| = \left(|z|-L\right)\bOne_{\{|z|>L\}}
\le \left(|z|+L\right)\bOne_{\{|z|>L\}}
\le 2|z|\bOne_{\{|z|>L\}} < \frac{2|z|^2}{L}. \end{align*}
Thus
\begin{align*}
    &\sidekernel_{y_\circ}(s,x,y)\left| \QCorrector_{0,y_\circ}(s,x,y)-c_L\left(\QCorrector_{0,y_\circ}(s,x,y)\right)\right|
    \le \frac{2\left|\QCorrector_{0,y_\circ}(s,x,y)\right|^2}{L}\KK[E:KandhatKConst] \kernelbound{\sfrac{1}{6}}_{0,y_\circ}(t,x,y)\\
    &\qquad \le \frac{2\KK[E:KandhatKConst]}{L}\left(\KK[E:Qsize]s^{\sfrac12}\exp\left[\nu \xx^2_{\SP}(t,x)+\nu \yy^2_{\SP}(t,y)\right]\right)^2 \kernelbound{\sfrac{1}{6}}_{0,y_\circ}(t,x,y) \\
    &\qquad \le \frac{2\KK[E:KandhatKConst]\TT}{L}\KK[E:Qsize]^2\kernelbound{\sfrac{1}{12}}_{0,y_\circ}(t,x,y)
    \le \frac{2\KK[E:KandhatKConst]\TT}{L}\KK[E:Qsize]^2\KK[E:sideintegrablekernel]\sidekernelboundintegrable_{y_\circ}(t,x,y)
\end{align*}
for $(s,x,y)\in \dom$ and consequently
\begin{align*} &\left|\int_{s=0}^t \int_{y_\circ\in \R}\sidekernel_{y_\circ}(s,x,y)\lb \QCorrector_{0,y_\circ}(s,x,y)-c_L\left(\QCorrector_{0,y_\circ}(s,x,y)\right)\rb g(s,y_\circ)ds\, dy_\circ\right| \\
&\qquad \le \frac{2\KK[E:KandhatKConst]\TT}{L}\KK[E:Qsize]^2\KK[E:sideintegrablekernel]\|g\| \int_{s=0}^t \int_{y_\circ\in \R}\sidekernelboundintegrable_{y_\circ}(s,x,y)ds\, dy_\circ
\le \frac{2\KK[E:KandhatKConst]\TT}{L}\KK[E:Qsize]^2\KK[E:sideintegrablekernel]\KK[E:integrablesidekernelbound]\|g\|\end{align*}
for $(t,x,y)\in \dom$.

Collecting things together, we get
\begin{align*} &\varlimsup_{\substack{(t,x,y)\to (t^*,0,y^*)\\ (t,x,y)\in \dom^\circ}}\left|\int_{s=0}^t \int_{y_\circ\in \R}\sidekernel_{y_\circ}(s,x,y)\QCorrector_{0,y_\circ}(s,x,y)g(s,y_\circ)ds\, dy_\circ\right.\\
&\qquad \qquad \left. -\int_{s=0}^{t^*} \int_{y_\circ\in \R}\sidekernel_{y_\circ}(s,0,y^*)\QCorrector_{0,y_\circ}(s,0,y^*)g(s,y_\circ)ds\, dy_\circ \right|
\le \frac{4\KK[E:KandhatKConst]\TT}{L}\KK[E:Qsize]^2\KK[E:sideintegrablekernel]\KK[E:integrablesidekernelbound]\|g\|. \end{align*}
Let $L\nearrow 0$ to conclude that
\begin{align*} &\lim_{\substack{(t,x,y)\to (t^*,0,y^*)\\ (t,x,y)\in \dom^\circ}}\int_{s=0}^t \int_{y_\circ\in \R}\sidekernel_{y_\circ}(s,x,y)\QCorrector_{0,y_\circ}(s,x,y)g(s,y_\circ)ds\, dy_\circ\\
&\qquad =\int_{s=0}^{t^*} \int_{y_\circ\in \R}\sidekernel_{y_\circ}(s,0,y^*)\QCorrector_{0,y_\circ}(s,0,y^*)g(s,y_\circ)ds\, dy_\circ. \end{align*}
Combine this with \eqref{E:basicint} and the claim follows..
\end{proof}

Let's also prove some continuity results.
\begin{proposition}\label{P:integralcontinuity}
Fix $g\in C_b((0,\TT)\times (0,\TT)\times \Rplusint\times \R)$. Then the maps
\begin{align*} (t,t',x,y)&\mapsto \int_{(\SP)\in \bdy_i}\kernelQ_{\SP}(t,x,y)g(t,t',\SP)ds\, dx_\circ\, dy_\circ \\
(t,t',x,y)&\mapsto \int_{(\SP)\in \bdy_i}\genP \kernelQ_{\SP}(t,x,y)g(t,t',\SP)ds\, dx_\circ\, dy_\circ \end{align*}
are continuous on $(0,\TT)\times (0,\TT)\times \R\times \R$. \end{proposition}
\begin{proof}
Fix $(\beta_1,\beta_2)\in \{(1,0),(0,1)\}$ and define
\begin{equation*} k_{\SP}(t,x,y)\Def \beta_1\kernelQ_{\SP}(t,x,y)+\beta_2\left(\genP \kernelQ_{\SP}\right)(t,x,y) \end{equation*}
for $(\SP)\in \RR$ and $(t,x,y)\in \openstrip$.  Using Remarks \ref{R:transformationone} and \ref{R:transformationtwo},
\begin{equation}\label{E:immigrantsong}\begin{aligned}
&\int_{(\SP)\in \bdy_i}k_{\SP}(t,x,y)g(t,t'\SP) dx_\circ\, dy_\circ\\
&\qquad = \int_{(\SP)\in \bdy_i}\bOne_{\{|x-x_\circ|\ge 2\KK[E:boundedness]\TT\}}k_{\SP}(t,x,y)g(t,t',\SP) dx_\circ\, dy_\circ\\
&\qquad \qquad +  \int_{(\SP)\in \bdy_i}\bOne_{\{|x-x_\circ|< 2\KK[E:boundedness]\TT\}}k_{\SP}(t,x,y)g(t,t',\SP) dx_\circ\, dy_\circ\\
&\qquad = \int_{(u_\circ,v_\circ)\in \RR}\Phi^{(1)}_g(t,t',x,y,u_\circ,v_\circ)du_\circ\, dv_\circ + \int_{(u_\circ,v_\circ)\in \RR}\Phi^{(2)}_g(t,t',x,y,u_\circ,v_\circ)du_\circ\, dv_\circ
\end{aligned}\end{equation}
where
\begin{align*} \Phi^{(1)}_g(t,t',x,y,u_\circ,v_\circ)&=t^{1/2}k_{x-u_\circ,y-t^{1/2}v_\circ}(t,x,y)\bOne_{\{|u_\circ|\ge 2\KK[E:boundedness]\TT\}}g(t,t',x-u_\circ,y-t^{1/2}v_\circ)\\
\Phi^{(2)}_g(t,t',x,y,u_\circ,v_\circ)&=t^2\hypo\left(x,y-t^{1/2}v_\circ \right)k_{x+\zeta_1(t,x,y,u_\circ,v_\circ),y-t^{1/2}v_\circ}(t,x,y)\bOne_{\{|\zeta_1(t,x,y,u_\circ,v_\circ)|< 2\KK[E:boundedness]\TT\}}\\
&\qquad \times g(t,t',x+\zeta_1(t,x,y,u_\circ,v_\circ),y-t^{1/2}v_\circ).
\end{align*}
From Proposition \ref{P:CorrectorWhy},
\begin{align*}
\left|\Phi^{(1)}_g(t,t',x,y,u_\circ,v_\circ)\right|
    &\le \KK[E:CorrectorWhy]\KK[E:integrablekernel]\|g\|t^{1/2}\kernelboundintegrable_{x-u_\circ,y-t^{1/2}v_\circ}(t,x,y)\bOne_{\{|u_\circ|\ge 2\KK[E:boundedness]\TT\}}\\
\left|\Phi^{(2)}_g(t,t',x,y,u_\circ,v_\circ)\right|
&\le \KK[E:CorrectorWhy]\KK[E:integrablekernel]\|g\|t^2\hypo\left(x,y-t^{1/2}v_\circ \right)\kernelboundintegrable_{x+\zeta_1(t,x,y,u_\circ,v_\circ),y-t^{1/2}v_\circ}(t,x,y)\\
&\qquad \times \bOne_{\{|\zeta_1(t,x,y,u_\circ,v_\circ)|< 2\KK[E:boundedness]\TT\}}. \end{align*}
The bounds of Remarks \ref{R:transformationone} and \ref{R:transformationtwo} the allow us to use dominated convergence.  For almost every $(u_\circ,v_\circ)\in \RR$,  
\begin{equation*} (t,t',x,y)\mapsto \Phi^{(1)}_g(t,t',x,y,u_\circ,v_\circ)\qquad \text{and}\qquad 
(t,t',x,y)\mapsto \Phi^{(2)}_g(t,t',x,y,u_\circ,v_\circ) \end{equation*}
are continuous.  The claim follows by passing to the limit and reversing the representation of \eqref{E:immigrantsong}.
\end{proof}

\begin{corollary}\label{C:intintcontinuity}
Fix $g\in C_b(\domint)$.  Then the maps
\begin{align*} (t,x,y)&\mapsto \int_{s=0}^t\int_{(\SP)\in \bdy_i}\kernelQ_{\SP}(t-s,x,y)g(s,\SP)ds\, dx_\circ\, dy_\circ \\
(t,x,y)&\mapsto \int_{s=0}^t \int_{(\SP)\in \bdy_i}\genP \kernelQ_{\SP}(t-s,x,y)g(s,\SP)ds\, dx_\circ\, dy_\circ \end{align*}
are continous on $\dom$.
\end{corollary}
\begin{proof}
As in the proof of Proposition \ref{P:integralcontinuity}, fix  $(\beta_1,\beta_2)\in \{(1,0),(0,1)\}$ and define
\begin{equation*} k_{\SP}(t,x,y)\Def \beta_1\kernelQ_{\SP}(t,x,y)+\beta_2\left(\genP \kernelQ_{\SP}\right)(t,x,y) \end{equation*}
for $(\SP)\in \RR$ and $(t,x,y)\in \openstrip$.
Then
\begin{align*} 
&\int_{s=0}^t\int_{(\SP)\in \bdy_i}k_{\SP}(t-s,x,y)g(s,\SP)ds\, dx_\circ\, dy_\circ \\
&\qquad = \int_{s=0}^t\int_{(\SP)\in \bdy_i}k_{\SP}(s,x,y)g(t-s,\SP)ds\, dx_\circ\, dy_\circ
=\int_{s=0}^{\TT} K_s(t,x,y) ds 
\end{align*}
where
\begin{equation*}K_s(t,x,y)
\Def \bOne_{\{0<s\le t\}}\int_{(\SP)\in \bdy_i}k_{\SP}(s,x,y)\tilde g(t,s,\SP) dx_\circ\, dy_\circ \end{equation*}
where in turn $\tilde g(t,s,\SP)\Def g\left(\max\{t-s,0\},x_\circ,y_\circ\right)$.
For every $s\in (0,\TT)$,
\begin{equation*}\left|K_s(t,x,y)\right|\le \KK[E:CorrectorWhy]\|g\|\max\{\KK[E:AAA],\KK[E:BBB]\}. \end{equation*}
so we can use dominated convergence.
By Proposition \ref{P:integralcontinuity}, $(t,x,y)\mapsto K_s(t,x,y)$
is continuous for almost every $s\in (0,\TT)$.  The claim follows.
\end{proof}

We also have
\begin{proposition}\label{P:sidecontinuity}
Fix $g\in C_b(\bdy_s)$.  Then the map
\begin{equation*} (t,y)\mapsto \int_{s=0}^t \int_{y_\circ\in \R}\sidekernelQ_{y_\circ}(t-s,0,y)g(s,y_\circ)ds\, dy_\circ \end{equation*}
is continuous on $\bdy_s$.
\end{proposition}
\begin{proof}
We use Remark \ref{R:transformationthree}.  Fix $(t^*,y^*)\in \bdy_s$.  We write
\begin{equation}\label{E:theocean}
    \int_{s=0}^t \int_{y_\circ\in \R}\sidekernelQ_{y_\circ}(t-s,0,y)g(s,y_\circ)ds\, dy_\circ
    = \int_{r=0}^{\TT} \int_{v_\circ\in \R}\Phi_g(t,r,y,v_\circ)ds\, dy_\circ
\end{equation}
where
\begin{equation*} \Phi_g(t,r,y,v_\circ)
= \sidekernelQ_{y-r^{1/2}v_\circ}(r,0,y)\bOne_{\{0<r<t\}}g(t-r,y-r^{1/2}v_\circ). \end{equation*}
Then
\begin{equation*} \left|\Phi_g(t,r,y,v_\circ)\right|
\le \KK[E:CorrectorWhy]\KK[E:sideintegrablekernel]\|g\|r^{1/2}\sidekernelboundintegrable_{y-r^{1/2}v_\circ}(r,0,y)\bOne_{\{0<r<t\}}. \end{equation*}
By \eqref{E:integrabilityfour}, we can use dominated convergence. 
For almost every $r\in (0,\TT)$ and $v_\circ \in \R$, $(t,y)\mapsto \Phi_g(t,r,y,v_\circ)$ is continuous.  The result follows by taking limits and reversing \eqref{E:theocean}.
\end{proof}

Finally
\begin{proposition}\label{P:sidebigcontinuity} Fix $g\in C_b(\bdy_s)$.  The maps
    \begin{align*}
    (t,x,y)&\mapsto \int_{s=0}^t \int_{y_\circ\in \R}\sidekernelQ_{y_\circ}(t-s,x,y)g(s,y_\circ)ds\, dy_\circ \\
(t,x,y)&\mapsto \int_{s=0}^t \int_{y_\circ\in \R}\genP \sidekernelQ_{y_\circ}(t-s,x,y)g(s,y_\circ)ds\, dy_\circ \end{align*}
are continuous on $\dom$.
\end{proposition}
\begin{proof}
Fix $(\bnorm_1,\bnorm_2)\in \{(0,1),(1,0\}$ and define
\begin{equation*} k_{y_\circ}(t,x,y)\Def \beta_1\sidekernelQ_{\SP}(t,x,y)+\beta_2\left(\genP \sidekernelQ_{\SP}\right)(t,x,y) \end{equation*}
for $y_\circ\in \R$ and $(t,x,y)\in \domplus$.
Then
\begin{equation} \label{E:abbeyroad} \begin{aligned}
    &\int_{s=0}^t \int_{y_\circ\in \R}k_{y_\circ}(t-s,x,y)g(s,y_\circ)ds\, dy_\circ\\
    &\qquad = \int_{r=0}^t \int_{y_\circ\in \R} \lb \bOne_{\{r\le  \sfrac{x}{2\KK[E:boundedness]}\}}
+\bOne_{\{r\ge  \sfrac{2x}{\ub}\}}\rb k_{y_\circ}(r,x,y)g(t-r,y_\circ)dr\, dy_\circ\\
    &\qquad \qquad + \int_{r=0}^t \int_{y_\circ\in \R}\bOne_{\{\sfrac{x}{2\KK[E:boundedness]}< r< \sfrac{2x}{\ub}\}}k_{y_\circ}(r,x,y)g(t-r,y_\circ)ds\, dy_\circ\\
&\qquad = \int_{r\in\R} \int_{v_\circ\in \R}\Phi^{(1)}_g(t,x,y,r,v_\circ)dr\,dv_\circ + \int_{r\in\R} \int_{v_\circ\in \R}\Phi^{(2)}_g(t,x,y,r,v_\circ)dr\,dv_\circ \end{aligned}\end{equation}
where
\begin{align*}
\Phi^{(1)}_g(t,x,y,r,v_\circ)&= k_{y-r^{1/2}v_\circ}(r,x,y)\lb \bOne_{\{r\le  \sfrac{x}{2\KK[E:boundedness]}\}}
+\bOne_{\{r\ge  \sfrac{2x}{\ub}\wedge t\}}\rb g(t-r,y-x^{1/2}v_\circ)\\
\Phi^{(2)}_g(t,x,y,r,v_\circ)&=\frac{\hypo\left(0,y-x^{1/2}v_\circ\right)}{\left|b_1\left(0,y-x^{1/2}v_\circ\right)\right|}x^2 k_{y-x^{1/2}v_\circ}(\zeta_2(x,y,r,v_\circ),y-x^{1/2}v_\circ)\\
   &\qquad \times \bOne_{\{\sfrac{x}{2\KK[E:boundedness]}< \zeta_2(x,y,r,v_\circ)< \sfrac{2x}{\ub}\wedge t\}}g(\zeta_2(x,y,r,v_\circ),y-x^{1/2}v_\circ).
\end{align*}
We have that
\begin{align*}
\left|\Phi^{(1)}_g(t,x,y,r,v_\circ)\right|&\le \KK[E:CorrectorWhy]\|g\| \sidekernelboundintegrable_{y-r^{1/2}v_\circ}(r,x,y)\lb \bOne_{\{r\le  \sfrac{x}{2\KK[E:boundedness]}\}}
+\bOne_{\{r\ge  \sfrac{2x}{\ub}\}}\rb \\
\left|\Phi^{(2)}_g(t,x,y,r,v_\circ)\right|&\le \KK[E:CorrectorWhy]\|g\|\frac{\hypo\left(0,y-x^{1/2}v_\circ\right)}{\left|b_1\left(0,y-x^{1/2}v_\circ\right)\right|}x^2 k_{y-r^{1/2}v_\circ}(\zeta_2(x,y,r,v_\circ),x,y)\\
&\qquad \times \bOne_{\{\sfrac{x}{2\KK[E:boundedness]}< \zeta_2(x,y,r,v_\circ)< \sfrac{2x}{\ub}\wedge \TT\}}.
\end{align*}
By Remarks \ref{R:transformationthree} and \ref{R:transformationfour}, we can then used dominated convergence.  For almost every $r\in \R$ and $v_\circ\in \R$, the maps
\begin{equation*} (t,x,y)\mapsto \Phi^{(1)}_g(t,x,y,r,v_\circ) \qquad \text{and}\qquad 
(t,x,y)\mapsto \Phi^{(2)}_g(t,x,y,r,v_\circ) \end{equation*}
are continuous in $\dom$.  The claim follows by taking limits and reversing \eqref{E:abbeyroad}.
\end{proof}

\section{Volterra Equation}\label{S:Volterra}
Let's finally solve \eqref{E:MainPDE}.  We will use the functional space $\Cspace$ as the space of inputs to a variation of parameters formula.  If $u$ is the solution of \eqref{E:MainPDE}, we want to find $(\psi,\psi^\bdy)\in \Cspace$ such that
\begin{equation} \label{E:VariationOfConstants} \begin{aligned} u(t,x,y)&=\int_{s=0}^t \int_{(\SP)\in \bdy_i}\kernelQ_{\SP}(t-s,x,y)\psi(s,x_\circ,y_\circ)dx_\circ\, dy_\circ\, ds \\
&\qquad + \int_{s=0}^t \int_{y_\circ\in \R}\sidekernelQ_{y_\circ}(t-s,x,y)\psi^\bdy (s,y_\circ)ds\, dy_\circ 
+ \int_{(x,y)\in \bdy_i}\kernelQ_{\SP}(t,x,y)u_\circ(x_\circ,y_\circ)dx_\circ\, dy_\circ \end{aligned} \end{equation}
for all $(t,x,y)\in \dom^\circ$, where $\kernelQ$ is as in \eqref{E:kernelQDef} and $\sidekernelQ$ is as in  
 \eqref{E:siderkernelQDef}.  See \cite{hwang2019vlasov} for an alternate functional-space treatment of a related problem.  Analogous boundary-domain integral equations in appropriate functional spaces have also been used to study stationary problems \cite{fresneda2020existence,fresneda2020new,fresneda2022boundary}.

We want to set up a Volterra calculation to reverse-engineer $\psi$ and $\psi^\bdy$. 
If \eqref{E:VariationOfConstants} holds, Propositions \ref{P:diracQ} and \ref{P:jumpboundaryQ} should imply that
\begin{align*} -f(t,x,y)&= \left(\genP u\right)(t,x,y)\\
&= -\psi(t,x,y)+\int_{s=0}^t \int_{(\SP)\in \bdy_i}\genP\kernelQ_{\SP}(t-s,x,y)\psi(s,x_\circ,y_\circ)dx_\circ\, dy_\circ\, ds \\
&\qquad + \int_{s=0}^t \int_{y_\circ\in \R}\genP\sidekernelQ_{y_\circ}(t-s,x,y)\psi^\bdy(s,y_\circ)ds\, dy_\circ \\
&\qquad + \int_{(\SP)\in \bdy_i} \genP\kernelQ_{\SP}(t,x,y)u_\circ(x_\circ,y_\circ)dx_\circ\, dy_\circ \qquad (t,x,y)\in \dom^\circ\\
u_\bdy(t,y)&= \psi^\bdy(t,y)+\int_{s=0}^t \int_{(\SP)\in \bdy_i}\kernelQ_{\SP}(t-s,0,y)\psi(s,x_\circ,y_\circ)dx_\circ\, dy_\circ\, ds \\
&\qquad + \int_{s=0}^t \int_{y_\circ\in \R}\sidekernelQ_{y_\circ}(t-s,0,y)\psi^\bdy(s,y_\circ)ds\, dy_\circ\\
&\qquad + \int_{(\SP)\in \bdy_i}\kernelQ_{\SP}(t,0,y)u_\circ(x_\circ,y_\circ)dx_\circ\, dy_\circ.
\qquad (t,y)\in \bdy_s \end{align*}
Rearranging, $(\psi,\psi^\bdy)$ should satisfy the Volterra integral equation
\begin{equation} \label{E:Volterra} \begin{aligned} \psi(t,x,y)&=f(t,x,y)+\int_{s=0}^t \int_{(\SP)\in \bdy_i} \genP\kernelQ_{\SP}(t-s,x,y)\psi(s,x_\circ,y_\circ)dx_\circ\, dy_\circ\, ds \\
&\qquad + \int_{s=0}^t \int_{y_\circ\in \R}\genP\sidekernelQ_{y_\circ}(t-s,x,y)\psi^\bdy(s,y_\circ)ds\, dy_\circ \\
&\qquad + \int_{(\SP)\in \bdy_i} \genP\kernelQ_{\SP}(t,x,y)u_\circ(x_\circ,y_\circ)dx_\circ\, dy_\circ \qquad (t,x,y)\in \dom^\circ\\
\psi^\bdy(t,y)&=u_\bdy(t,y)-\int_{s=0}^t \int_{(\SP)\in \bdy_i}\kernelQ_{\SP}(t-s,0,y)\psi(s,x_\circ,y_\circ)dx_\circ\, dy_\circ\, ds \\
&\qquad - \int_{s=0}^t \int_{y_\circ\in \R}\sidekernelQ_{y_\circ}(t-s,0,y)\psi^\bdy(s,y_\circ)ds\, dy_\circ\\
&\qquad - \int_{(x,y)\in \bdy_i}\kernelQ_{\SP}(t,0,y)u_\circ(x_\circ,y_\circ)dx_\circ\, dy_\circ.
\qquad (t,y)\in \bdy_s\end{aligned}\end{equation}
This suggests an iteration in $\Cspace$ defined by starting at $\psi_0 \equiv 0$ on $\dom^\circ$ and $\psi_0^\bdy\equiv 0$ on $\bdy_s$
and then recursively defining
\begin{align*} \psi_{n+1}(t,x,y)&=f(t,x,y)+\int_{s=0}^t \int_{(\SP)\in \bdy_i}\genP\kernelQ_{\SP}(t-s,x,y)\psi_n(s,x_\circ,y_\circ)dx_\circ\, dy_\circ\, ds \\
&\qquad + \int_{s=0}^t \int_{y_\circ\in \R}\genP\sidekernelQ_{y_\circ}(t-s,x,y)\psi_n^\bdy (s,y_\circ)ds\, dy_\circ \\
&\qquad + \int_{(\SP)\in \bdy_i} \genP\kernelQ_{\SP}(t,x,y)u_\circ(x_\circ,y_\circ)dx_\circ\, dy_\circ \qquad (t,x,y)\in \dom^\circ \\
\psi_{n+1}^\bdy(t,y)&=u_\bdy(t,y)-\int_{s=0}^t \int_{(\SP)\in \bdy_i}\kernelQ_{\SP}(t-s,0,y)\psi_n(s,x_\circ,y_\circ)dx_\circ\, dy_\circ\, ds \\
&\qquad - \int_{s=0}^t \int_{y_\circ\in \R}\sidekernelQ_{y_\circ}(t-s,0,y)\psi_n^\bdy(s,y_\circ)ds\, dy_\circ\\
&\qquad - \int_{(x,y)\in \bdy_i}\kernelQ_{\SP}(t,0,y)u_\circ(x_\circ,y_\circ)dx_\circ\, dy_\circ \qquad (t,y)\in \bdy_s
\end{align*}
for $n\in \{0,1\dots\}$.
We note in particular that
\begin{align*} \psi_1(t,x,y)&=f(t,x,y) + \int_{(\SP)\in \bdy_i}\genP\kernelQ_{\SP}(t,x,y)u_\circ(x_\circ,y_\circ)dx_\circ\, dy_\circ \qquad (t,x,y)\in \dom^\circ\\
\psi_1^\bdy(t,y)&=u_\bdy(t,y) - \int_{(x,y)\in \bdy_i}\kernelQ_{\SP}(t,0,y)u_\circ(x_\circ,y_\circ)dx_\circ\, dy_\circ. \qquad (t,y)\in \bdy_s
\end{align*}
To show convergence of $\{(\psi_n,\psi_n^\bdy)\}_n$,
let's consider differences
\begin{align*} \tilde \psi_n(t,x,y)&\Def \psi_{n+1}(t,x,y)-\psi_n(t,x,y) \qquad (t,x,y)\in \dom^\circ\\
\tilde \psi_n^\bdy(t,y)&\Def \psi_{n+1}^\bdy(t,y)-\psi_n^\bdy(t,y)\qquad (t,y)\in \bdy_s \end{align*}
for $n\in \{0,1,2,\dots\}$.
Explicitly, let's \emph{define}
\begin{equation} \label{E:explicitpsi}\begin{aligned} \tilde \psi_0(t,x,y)&= f(t,x,y) + \int_{(\SP)\in \bdy_i}\genP\kernelQ_{\SP}(t,x,y)u_\circ(x_\circ 0,y_\circ)dx_\circ\, dy_\circ \qquad (t,x,y)\in \domint \\
\tilde \psi_0^\bdy(t,y)&= u_\bdy(t,y)-\int_{(x,y)\in \bdy_i}\kernelQ_{\SP}(t,0,y)u_\circ(x_\circ,y_\circ)dx_\circ\, dy_\circ \qquad (t,y)\in \bdy_s\\
\tilde \psi_{n+1}(t,x,y)&=\int_{s=0}^t \int_{(\SP)\in \bdy_i}\genP\kernelQ_{\SP}(t-s,x,y)\tilde \psi_n(s,x_\circ,y_\circ)dx_\circ\, dy_\circ\, ds \\
&\qquad + \int_{s=0}^t \int_{y_\circ\in \R}\genP\sidekernelQ_{y_\circ}(t-s,x,y)\tilde \psi_n^\bdy(s,y_\circ)dy_\circ\, ds \qquad (t,x,y)\in \domint \\
\tilde \psi_{n+1}^\bdy(t,y)&=-\int_{s=0}^t \int_{(\SP)\in \bdy_i}\kernelQ_{\SP}(t-s,0,y)\tilde \psi_n(s,x_\circ,y_\circ)dx_\circ\, dy_\circ\, ds \\
&\qquad - \int_{s=0}^t \int_{y_\circ\in \R}\sidekernelQ_{y_\circ}(t-s,0,y)\tilde \psi_n^\bdy(s,y_\circ)dy_\circ\, ds \qquad (t,y)\in \bdy_s
\end{aligned}\end{equation}
for $n\in \N$.

We would like to prove estimates on $\{(\tilde \psi_n,\tilde \psi_n^\bdy)\}_{n\in \N}$ leading to convergence.  Let's define constants
\begin{align}\label{E:SummK1}\KK[E:SummK1] &\Def \max\lb \|f\|
+ \KK[E:BBB]\|u_\circ\|,\|u_\bdy\|+ \KK[E:AAA]\|u_\circ\|\rb \\
\label{E:SummK2} \KK[E:SummK2] &\Def 
\max\lb \KK[E:BBB]\TT +\KK[E:DDD],\KK[E:AAA]+\KK[E:EEE],\KK[E:BBB] +\KK[E:DDD]\rb.\end{align}
\begin{lemma}\label{L:VolterraBounds} For $n\in \{0,1\dots\}$
\begin{equation} \label{E:iterativeclaim}\begin{gathered}
\left|\tilde \psi_{2n}(t,x,y)\right|\le
\frac{\KK[E:SummK1]\KK[E:SummK2]^{2n}t^n}{n!} \qquad \text{and}\qquad 
\left|\tilde \psi_{2n+1}(t,x,y)\right|\le
\frac{\KK[E:SummK1]\KK[E:SummK2]^{2n+1}t^n}{n!} \qquad \text{for all $(t,x,y)\in \domint$} \\
\left|\tilde \psi_{2n}^\bdy(t,y)\right|\le
\frac{\KK[E:SummK1]\KK[E:SummK2]^{2n}t^n}{n!} \qquad \text{and}\qquad 
\left|\tilde \psi_{2n+1}^\bdy(t,y)\right|\le
\frac{\KK[E:SummK1]\KK[E:SummK2]^{2n+1}t^{n+1}}{(n+1)!} \qquad \text{for all $(t,y)\in \bdy_s$}
\end{gathered}\end{equation}
\end{lemma}

\begin{remark}
The bounds on $\kernelQ$ and $\genP\kernelQ$ of Theorem \ref{T:collected} allow the parts of \eqref{E:explicitpsi} containing these terms to be treated in standard ways in recursively bounding $\tilde \psi_n$ and $\tilde \psi^\bdy_n$.  The bound on $\genP \sidekernelQ$ of Theorem \ref{T:collected} reflects the fact that $\genP \sidekernelQ$ was constructed in Section \ref{S:Corrections} to have similar asymptotics to $\sidekernelQ$; i.e., $\genP \sidekernelQ$ is at worst a Dirac function near the side boundary.  This means that the integral against $\genP\sidekernelQ$ in the recursion for $\tilde \psi_n$ of \eqref{E:explicitpsi} can only be bounded by the maximum of its integrand.  On the other hand, the integral against $\sidekernelQ$ in the recursion for $\tilde \psi^\bdy_n$ of \eqref{E:explicitpsi} is an integral along the side boundary \textup{(}where $x=0$\textup{)}, and $\sidekernelQ$ at the side boundary is bounded \textup{(}as in \eqref{E:EEE}\textup{)}.  
We can get familiar estimates \textup{(}involving powers of time divided by factorials\textup{)} as in \eqref{E:iterativeclaim} by working through this integral against $\sidekernelQ$ along the side boundary.  The periodic nature of \eqref{E:iterativeclaim} comes from this fact; standard integral estimates in the pair of terms $(\tilde \psi_n,\tilde \psi^\bdy_n)$ are available in only one of the pair of equations.
\end{remark}
\begin{proof}[Proof of Lemma \ref{L:VolterraBounds}]
For $(t,x,y)\in \domint$,
\begin{equation*}|\tilde \psi_0(t,x,y)|
\le \|f\|+\lb \int_{(\SP)\in \bdy_i}\left|\genP \kernelQ_{\SP}(t,x,y)\right|dx_\circ\, dy_\circ\rb  \|u_\circ\| \le \lb \|f\|
+ \KK[E:BBB]\|u_\circ\|\rb \end{equation*}
and for $(t,y)\in \bdy_s$,
\begin{equation*}
|\tilde \psi_0^\bdy(t,y)|
\le
\|u_\bdy\|+ \lb \int_{(\SP)\in \bdy_i}\left|\kernelQ_{\SP}(t,x,y)\right|dx_\circ\, dy_\circ\rb \|u_\circ\| \le \|u_\bdy\|+ \KK[E:AAA]\|u_\circ\|.
\end{equation*}
This implies \eqref{E:iterativeclaim} for $2n=0$.

Assume next that the bounds of \eqref{E:iterativeclaim} hold for some $2n$.  For $(t,x,y)\in \domint$,
\begin{align*}
\left|\tilde \psi_{2n+1}(t,x,y)\right|
&\le \int_{s=0}^t \lb \int_{(\SP)\in \bdy_i}\left|\genP \kernelQ_{\SP}(t-s,x,y)\right|dx_\circ\, dy_\circ\rb \frac{\KK[E:SummK1]\KK[E:SummK2]^{2n} s^n}{n!}  ds \\
&\qquad +\int_{s=0}^t\int_{y_\circ\in \R}\left|\genP \sidekernelQ_{y_\circ}(t-s,x,y)\right| \frac{\KK[E:SummK1]\KK[E:SummK2]^{2n}  s^n}{n!} dy_\circ\, ds\\
&\le \frac{\KK[E:SummK1]\KK[E:SummK2]^{2n}}{n!} \lb \KK[E:BBB]\int_{s=0}^t \lb \sup_{0\le s'\le t}(s')^n\rb ds+\KK[E:DDD]\sup_{0\le s\le t}s^n\rb  \\
&= \frac{\KK[E:SummK1]\KK[E:SummK2]^{2n}}{n!}\lb \KK[E:BBB]\TT t^n+\KK[E:DDD]t^n\rb  
= \frac{\KK[E:SummK1]\KK[E:SummK2]^{2n}}{n!}\lb \KK[E:BBB]\TT +\KK[E:DDD]\rb t^n.
\end{align*}
For $(t,y)\in \bdy_s$,
\begin{align*}
\left|\tilde \psi_{2n+1}^\bdy(t,y)\right|
&\le \int_{s=0}^t \lb \int_{(\SP)\in \bdy_i}\left|\kernelQ_{\SP}(t-s,0,y_\circ)\right|dx_\circ\, dy_\circ\rb \frac{\KK[E:SummK1]\KK[E:SummK2]^{2n} s^n}{n!} ds \\
&\qquad + \int_{s=0}^t   \lb \int_{y_\circ\in \R}\left|\sidekernelQ_{y_\circ}(t-s,0,y) \right|dy_\circ\rb \frac{\KK[E:SummK1]\KK[E:SummK2]^{2n} s^n}{n!} \,ds\\
&\le \frac{\KK[E:SummK1]\KK[E:SummK2]^{2n}}{n!}\lb \KK[E:AAA]\int_{s=0}^t s^n ds + \KK[E:EEE]\int_{s=0}^t s^n ds\rb 
= \frac{\KK[E:SummK1]\KK[E:SummK2]^{2n}}{n!}\lb \KK[E:AAA]  + \KK[E:EEE]\rb \frac{t^{n+1}}{n+1}. 
\end{align*}
This gives us \eqref{E:iterativeclaim} for $2n+1$.

Assume next that the bounds of \eqref{E:iterativeclaim} hold for some $2n+1$.
For $(t,x,y)\in \domint$,
\begin{align*}
\left|\tilde \psi_{2n+2}(t,x,y)\right|
&\le \int_{s=0}^t \lb \int_{(\SP)\in \bdy_i}\left|\genP \kernelQ_{\SP}(t-s,x,y)\right|dx_\circ\, dy_\circ\rb \frac{\KK[E:SummK1]\KK[E:SummK2]^{2n+1} s^n}{n!}  ds \\
&\qquad +\int_{s=0}^t\int_{y_\circ\in \R}\left|\genP \sidekernelQ_{y_\circ}(t-s,x,y)\right| \frac{\KK[E:SummK1]\KK[E:SummK2]^{2n+1}  s^{n+1}}{(n+1)!} dy_\circ\, ds\\
&\le \KK[E:SummK1]\KK[E:SummK2]^{2n+1} \lb \frac{\KK[E:BBB]}{n!}\int_{s=0}^t s^n  ds+\frac{\KK[E:DDD]}{(n+1)!}\sup_{0\le s\le t}s^{n+1}\rb  \\
&= \KK[E:SummK1]\KK[E:SummK2]^{2n+1}\lb \frac{\KK[E:BBB]t^{n+1}}{(n+1)!}+\frac{\KK[E:DDD]t^{n+1}}{(n+1)!}\rb 
= \frac{\KK[E:SummK1]\KK[E:SummK2]^{2n+1}}{(n+1)!}\lb \KK[E:BBB] +\KK[E:DDD]\rb t^{n+1}.
\end{align*}
For $(t,y)\in \bdy_s$,
\begin{align*}
\left|\tilde \psi_{2n+2}^\bdy(t,y)\right|
&\le \int_{s=0}^t \lb \int_{(\SP)\in \bdy_i}\left|\kernelQ_{\SP}(t-s,0,y_\circ)\right|dx_\circ\, dy_\circ\rb \frac{\KK[E:SummK1]\KK[E:SummK2]^{2n+1} s^n}{n!} ds \\
&\qquad + \int_{s=0}^t  \lb \int_{y_\circ\in \R}\left|\sidekernelQ_{y_\circ}(t-s,0,y)\right|dy_\circ\rb \frac{\KK[E:SummK1]\KK[E:SummK2]^{2n+1} s^{n+1}}{(n+1)!} ds\\
&\le \KK[E:SummK1]\KK[E:SummK2]^{2n+1}\lb \frac{\KK[E:AAA]}{n!}\int_{s=0}^t s^n ds + \frac{\KK[E:EEE]}{(n+1)!}\int_{s=0}^t \lb \sup_{0\le s'\le t}(s')^n\rb  ds\rb \\
&= \KK[E:SummK1]\KK[E:SummK2]^{2n+1}\lb \frac{\KK[E:AAA]t^{n+1}}{(n+1)!}  + \frac{\KK[E:CCC]t^{n+1}}{(n+1)!}\rb 
=\frac{\KK[E:SummK1]\KK[E:SummK2]^{2n+1}}{(n+1)!}\lb \KK[E:AAA]+\KK[E:EEE] \rb t^{n+1}.
\end{align*}
This gives us \eqref{E:iterativeclaim} for $2n+2$.
\end{proof}

Noting that 
\begin{equation*} \sum_{n=0}^\infty \frac{\KK[E:SummK2]^{2n}\TT^n}{n!}<\infty, \end{equation*}
we have, keeping in mind the continuity results of Propositions \ref{P:integralcontinuity}, \ref{P:sidecontinuity}, \ref{P:sidebigcontinuity}, and Corollary \ref{C:intintcontinuity}, that
\begin{align*} \psi(t,x,y)&\Def \sum_{n=0}^\infty \tilde \psi_n(t,x,y) \qquad (t,x,y)\in \domint \\
\psi^\bdy(t,y)&\Def \sum_{n=0}^\infty \tilde \psi_n^\bdy(t,y) \qquad (t,y)\in \bdy_s\end{align*}
exist in the topology of uniform convergence on compacts.  The pair $(\psi,\psi)$ then satisfies \eqref{E:Volterra}.  We finally define $u$ by \eqref{E:VariationOfConstants}.

We want to show that $u$ satisfies the desired PDE.
\begin{theorem}\label{T:Main} The function $u$ is the solution to
\begin{equation} \label{E:MainViscosity} \begin{aligned}\left(\genP u\right)(t,x,y)+f(t,x,y)&=0 \qquad (t,x,y)\in \dom^\circ \\
u(t,0,y)&=u_\bdy(t,y)\qquad (t,y)\in \bdy_s \\
u(0,x,y)&=u_\circ(x,y).\qquad  (x,y)\in \bdy_i \end{aligned}\end{equation}
\end{theorem}
\begin{proof} 
Let's remove the singularity in $\kernel$; for $\eps>0$, define
\begin{align*} u_\eps(t,x,y)&=\int_{s\in (0,t]} \int_{(\SP)\in \bdy_i}\kernelQ_{\SP}(t-s+\eps,x,y)\psi(s,x_\circ,y_\circ)dx_\circ\, dy_\circ\\
&\qquad + \int_{s=0}^t \int_{y_\circ\in \R}\sidekernelQ_{y_\circ}(t-s,x,y)\psi^\bdy(s,y_\circ)ds\, dy_\circ\, ds + \int_{(x,y)\in \bdy_i}\kernelQ_{\SP}(t,x,y)u_\circ(x_\circ,y_\circ)dx_\circ\, dy_\circ \end{align*}
which is in $C^\infty$ on $\dom^\circ$.
Using Proposition \ref{P:diracQ}, we then have that
\begin{equation}\label{E:visdom} \left(\genP u_\eps\right)(t,x,y)= \Phi_\eps(t,x,y)
\end{equation}
for $(t,x,y)\in \dom^\circ$, where
\begin{align*}
&\Phi_\eps(t,x,y)\\
&\qquad = -\int_{(\SP)\in \bdy_i}\kernelQ_{\SP}(\eps,x,y)\psi(t,x_\circ,y_\circ)dx_\circ\, dy_\circ \\
&\qquad\qquad + \int_{s=0}^t \int_{(\SP)\in \bdy_i}\genP \kernelQ_{\SP}(t-s+\eps,x,y)\psi(s,x_\circ,y_\circ)dx_\circ\, dy_\circ\, ds \\
&\qquad\qquad + \int_{s=0}^t \int_{y_\circ\in \R}\genP \sidekernelQ_{y_\circ}(t-s,x,y)\psi^\bdy(s,y_\circ)ds\, dy_\circ\, ds 
+ \int_{(x,y)\in \bdy_i}\genP \kernelQ_{\SP}(t,x,y)u_\circ(x_\circ,y_\circ)dx_\circ\, dy_\circ \\
&\qquad= -\int_{(\SP)\in \bdy_i}\kernelQ_{\SP}(\eps,x,y)f(t,x_\circ,y_\circ)dx_\circ\, dy_\circ \\
&\qquad\qquad - \int_{(\SP)\in \RR}\kernelQ_{\SP}(\eps,x,y)\int_{s=0}^t \int_{(x'_\circ,y'_\circ)\in \bdy_i}\genP\kernelQ_{x'_\circ,y'_\circ}(t-s,x_\circ,y_\circ)\psi(s,x'_\circ,y'_\circ)dx'_\circ\, dy'_\circ\, ds\, dx_\circ\, dy_\circ\\
&\qquad\qquad - \int_{(\SP)\in \bdy_i}\kernelQ_{\SP}(\eps,x,y)\int_{s=0}^t \int_{y'_\circ\in \R} \genP\sidekernelQ_{y'_\circ}(t-s,x_\circ,y_\circ)\psi^\bdy(s,y_\circ)ds\, dy'_\circ\, dx_\circ\, dy_\circ\\
&\qquad\qquad - \int_{(\SP)\in \bdy_i}\kernelQ_{\SP}(\eps,x,y)\int_{(x'_\circ,y'_\circ)\in \bdy_i} \genP\kernelQ_{x'_\circ,y'_\circ}(t,x_\circ,y_\circ)u_\circ(x'_\circ,y'_\circ)dx'_\circ\, dy'_\circ\, dx_\circ\, dy_\circ\\
&\qquad\qquad + \int_{s=0}^t \int_{(\SP)\in \bdy_i}\genP \kernelQ_{\SP}(t-s+\eps,x,y)\psi(s,x_\circ,y_\circ)dx_\circ\, dy_\circ\, ds \\
&\qquad\qquad + \int_{s=0}^t \int_{(\SP)\in \bdy_i}\genP \sidekernelQ_{\SP}(t-s,x,y)\psi^\bdy(s,x_\circ,y_\circ)dx_\circ\, dy_\circ\, ds \\
&\qquad\qquad + \int_{(x,y)\in \bdy_i}\genP \kernelQ_{\SP}(t,x,y)u_\circ(x_\circ,y_\circ)dx_\circ\, dy_\circ. 
\end{align*}
By Proposition \ref{P:diracQ},
\begin{align*}&\lim_{\eps\to 0}\int_{(\SP)\in \bdy_i}\kernelQ_{\SP}(\eps,x,y)f(t,x_\circ,y_\circ)dx_\circ\, dy_\circ=f(t,x,y) \\
&\lim_{\eps\to 0}\int_{(\SP)\in \bdy_i}\kernelQ_{\SP}(\eps,x,y)_{s=0}^t \int_{(x'_\circ,y'_\circ)\in \bdy_i} \genP\kernelQ_{x'_\circ,y'_\circ}(t-s,x_\circ,y_\circ)\psi(s,x'_\circ,y'_\circ)dx'_\circ\, dy'_\circ\, ds\, dx_\circ\, dy_\circ\\
&\qquad = \int_{s=0}^t \int_{(x'_\circ,y'_\circ)\in \bdy_i}\genP\kernelQ_{x'_\circ,y'_\circ}(t-s,x,y)\psi(s,x'_\circ,y'_\circ)dx'_\circ\, dy'_\circ\, ds\\
&\lim_{\eps\to 0}\int_{(\SP)\in \bdy_i}\kernelQ_{\SP}(\eps,x,y)\int_{s=0}^t \int_{y'_\circ\in \R}\genP\sidekernelQ_{y'_\circ}(t-s,x_\circ,y_\circ)\psi^\bdy(s,y_\circ)ds\, dy'_\circ\, dx_\circ\, dy_\circ\\
&\qquad = \int_{s=0}^t \int_{y'_\circ\in \R}\genP_{0,y'_\circ}\sidekernelQ_{y'_\circ}(t-s,x,y)\psi^\bdy(s,y_\circ)ds\, dy'_\circ\\
&\lim_{\eps\to 0}\int_{(\SP)\in \bdy_i}\kernelQ_{\SP}(\eps,x,y)\int_{(x'_\circ,y'_\circ)\in \bdy_i}\genP\kernelQ_{x'_\circ,y'_\circ}(t,x_\circ,y_\circ)u_\circ(x'_\circ,y'_\circ)dx'_\circ\, dy'_\circ\, dx_\circ\, dy_\circ\\
&\qquad = \int_{(x'_\circ,y'_\circ)\in \bdy_i}\genP\kernelQ_{x'_\circ,y'_\circ}(t,x,y)u_\circ(x'_\circ,y'_\circ)dx'_\circ\, dy'_\circ
\end{align*}
for $(t,x,y)\in \dom^\circ$.  Corollary \ref{C:intintcontinuity} implies that
\begin{align*} &\lim_{\eps\searrow 0}\int_{s=0}^t \int_{(\SP)\in \bdy_i}\genP \kernelQ_{\SP}(t-s+\eps,x,y)\psi(s,x_\circ,y_\circ)dx_\circ\, dy_\circ\, ds\\
&\qquad =\lim_{\eps\searrow 0}\int_{s=0}^t \int_{(\SP)\in \bdy_i}\genP \kernelQ_{\SP}(t-s,x,y)\psi(s,x_\circ,y_\circ)dx_\circ\, dy_\circ\, ds.\end{align*}
Consequently $\lim_{\eps \to 0}\Phi_\eps(t,x,y)=-f(t,x,y)$
for all $(t,x,y)\in \domint$.

For each $t>0$ and $y \in \R$ we also have that $\lim_{x\searrow 0} u^\eps(t,x,y)=\Phi^\bdy(t,y)$,
where
\begin{align*} \Phi^\bdy(t,y)
&= \int_{s=0}^t \int_{(\SP)\in \bdy_i}\kernelQ_{\SP}(t-s+\eps,0,y)\psi(s,x_\circ,y_\circ)dx_\circ\, dy_\circ \, ds\\
&\qquad + \psi^\bdy(t,y)+\int_{s=0}^t \int_{y_\circ\in \R}\sidekernelQ_{y_\circ}(t-s,0,y)\psi^\bdy(s,y_\circ) dy_\circ\, ds \\
&\qquad + \int_{(\SP)\in \bdy_i}\kernelQ_{\SP}(t,0,y)u_\circ(x_\circ,y_\circ)dx_\circ\, dy_\circ \\
&= u_\bdy(t,y)  + \int_{s=0}^t \int_{(\SP)\in \bdy_i}\kernelQ_{\SP}(t-s+\eps,0,y)\psi(s,x_\circ,y_\circ)dx_\circ\, dy_\circ \, ds\\
&\qquad -\int_{s=0}^t \int_{(\SP)\in \bdy_i}\kernelQ_{\SP}(t-s,0,y)\psi(s,x_\circ,y_\circ)dx_\circ\, dy_\circ \, ds\end{align*}
for $(t,y)\in \bdy_s$.  From Proposition \ref{P:diracQ}, we have that 
$\lim_{\eps \to 0}\Phi^\bdy(t,y)=u_\bdy(t,y)$.

We also have that 
\begin{equation*} u_\eps(0,x,y)=\int_{(\SP)\in \RR}\kernelQ_{\SP}(\eps,x,y)u_\circ(\SP)dx_\circ\, dy_\circ \end{equation*}
and Proposition \ref{P:diracQ} implies that $\lim_{\eps \to 0}u_\eps(0,x,y)=u_\circ(x,y)$ for all $(x,y)\in \Rplus\times \R$.

Since $u_\eps$ satisfies \eqref{E:visdom} in the classical sense, it also satisfies \eqref{E:visdom} in the viscosity sense.
Taking limits, we see that $u$ satisfies \eqref{E:MainViscosity} in a viscosity sense.  Since the classical solution is also a viscosity solution, uniqueness of viscosity solutions \cite{jensen1988uniqueness} completes the proof.
\end{proof}

\iftoggle{ARXIV}{
    \bibliographystyle{amsalpha}
    \newcommand{\etalchar}[1]{$^{#1}$}
\providecommand{\bysame}{\leavevmode\hbox to3em{\hrulefill}\thinspace}
\providecommand{\MR}{\relax\ifhmode\unskip\space\fi MR }
\providecommand{\MRhref}[2]{%
  \href{http://www.ams.org/mathscinet-getitem?mr=#1}{#2}
}
\providecommand{\href}[2]{#2}

}{
    \printbibliography 
}

\end{document}